\documentclass[10pt,reqno]{amsart}
\usepackage{hyperref}
\usepackage{latexsym,amsmath}
\usepackage{enumerate}
\usepackage{amsfonts}
\usepackage{amssymb}
\usepackage{geometry}
\usepackage{latexsym}
\usepackage{fixmath}
\usepackage[prependcaption]{todonotes}

\usepackage[T1]{fontenc}
\usepackage{fourier}
\usepackage{bbm}
\usepackage{color}
\usepackage{fullpage}

\newcommand{\neigh}{\nabla}

\newcommand{\SYM}{{\bf SYM}}
\newcommand{\BAL}{{\bf BAL}}
\newcommand{\POS}{{\bf POS}}
\newcommand{\MIN}{{\bf MIN}}
\newcommand{\UNI}{{\bf UNI}}

\renewcommand{\epsilon}{\eps}

\newcommand{\corr}{\mathrm{corr}}

\newcommand{\cPcent}{\cP_*^{2}}

\newcommand{\ism}{\cong}

\newcommand{\GG}{\mathbb G}
\newcommand{\FF}{\mathbb F}

\newcommand{\TT}{\mathbb T}

\newcommand\MU{\vec\mu}

\newcommand\vb{\vec b}
\newcommand\vm{\vec m}
\newcommand\vn{\vec n}
\newcommand\NU{\vec\nu}
\newcommand\cMU{\check\MU}

\newcommand\vU{\vec U}

\newcommand\GAMMA{{\vec\gamma}}

\newcommand\PSI{\vec\psi}
\newcommand\RHO{{\vec\rho}}
\newcommand\PHI{\vec\Phi}

\newcommand\nix{\,\cdot\,}

\newcommand\dd{{\mathrm d}}

\newcommand\G{\vec G}

\numberwithin{equation}{section}

\renewcommand{\vec}[1]{\boldsymbol{#1}}

\newcommand\SIGMA{\vec\sigma}
\newcommand\CHI{\vec\chi}
\newcommand\TAU{\vec\tau}

\newtheorem{definition}{Definition}[section]
\newtheorem{claim}[definition]{Claim}
\newtheorem{example}[definition]{Example}

\newtheorem{theorem}[definition]{Theorem}
\newtheorem{lemma}[definition]{Lemma}
\newtheorem{proposition}[definition]{Proposition}
\newtheorem{corollary}[definition]{Corollary}

\newcommand\fS{\mathfrak{S}}

\newcommand\fF{\mathfrak{F}}

\newcommand\cA{\mathcal{A}}
\newcommand\cB{\mathcal{B}}
\newcommand\cC{\mathcal{C}}

\newcommand\cF{\mathcal{F}}
\newcommand\cG{\mathcal{G}}
\newcommand\cE{\mathcal{E}}
\newcommand\cU{\mathcal{U}}

\newcommand\cH{\mathcal{H}}
\newcommand\cS{\mathcal{S}}

\newcommand\cI{\mathcal{I}}
\newcommand\cK{\mathcal{K}}
\newcommand\cJ{\mathcal{J}}

\newcommand\cM{\mathcal{M}}

\newcommand\cP{\mathcal{P}}
\newcommand\cX{\mathcal{X}}
\newcommand\cY{\mathcal{Y}}

\newcommand\cZ{\mathcal{Z}}
\def\cR{{\mathcal R}}
\def\cC{{\mathcal C}}
\def\cE{{\mathcal E}}

\newcommand\vx{\vec x}
\newcommand\vy{\vec y}
\newcommand\THETA{\vec\theta}

\newcommand{\beq}{\begin{equation}} \newcommand{\eeq}{\end{equation}}
\newcommand{\dc}{\dcond}
\newcommand{\dKS}{d_{\mathrm{KS}}}
\newcommand{\dr}{d_{\mathrm{rec}}}

\newcommand\thet{\vartheta}

\newcommand\eul{\mathrm{e}}
\newcommand\eps{\varepsilon}

\newcommand\ZZ{\mathbb{Z}}
\newcommand\NN{\mathbb{N}}

\newcommand\Var{\mathrm{Var}}
\newcommand\Erw{\mathbb{E}}
\newcommand{\vecone}{\vec{1}}

\newcommand{\Po}{{\rm Po}}

\newcommand\TV[1]{\left\|{#1}\right\|_{\mathrm{TV}}}
\newcommand\tv[1]{\|{#1}\|_{\mathrm{TV}}}
\newcommand\dTV{d_{\mathrm{TV}}}

\newcommand{\bink}[2] {{\binom{#1}{#2}}}

\newcommand\bc[1]{\left({#1}\right)}
\newcommand\cbc[1]{\left\{{#1}\right\}}
\newcommand\bcfr[2]{\bc{\frac{#1}{#2}}}
\newcommand{\bck}[1]{\left\langle{#1}\right\rangle}
\newcommand\brk[1]{\left\lbrack{#1}\right\rbrack}
\newcommand\scal[2]{\bck{{#1},{#2}}}

\newcommand\abs[1]{\left|{#1}\right|}

\newcommand\RR{\mathbb{R}}

\newcommand{\whp}{w.h.p.}

\newcommand{\stacksign}[2]{{\stackrel{\mbox{\scriptsize #1}}{#2}}}
\newcommand{\tensor}{\otimes}

\newcommand{\Erdos}{Erd\H{o}s}
\newcommand{\Renyi}{R\'enyi}

\newcommand\pr{\mathbb{P}} 
\renewcommand\Pr{\pr} 

\newcommand\Lem{Lemma}
\newcommand\Prop{Proposition}
\newcommand\Thm{Theorem}
\newcommand\Def{Definition}
\newcommand\Cor{Corollary}
\newcommand\Sec{Section}

\newcommand\id{\mathrm{id}}
\newcommand{\Eig[1]}{\mathrm{Eig}^{\ast} (#1)}
\newcommand{\eig}{\mathrm{Eig}}
\newcommand{\Pomast}{\cP^2_\ast (\Omega)}

\newcommand{\dcond}{d_{\mathrm{cond}}}
\newcommand{\dsat}{d_{\mathrm{sat}}}
\DeclareMathOperator{\Tr}{tr}

\begin{document}

\title{The replica symmetric phase of random constraint satisfaction problems}

\author{Amin Coja-Oghlan$^{*}$, Tobias Kapetanopoulos$^{**}$, Noela M\"uller}
\thanks{$^{*}$The research leading to these results has received funding from the European Research Council under the European Union's Seventh 
Framework Programme (FP7/2007-2013) / ERC Grant Agreement n.\ 278857--PTCC\\
$^{**}$Supported by Stiftung Polytechnische Gesellschaft PhD grant}

\address{Amin Coja-Oghlan, {\tt acoghlan@math.uni-frankfurt.de}, Goethe University, Mathematics Institute, 10 Robert Mayer St, Frankfurt 60325, Germany.}

\address{Tobias Kapetanopoulos, {\tt kapetano@math.uni-frankfurt.de}, Goethe University, Mathematics Institute, 10 Robert Mayer St, Frankfurt 60325, Germany.}

\address{Noela M\"uller, {\tt nmueller@math.uni-frankfurt.de}, Goethe University, Mathematics Institute, 10 Robert Mayer St, Frankfurt 60325, Germany.}

\begin{abstract}
Random constraint satisfaction problems play an important role in computer science and combinatorics.
For example, they provide challenging benchmark instances for algorithms and they have been harnessed in probabilistic constructions of combinatorial structures with peculiar features.
In an important contribution [Krzakala et al., PNAS~2007] physicists made several predictions on the precise location and nature of phase transitions in random constraint satisfaction problems.
Specifically, they predicted that their satisfiability thresholds are quite generally preceded by several other thresholds that have a substantial impact both combinatorially and computationally.
These include the condensation phase transition, where long-range correlations between variables emerge, and the reconstruction threshold.
In this paper we prove these physics predictions for a broad class of random constraint satisfaction problems.
Additionally, we obtain contiguity results that have implications on Bayesian inference tasks, a subject that has received a great deal of interest recently (e.g., [Banks et al., COLT 2016]).
\end{abstract}

\maketitle

 \section{Introduction}\label{Sec_intro}

\subsection{Background and motivation}
Random constraint satisfaction problems (`CSPs') have come to play a prominent role at the junction of combinatorics, computer science and statistical physics~\cite{ANP}.
In combinatorics the study of random CSPs goes back to the seminal paper by \Erdos\ and \Renyi\ that started the theory of random graphs~\cite{ER60}.
In modern language they posed the problem of pinpointing the threshold for $q$-colorability in random graphs, a question that remains open to this day but that has nevertheless sparked  pathbreaking contributions (e.g., \cite{AchNaor,ShamirSpencer}).
In computer science random CSPs are of fundamental interest as algorithmic benchmarks for computationally hard problems such as graph colouring or $k$-SAT and as gadgets for cryptographic constructions or reductions in complexity theory (e.g.,~\cite{Cheeseman,Feige,Feldman2015,Andreas,Goldreich}).

Random CSPs also occur as models of disordered systems in statistical physics.
Specifically, while in classical models such as the Ising model on $\ZZ^d$ the interactions follow a regular lattice structure, geometries induced by sparse random graphs have been proposed as models of (spin-)glasses~\cite{MM}.
Over the last 20 years physicists have devised a non-rigorous but analytic technique for the study of these models called the cavity method.
The rigorous vindication of its `predictions' has emerged as a challenging but fruitful endeavour in the course of which novel proof techniques  have been discovered (e.g., the interpolation method~\cite{bayati,FranzLeone,Guerra,PanchenkoTalagrand}).

A fundamental question in the study of random CSPs concerns their {\em satisfiability thresholds}, which mark the largest density of constraints to variables up to which a solution likely exists.
There has been tremendous progress over the past two decades (e.g., \cite{nae,AchNaor,yuval,KostaSAT,DSS1,DSS3}).
But in an important paper~\cite{pnas} physicists predicted the existence of several further phase transitions preceding the satisfiability threshold.
At these other transition points the geometry of the solution space and thus, probabilistically speaking, the Boltzmann distribution induced by the CSP instance undergo qualitative changes.
These are expected to affect, e.g., the performance of algorithms attempting to construct solutions or the mixing times of Markov chains~\cite{Barriers,GS1,GS2,montanari2011reconstruction}.

The most important one of these phase transitions 
is called the {\em condensation phase transition}.
Generally expected to occur at a constraint density within a whisker of the satisfiability threshold, it is thought to mark the onset of extensive long-range correlations.
More precisely, for densities below condensation the correlations between variables that are far apart in the hypergraph induced by the CSP instance are expected to decay.
The regime of densities below the condensation phase transition is therefore called the {\em replica symmetric} phase.
By contrast,  long-range correlations are deemed to persist beyond the condensation threshold; in physics jargon, replica symmetry is broken.
Furthermore, the {\em reconstruction threshold}, which in most examples occurs at a constraint density well below the condensation threshold, marks the onset of point-to-set correlations where the value assigned to a variable $x$ remains correlated with the values assigned jointly to all the variables at distance $\ell$ from $x$ even as $\ell\to\infty$.
In the physics literature this has been associated with the shattering of the set of solutions into numerous tiny clusters~\cite{MM,MPZ}.

This paper contributes a systematic rigorous study of the replica symmetric phase for a broad family of random CSPs, for which we prove many of the conjectures from~\cite{pnas}.
In particular, we pinpoint the precise condensation phase transition and we establish the absence of long-range correlations below this threshold.
Concrete examples of CSPs covered by theses results include the random graph colouring problem, random hypergraph colouring and the random $k$-NAESAT problem.
In all of these specific examples the generic approach developed here enables us to significantly strengthen prior results that were derived via problem-specific arguments.

In terms of techniques, the present paper builds upon~\cite{SoftCon,CKPZ}.
These papers almost exclusively dealt with models with soft constraints only (such as the Potts antiferromagnet), whereas here we extend those methods to the case of hard constraints that strictly forbid certain value combinations (such as graph colouring).
While this difference may seem innocuous, the presence of hard constraints causes substantial technical complications.
Before stating the main results about general CSPs in \Sec~\ref{Sec_results}, in the following paragraphs we present some of their implications on two particularly well-studied examples, the random $k$-NAESAT problem and the random graph colouring problem.

\subsection{Random $k$-NAESAT}
Let $k\geq3$ be an integer and consider the usual model  $\FF_k(n,m)$ of a random propositional formula over the Boolean variables $x_1,\ldots,x_n$.
Thus, $\FF_k(n,m)$ is obtained by inserting $m$ independent random clauses of length $k$ such that no variable appears twice in the same clause.
We recall that a Boolean assignment $\sigma$ of $x_1,\ldots,x_n$ is {\em NAE-satisfying} if under both $\sigma$ and its binary inverse $\bar\sigma$ all $m$ clauses evaluate to `true'.
Here NAE stands for  `Not-All-Equal', because every clause must contain at least one literal that evaluates to true as well as at least one that evaluates to false.
To parametrise the problem conveniently we will consider formulas with 
	$\vm=\Po(dn/k)$ clauses for a fixed number $d>0$.
Thus, any variable occurs in $d$ clauses on average.
The problem of deciding whether a given $k$-CNF formula is NAE-satisfiable is NP-complete~\cite{Schaefer}.

The random $k$-NAESAT problem is one of the standard examples of random CSPs and has received a great deal of attention.
In particular, in an influential paper Achlioptas and Moore~\cite{nae} pioneered the use of the second moment method for estimating the partition functions of random CSPs with the example of
 random $k$-NAESAT.
To be precise, in the case of $k$-NAESAT the partition function $Z(\FF_k(n,m))$ is simply the total number of NAE-satisfying assignments of the random formula.
A straightforward first moment calculation shows that with high probability,
	\begin{equation}\label{eqh2c1m}
	\sqrt[n]{Z(\FF_k(n,\vm))}\leq2(1-2^{1-k})^{d/k+o(1)}.
	\end{equation}
Indeed, there are $2^n$ possible truth assignments.
Moreover, the probability that any fixed truth assignment fails to NAE-satisfy one random $k$-clause is $2^{1-k}$ because out of the $2^k$ possible assignments of $k$ variables precisely two fail to be NAE-satisfying.
In particular, (\ref{eqh2c1m}) implies that $\FF_k(n,\vm)$ fails to be NAE-satisfiable \whp\ if
	$$d>k2^{k-1}\ln 2-k\ln 2/2.$$

The upper bound (\ref{eqh2c1m}) is clearly tight for small densities $d$.
For instance, if $d<1/(k-1)$ is so small that the random hypergraph induced by $\FF_k(n,\vm)$ does not contain a giant component \whp, then 
$Z(\FF_k(n,\vm))=\Theta(2^n(1-2^{1-k})^{\vm})$ \whp, as is easily verified by counting NAE-solutions of acyclic formulas.
But remarkably, Achlioptas and Moore showed via the second moment method that \eqref{eqh2c1m} remains tight for much larger densities, namely for 
	$d<k2^{k-1}\ln 2-k(1+\ln2/2).$
Subsequently Coja-Oghlan and Zdeborov\'a~\cite{COZ} improved this bound slightly and showed that (\ref{eqh2c1m}) continues to be tight so long as
	\begin{equation}\label{eqh2c2m}
	d<k2^{k-1}\ln 2-k\bc{\frac{\ln2}2+\frac14}+\eps_k,
	\end{equation}
where $\eps_k$ hides an error term that tends to zero in the limit of large $k$.
In fact, up to the precise value of $\eps_k$ the bound \eqref{eqh2c2m} matches the density up to which (\ref{eqh2c1m}) has been predicted to be tight via the cavity method~\cite{pnas}.
However, due to the $\eps_k$ the expression \eqref{eqh2c2m} is informative only for (very) large $k$.

By contrast, the following theorem establishes the {\em exact} physics prediction for every $k\geq3$.
To state the result  we introduce $\Lambda(x)=x\ln x$ with the convention that $\Lambda(0)=0$.
Further, $\GAMMA$ signifies a $\Po(d)$ random variable.
Finally, let $\cP_*[0,1]$ be the set of all probability measures $\pi$ on $[0,1]$ with mean $1/2$ and let
$(\RHO_i^{(\pi)})_{i\geq1}\in[0,1]^\infty$ denote a family of samples from $\pi$, mutually independent and independent of $\GAMMA$.

\begin{theorem}\label{Thm_NAEcond}
For $k\geq3$, $d>0$ and $\pi\in\cP_*[0,1]$ let
	\begin{align*}
	\cB(d,\pi)&=
	\Erw\brk{\frac{
			\Lambda\bc{\prod_{i=1}^{\vec\gamma}\bc{1-\prod_{j=1}^{k-1}\RHO_{ki+j}^{(\pi)}}+\prod_{i=1}^{\vec\gamma}\bc{1-\prod_{j=1}^{k-1}(1-\RHO_{ki+j}^{(\pi)})}}}{2(1-2^{1-k})^{\vec\gamma}}
	-\frac{d(k-1)\Lambda\bc{1-\prod_{j=1}^k\RHO_j^{(\pi)}-\prod_{j=1}^k(1-\RHO_j^{(\pi)})}}{k(1-2^{1-k})}},\\
\dcond &=  \inf\left\{d>0\,:\, \sup_{\pi\in\cP_*[0,1]} \cB(d,\pi) > \ln 2 + \frac{d}{k}\ln(1-2^{1-k})\right\}.%\label{eq:dcond}
\end{align*}
Then  for all $d<\dc$,
	\begin{align*}
	\sqrt[n]{Z(\FF_k(n,\vm))}&\quad\stacksign{$n\to\infty$}\longrightarrow\quad 2(1-2^{1-k})^{d/k}&&\mbox{in probability.}
	\end{align*}
By contrast, for any $d>\dc$ there exists $\eta>0$ such that
	\begin{align}\label{eqThm_NAEcond}
	\limsup_{n\to\infty}\pr\brk{\sqrt[n]{Z(\FF_k(n,\vm))}>2(1-2^{1-k})^{d/k}-\eta}^{\frac1n}<1.
	\end{align}
\end{theorem}

Thus, $\dc$ marks the precise threshold up to which (\ref{eqh2c1m}) is tight.
Indeed, (\ref{eqThm_NAEcond}) shows that $\sqrt[n]{Z(\FF_k(n,\vm))}$ takes a strictly smaller value with probability $1-\exp(-\Omega(n))$ for $d>\dc$.
Admittedly, the formula for $\dc$, involving an optimisation problem over a probability measure on the unit interval, is not explicit and potentially difficult to evaluate.
But given the combinatorial intricacy of the (NP-hard) $k$-NAESAT problem 
we may just not be entitled to a simple answer.
More generally, the physics predictions typically take the form of distributional optimisation problems.
Yet it also seems plain that elementary techniques such as the combinatorial second moment method will hardly suffice to establish such predictions precisely.

\Thm~\ref{Thm_cond} shows that $\dc$ is a genuine phase transition, called the condensation phase transition, since the functions
$d\mapsto\Erw\sqrt[n]{Z(\FF_k(n,\vm))}$ fail to converge to an analytic limit at the point $\dc$.
Indeed, the theorem implies that the limit exists and matches the entire function $2(1-2^{1-k})^{d/k}$ for $d<\dc$.
By contrast, for $d>\dc$ the limit may not exist, and even if it does it is strictly smaller than $2(1-2^{1-k})^{d/k}$.

Additionally, up to $\dc$ there occurs an important decay of correlation phenomenon.
Formally, let $\SIGMA,\TAU$ signify two independently chosen random NAE-satisfying assignments of $\FF_k(n,\vm)$ (given that the formula is NAE-satisfiable).
Representing the Boolean values false and true by $\pm1$, we think of $\SIGMA,\TAU$ as vectors in $\{\pm1\}^n$.
Let us denote the expectation with respect to the choice of $\SIGMA,\TAU$ given the random formula $\FF_k(n,\vm)$ by $\bck\nix_{\FF_k(n,\vm)}$,
whereas we use the standard symbols $\Erw\brk\nix$, $\pr\brk\nix$ to refer to the choice of $\FF_k(n,\vm)$ itself.
The second moment argument of Achlioptas and Moore~\cite{nae} was based on showing by elementary calculations that for $d/k<2^{k-1}\ln 2-(1+\ln2/2)$, the vectors $\SIGMA,\TAU$ are nearly perpendicular \whp\
Formally, their inner product satisfies $\SIGMA\cdot\TAU=o(n)$ \whp\
According to the cavity method, this property should extend right up to the condensation threshold $\dc$.
The following theorem verifies this conjecture.

\begin{theorem}\label{Thm_NAEoverlap}
Let $k\geq3$. For all $0<d<\dc$ we have 
	\begin{equation}\label{eqThm_NAEoverlap}
	\lim_{n\to\infty}\frac1n\Erw\brk{\bck{\abs{\SIGMA\cdot\TAU}}_{\FF_k(n,\vm)}\mid Z(\FF_k(n,\vm))>0}=0.
	\end{equation}
\end{theorem}

Due to standard results about probability measures on the cube $\{\pm1\}^n$ we can express (\ref{eqThm_NAEoverlap}) in terms of pairwise correlations between the truth values assigned to variables~\cite{Victor}.
Specifically, \eqref{eqThm_NAEoverlap} is equivalent to the statement
	\begin{align}\label{eqThm_NAEoverlap2}
	\lim_{n\to\infty}\frac1{n^2}\sum_{i,j=1}^n\Erw\brk{\abs{\bck{\SIGMA(x_i)\cdot\SIGMA(x_j)}_{\FF_k(n,\vm)}}\,\big|\,Z(\FF_k(n,\vm))>0}&=0
	\end{align}
Hence, for $d<\dc$ the truth values $\SIGMA(x_i)$, $\SIGMA(x_j)$ assigned to two randomly chosen variables $x_i,x_j$ are asymptotically independent.
Physics calculations predict that neither \eqref{eqThm_NAEoverlap} nor \eqref{eqThm_NAEoverlap2} continue to hold for $d>\dc$.

Finally, let us refer to
	$$\dsat=\inf\cbc{d>0:\liminf_{n\to\infty}\,\pr\brk{Z(\FF_k(n,\vm))>0}<1}$$
as the {\em satisfiability threshold} of the random $k$-NAESAT problem.
Coja-Oghlan and Panagiotou~\cite{Catching} determined the asymptotic value of $\dsat$, showing that
	\begin{align}\label{eqCOP}
	\dsat&=k2^{k-1}\ln 2-k\bc{\frac{\ln 2}{2}+\frac14}+\eps_k&\mbox{where $\eps_k\to0$ as $k\to\infty$.}
	\end{align}
While (\ref{eqCOP}) is asymptotically tight in the limit of large $k$, the condensation threshold $\dc$ from \Thm~\ref{Thm_NAEcond} yields a lower bound on $\dsat$ for {\em every} $k\geq3$. This is the best current lower bound for any specific $k$.

\subsection{Random graph coloring}\label{Sec_intro_col}
Let $\GG=\GG(n,p)$ denote the random graph on $n$ vertices $\{1,\ldots,n\}$ where each of the $\bink n2$ possible edges is present with probability $p$ independently.
If we set $p=d/n$ for a fixed $d>0$ and a large $n$, then the average degree of the random graph will be asymptotically equal to $d$.
Let $q\geq3$ be a number of colours and let $Z_q(\GG(n,p))$ be the number of $q$-colourings of the random graph.
Understanding the random variable $Z_q(\GG(n,p))$ for given $d,q$ is one of the longest-standing challenges in the theory of random graphs.
In fact, the problem of identifying the {\em $q$-colorability threshold}, i.e., the largest value of $d$ up to which $Z_q(\GG(n,p))>0$ \whp, goes back to the seminal paper of \Erdos\ and \Renyi~\cite{ER60}.

Like in the random $k$-NAESAT problem it is easy to determine the number of $q$-colourings for $d<1$, where the there is no giant component yet.
In this regime it is easily verified that 
	\begin{equation}\label{eqqcol}
	\sqrt[n]{Z_q(\GG(n,p))}\quad\stacksign{$n\to\infty$}\longrightarrow\quad q(1-1/q)^{d/2}\qquad\mbox{ in probability.}
	\end{equation}
In~\cite{CKPZ} the largest average degree $\dc$ up to which this convergence in probability occurs was determined.
The precise formula involves a stochastic optimisation problem akin to the one in \Thm~\ref{Thm_NAEcond}. Asymptotically in the limit of large $q$ we have $\dc=(2q-1)\ln q-2\ln 2+\eps_q$.
By comparison, for $d>(2q-1)\ln q-1+\eps_q$ the random graph fails to be $q$-colourable \whp~\cite{Covers}.

Equation (\ref{eqqcol}) provides a `first order' approximation to $Z_q(\GG(n,p))$ up to errors of size $\exp(o(n))$.
But how large might the fluctuations of $Z_q(\GG(n,p))$ be?
Clearly, adding, removing or rewiring a single edge is apt to change $Z_q(\GG(n,p))$ by a constant factor (or even more).
Consequently, since key variables such as the number of vertices and edges in the giant component have fluctuations of order $\Theta(\sqrt n)$ even once we condition on the total number $\vm$ of edges, one might expect $Z_q(\GG(n,p))$ to exhibit  multiplicative fluctuations of order at least $\exp(\Theta(\sqrt n))$.
However,  Bapst et al.~\cite{Silent} proved that for $q$ exceeding a certain (undetermined but large) constant $q_0$ the random variable $Z_q(\GG(n,p))$ is concentrated remarkably tightly for all $d<\dc$.
More specifically, $Z_q(\GG(n,p))$ has {\em bounded} multiplicative fluctuations once we condition on the number $\vm$ of edges of the random graph.
In fact, Ra\ss mann~\cite{Feli2} determined the precise limiting distribution of $Z_q(\GG(n,p))$ given $\vm$ for all $d<\dc$ under the assumption that $q>q_0$ is sufficiently large.
As an application of our general results we obtain the limiting distribution of $\ln Z_q(\GG(n,p))$ for $d<\dc$ for all $q\geq3$, thereby closing the gap left by~\cite{Silent,Feli2}.

\begin{theorem}\label{Thm_qcol}
Let $q\geq3$ and $0<d<\dc(q)$.
With $(K_\ell)_{\ell\geq3}$ a sequence of independent Poisson variables with means $\Erw[K_\ell]=d^\ell/(2\ell)$, let
	$$\cK=\prod_{\ell=3}^\infty (1+\delta_\ell)^{K_\ell}\exp\bc{-\frac{d^\ell\delta_\ell}{2\ell}}\qquad\mbox{where}\quad\delta_\ell=
		-(1-q)^{1-\ell}.$$
Then $\cK>0$ almost surely,  and we have the following convergence in distribution:
	$$\frac{Z_{q}(\GG(n,p))}{q^{n}
		\bc{1-1/q}^{\vm}}
		\quad\stacksign{$n\to\infty$}{\longrightarrow}\quad 
			\sqrt q\bc{1+\frac{d}{q-1}}^{\frac{1-q}2}\exp\bc{-\frac{d\delta_1}2-\frac{d^2\delta_2}4}\cK.$$
\end{theorem}

As an application of \Thm~\ref{Thm_qcol} we obtain a result that characterises the combinatorial structure of typical $q$-colourings of the random graph for all $d<\dc$ very accurately.
A similar result was obtained previously in~\cite{Nor}, but required the extraneous assumption that $q>q_0$ for some very large constant $q_0$.
To formulate the result, let us denote by $\nabla_\ell(\GG,v)$ the subgraph of $\GG$ induced on the set of vertices at distance at most $\ell$ from vertex $v$.
For a fixed $\ell$ and large $n$ this subgraph is a tree \whp\
Furthermore, let $\mu_{\GG, \nabla_{\ell}(\GG,v)}$ denote the distribution on the set of $q$-colourings of $\nabla_\ell(\GG,v)$ induced by a uniformly random $q$-colouring of the entire graph.
For comparison, let $\mu_{\nabla_\ell(\GG,v)}$ be the uniform distribution on the set of all $q$-colourings of the subgraph $\nabla_\ell(\GG,v)$ only.
Clearly, a priori $\mu_{\GG, \nabla_{\ell}(\GG,v)}$ and $\mu_{\nabla_{\ell}(\GG,v)}$ could be quite different because the latter ignores the `external' connections of the boundary vertices at distance $\ell$ from $v$ via (long) paths through $\GG-\nabla_{\ell}(\GG,v)$.
Yet the next theorem shows that for almost all vertices $v$ the two distributions asymptotically coincide.

\begin{theorem}\label{Thm_qcol_lwc}
Let $q\geq3$, $0<d<\dc(q)$ and $\ell\geq1$.
Then
	$\displaystyle\lim_{n\to\infty}\frac1n\sum_{v=1}^n\Erw\brk{\dTV(\mu_{\GG, \nabla_{\ell}(\GG,v)},\mu_{\nabla_{\ell}(\GG,v)})}=0.$
\end{theorem}

As an application of \Thm~\ref{Thm_qcol_lwc} we obtain a 
further result about the reconstruction problem.
We will give a precise definition in \Sec~\ref{Sec_results} below, but intuitively reconstruction occurs when the colour of the vertex $v$ remains correlated with the colours assigned to {\em all} the boundary vertices at distance precisely $\ell$ from $v$ even for large values of $\ell$.
A well known conjecture from~\cite{pnas} asserts that the threshold for reconstruction on the random graph coincides with the reconstruction threshold on the Galton-Watson tree that mimics the local structure of the random graph.
Previously this was confirmed only under the assumption that $q$ be large enough~\cite{Silent,GM,montanari2011reconstruction}.

\section{Main results}\label{Sec_results}

\subsection{Random constraint satisfaction problems}\label{Sec_results1}
In this section we present the main results of the paper for a general family of random CSPs.
To set the stage we introduce a comprehensive model of random CSPs.
The variables take values in a finite domain $\Omega\neq\emptyset$.
They are bound by constraints that each involve precisely $k\geq2$ variables and either discourage or outright forbid certain value combinations.
The formal definition reads as follows.

\begin{definition}\label{Def_CSP}
Let $\Omega\neq\emptyset$ be a finite set and let $\Psi$ be a finite set of functions $\Omega^k\to[0,1]$.
A {\em $\Psi$-constraint satisfaction problem} $G=(V,F,(\partial a)_{a\in F},(\psi_a)_{a\in F})$ comprises
	\begin{itemize}
	\item a set $V$ of \,{\em variables},
	\item a set $F$ of \,{\em constraints},
	\item an ordered $k$-tuple $\partial a=(\partial_1a,\ldots,\partial_ka)\in V^k$ for each $a\in F$ and
	\item a {\em constraint function} $\psi_a\in\Psi$ for each $a\in F$.
	\end{itemize}
An assignment $\sigma\in\Omega^V$ {\em satisfies} $G$ if $\psi_a(\sigma(\partial_1a),\ldots,\sigma(\partial_ka))>0$ for all $a\in F$; in symbols, $\sigma\models G$.
\end{definition}

A $\Psi$-CSP $G$ induces a bipartite graph with vertex sets $V$ and $F$ where $a\in F$ is adjacent to $\partial_1a,\ldots,\partial_ka$.
We will therefore use graph-theoretic terminology and, e.g., refer to       $\partial_1a,\ldots,\partial_ka$ as the neighbours of $a$.
Moreover, the length of shortest paths in the bipartite graph induces a metric on the nodes of $G$.

For a $\Psi$-CSP $G$ and an assignment $\sigma\in\Omega^V$ we let
	$$\psi_G(\sigma)=\prod_{a\in F}\psi_a(\sigma(\partial_1a),\ldots,\sigma(\partial_ka)).$$
Moreover, we introduce the {\em partition function} $Z(G)=\sum_{\sigma\in\Omega^V}\psi_G(\sigma)$
as well as the {\em Boltzmann distribution} 
	$$\mu_G(\sigma)=\psi_G(\sigma)/Z(G)\qquad\qquad(\sigma\in\Omega^V),$$
providing that $Z(G)>0$.
Further, we let
	 $S(G)=\{\sigma\in\Omega^V:\sigma\models G\}$ be the set of satisfying assignments.
In many cases the functions $\psi\in\Psi$ are $\{0,1\}$-valued.
Then $Z(G)=|S(G)|$ is just the number of solutions.
But as we will see in \Sec~\ref{Sec_examples} there are interesting cases where the functions $\psi$ take values strictly between $0$ and $1$. 

Standard examples of CSPs fit the framework provided by \Def~\ref{Def_CSP}.

\begin{example}[hypergraph colouring]\label{Ex_hyp}
Suppose that $k\geq2$ is an integer, that $q\geq2$ is a number of colours and that $g=(V,E)$ is a $k$-uniform hypergraph.
Recall that a $q$-colouring of $g$ is a map $\sigma:V\to\Omega=\{1,\ldots,q\}$ such that for every edge $e\in E$ there exist $v,w\in e$ with $\sigma(v)\neq\sigma(w)$ (i.e., no edge is monochromatic).
Let $\Psi_{k,q}=\{\psi_{k,q}\}$ be the singleton containing the function
	$$\psi_{k,q}:\Omega^k\to\{0,1\},\qquad\sigma\mapsto1-\vecone\{\sigma_1=\cdots=\sigma_k\}.$$
Then we can express the $q$-colorability problem on $g$ as a $\Psi_{k,q}$-CSP $G$ 
whose variables are the vertices $V$ and whose constraints are the edges $E$ of $g$.
For each edge $e$ the $k$-tuple $\partial e$ simply contains the vertices incident with $e$ in $g$ (in any order) and $\psi_e=\psi_{k,q}$.
Of course, the case $k=2$ corresponds to the classical graph colouring problem.
\end{example}

\begin{example}[$k$-NAESAT]\label{Ex_NAE} 
Suppose that $k\geq2$ is an integer and that $g=a_1\wedge\cdots\wedge a_m$ is a propositional formula
over a set $V=\{x_1,\ldots,x_n\}$ of Boolean variables with clauses $a_1,\ldots,a_m$, each containing precisely $k$ literals.
Let $\Omega=\{-1,1\}$ represent the Boolean values `true' and `false' and recall that an assignment $\sigma\in\Omega^V$ is {\em NAE-satisfying} for $g$ if the expression evaluates to `true' under both $\sigma$ and its binary inverse $-\sigma$.
This problem can be expressed as a CSP over the set $\Psi_{k-\mathrm{NAE}}$ containing the $2^k$ constraint functions
	\begin{align*}
	\psi_\tau&:\Omega^k\to\{0,1\},\qquad\sigma\mapsto1-\vecone\{\sigma=\tau\}-\vecone\{\sigma=-\tau\}&(\tau\in\Omega^k).
	\end{align*}
Indeed, we turn $g$ into a $\Psi_{k-\mathrm{NAE}}$-CSP with variables $V$ and constraints $F=\{a_1,\ldots,a_m\}$.
We let $\partial a_i$ be the $k$-tuple of variables occurring in the clause $a_i$.
Moreover, letting $\tau_{i,j}=1$ if the $j$th literal of $a_i$ is negated and $\tau_{i,j}=-1$ otherwise, we let
	$\psi_{a_i}=\psi_{\tau_{i,1},\ldots,\tau_{i,k}}$.
\end{example}

We consider the following `\Erdos-\Renyi' like model of random CSP instances.

\begin{definition}
Suppose that $\Psi$ is a finite set of functions $\Omega^k\to[0,1]$ and that $P$ is a probability distribution on $\Omega^k$.
Then $\GG(n,m,P)$ is the random $\Psi$-CSP with variables $V_n=\{x_1,\ldots,x_n\}$  and constraints $F_m=\{a_1,\ldots,a_m\}$ such that
	\begin{itemize}
	\item $\partial a_1,\ldots,\partial a_m\in V_n^k$  are chosen uniformly from the set of all $n(n-1)\cdots(n-k+1)$ tuples consisting of pairwise distinct variables,
		subject to the condition that the $k$-sets $(\{\partial_1a_i,\ldots,\partial_ka_i\})_{i\leq m}$ are pairwise distinct.
	\item the constraint functions $\psi_{a_1},\ldots,\psi_{a_m}\in P$ are chosen independently from the distribution $P$.
	\end{itemize}
\end{definition}

\noindent
Thus, the constraints $a_1,\ldots,a_m$ are chosen nearly independently.
The only condition is that the hypergraph induced on $V_n$ with edges $\{\{\partial_1a_i,\ldots,\partial_ka_i\}:i=1,\ldots,m\}$  be $k$-uniform and simple.
This condition is necessary to accommodate interesting examples such as the random graph colouring problem.

The main results of this paper apply to all CSPs that satisfy a few (relatively) easy-to-check assumptions.
These come solely in terms of the distribution $P$ on $\Psi$.
Throughout the paper we always denote by $\PSI$ an element of $\Psi$ drawn from $P$.
Moreover, we let
	\begin{align*}
	q&=|\Omega|,&\xi&=q^{-k}\sum_{\sigma\in\Omega^k}\Erw[\PSI(\sigma)].
	\end{align*}
Furthermore, for $\psi:\Omega^k\to[0,1]$ and a permutation $\theta$ of $\{1,\ldots,k\}$ we let
	$$\psi^\theta:\Omega^k\to [0,1],\qquad\sigma\mapsto\psi(\sigma_{\theta(1)},\ldots,\sigma_{\theta(k)})$$
denote the function obtained by permuting the coordinates according to $\theta$.
\emph{From here on we tacitly assume that the set $\Psi$ is closed under permutations, i.e., for every $\psi\in\Psi$ we have $\psi^\theta\in\Psi$.
Moreover, we always assume that $P(\psi)>0$ for all $\psi\in\Psi$ and that}
	\begin{equation}\label{eqminmax}
	\min_{\psi\in\Psi,\sigma\in\Omega^k}\psi(\sigma)<\max_{\psi\in\Psi,\sigma\in\Omega^k}\psi(\sigma).
	\end{equation}

Let us write $\cP(\Omega)$ for the set of all probability distributions on $\Omega$. We identify $\cP(\Omega)$ with the standard simplex in $\RR^\Omega$.
Moreover, we let $\cP_*^2(\Omega)$ be the set of all probability distributions $\pi$ on $\cP(\Omega)$ such that 
	$\int_{\cP(\Omega)}\mu(\omega)\dd\pi(\mu)=1/q$ for all $\omega\in\Omega$.
With these conventions the assumptions on  $P$ read as follows.
	\begin{description}
	\item[SYM] For all $i\in\{1,\ldots,k\}$, $\omega\in\Omega$ and $\psi\in\Psi$ we have 
		$$
		\sum_{\tau\in\Omega^k}\vecone\{\tau_i=\omega\}\psi(\tau)=q^{k-1}\xi
		$$
		and for every permutation $\theta$ and every $\psi\in\Psi$ we have $P(\psi)=P(\psi^\theta)$.
	\item[BAL]   The function
		$$\phi:\mu\in\cP(\Omega)\mapsto\sum_{\tau\in\Omega^k}\Erw[\PSI(\tau)]\prod_{i=1}^k\mu(\tau_i)$$
	is concave and attains its maximum at the uniform distribution on $\Omega$.
\item[MIN]
	Let $\cR(\Omega)$ be the set of all probability distribution $\rho=(\rho(s,t))_{s,t\in\Omega}$ on $\Omega\times\Omega$ such that
		$\sum_{s\in \Omega}\rho(s,t)=\sum_{s\in \Omega}\rho(t,s)=q^{-1}$ for all $t\in\Omega$.
	The function 
		$$\varphi:\rho \in \cR(\Omega) \mapsto\sum_{\sigma, \tau\in \Omega^k}\Erw[\PSI(\sigma)\PSI(\tau)]  \prod_{i=1}^k\rho(\sigma_i,\tau_i)$$
	has the uniform distribution on $\Omega\times\Omega$ as its unique global minimiser.
\item [POS] 	For all $\pi,\pi'\in\cP_*^2(\Omega)$ the following is true. With $\RHO_1,\RHO_2,\ldots$ chosen from $\pi$,
			$\RHO_1',\RHO_2',\ldots$ chosen from $\pi'$ and $\PSI\in\Psi$ chosen from $P$, all mutually independent,  we have
			for every $\ell\geq2$,
	\begin{align*}
	&\qquad\qquad\Erw\left[\left(1-\sum_{\tau\in\Omega^k}\PSI(\tau)\prod_{i=1}^ k\RHO_i(\tau_i)\right)^\ell+(k-1)\left(1-\sum_{\tau\in\Omega^k}\PSI(\tau)	\prod_{i=1}^k \RHO_i'(\tau_i)\right)^\ell
	-k\left(1-\sum_{\tau\in\Omega^k}\PSI(\tau)\RHO_1(\tau_1)\prod_{i=2}^k\RHO_i'(\tau_i)\right)^\ell\right]\ge 0.
	\end{align*}
\item [UNI] If $G$ is a $\Psi$-CSP such that for every constraint $a$ the variables $\partial_1a,\ldots,\partial_ka$ are pairwise distinct and the bipartite graph 
induced by $G$ is unicyclic,  then $G$ has a satisfying assignment.
\end{description}

Conditions {\bf SYM} and {\bf BAL} are symmetry assumptions.
Specifically, {\bf SYM} requires that no constraint exhibits an inherent `preference' for any of the values $\omega\in\Omega$ if the values of the other variables are random.
{\bf BAL} is going to ensure that in a typical solution $\sigma$ to a random CSP there are about $n/q$ variables that take each value $\omega\in\Omega$.
Assumptions {\bf MIN} and {\bf POS} impose convexity conditions that are required for technical reasons.
Finally, {\bf UNI} is going to ensure that in the regime of constraint densities that we study, the probability of being satisfiable is either $1-o(1)$ or $o(1)$.
(In particular, the condition rules out the random graph $2$-colouring problem.)
Conditions {\bf SYM}--{\bf POS} occurred in earlier work on problems with soft constraints~\cite{SoftCon,CKPZ}.

Crucially, the above conditions {\em only} refer to the distribution $P$ on the set $\Psi$ of weight functions.
They are usually (relatively) easy to check.
Indeed, in \Sec~\ref{Sec_examples} we will verify the conditions for several well known examples.
Not all of our results require all of the assumptions, and we shall always indicate in brackets which ones are needed.

\subsection{The condensation phase transition}
In order to state the main theorems in a unified way we let $\vm$ be a random variable with distribution $\Po(dn/k)$ and we introduce $\GG=\GG(n,\vm,P)$.
This way we are left with just the single parameter $d$.
As in the examples in \Sec~\ref{Sec_intro} we can easily calculate $Z(\GG)$ for small values of $d$.
For instance, for $d<1/(k-1)$ the bipartite graph induced by the random CSP does not feature a giant component.
Therefore, {\bf SYM} implies that $Z(\GG)=q^n\xi^{\vm+o(n)}$ \whp\
The following theorem determines the precise threshold up to which this identity holds, the {\em condensation threshold}.
Recall that $\Lambda(x)=x\ln x$.

\begin{theorem}[\SYM, \BAL, \MIN, \UNI]\label{Thm_cond}
 Let $d>0$. With $\vec\gamma$ a $\Po(d)$-random variable, $\RHO_1^{(\pi)},\RHO_2^{(\pi)},\ldots$
chosen from $\pi\in\cPcent(\Omega)$ and $\PSI_1,\PSI_2,\ldots\in\Psi$ chosen from $P$, all mutually independent, let
	\begin{align}\label{eqMyBethe}
	\cB(d,P,\pi)&=
	\Erw\brk{q^{-1}\xi^{-\vec\gamma}
			\Lambda\bc{\sum_{\sigma\in\Omega}\prod_{i=1}^{\vec\gamma}\sum_{\tau\in\Omega^k}\vecone\{\tau_k=\sigma\}\PSI_i(\tau)\prod_{j=1}^{k-1}\RHO_{ki+j}^{(\pi)}(\tau_j)}
	-\frac{d(k-1)}{k\xi}\Lambda\bc{\sum_{\tau\in\Omega^k}\PSI_1(\tau)\prod_{j=1}^k\RHO_j^{(\pi)}(\tau_j)}},\\
\dcond &=  \inf\left\{d>0\,:\, \sup_{\pi\in\Pomast} \cB(d,P,\pi) > \ln q + \frac{d}{k}\ln \xi\right\}.\label{eq:dcond}
\end{align}
Then for all $d<\dc$ we have 
	\begin{align}\label{eqcond1}
	\sqrt[n]{ Z(\GG)}&\qquad\stacksign{$n\to\infty$}\longrightarrow\qquad q\xi^{d/k}\qquad\mbox{in probability}.
	\end{align}
By contrast,  if $P$ also satisfies {\bf POS}, then for any $d >\dc$ there exists $\eps>0$ such that
	\begin{align}\label{eqcond2}
	\limsup_{n\to\infty}
		\pr\brk{\sqrt[n]{Z(\GG)}>q \xi^{d/k}-\eps}^{\frac1n}<1-\eps.
	\end{align}
\end{theorem}

Thus, for $d<\dc$ we have $Z(\GG)=q^{n+o(n)}\xi^{\vm}$ with high probability.
By contrast, $Z(\GG)$ is exponentially smaller than this expression for $d>\dc$.
To be precise, $Z(\GG)\leq q^{n-\Omega(n)}\xi^{\vm}$ with probability $1-\exp(-\Omega(n))$ for $d>\dc$.
Consequently, since $d\mapsto q\xi^{d/k}$ is an entire function,
\Thm~\ref{Thm_cond} shows that $\Erw\sqrt[n]{ Z(\GG)}$, viewed as a function of $d$, fails to converge to an analytic limit at $\dc$ as $n\to\infty$.
Therefore, $\dc$ marks a genuine phase transition.

Further, let us call
	$$\dsat=\inf\cbc{d>0:\liminf_{n\to\infty}\,\pr\brk{Z(\GG)>0}<1}$$
the {\em satisfiability threshold} of the random CSP.
Since \eqref{eqminmax} guarantees that $\xi>0$, we have $q\xi^{d/k}>0$ for all $d>0$.
Hence, (\ref{eqcond1}) shows that  $ Z(\GG)>0$ \whp\ for all $d<\dc$.
In effect,
	\begin{align}\label{eqdcdsat}
	\dc\leq\dsat.
	\end{align}

Most of the prior contributions on lower-bounding satisfiability thresholds of various CSPs via the second moment method (e.g.,~\cite{nae,AchNaor,Ayre2,Greenhill})
actually lower-bound the condensation threshold.
To be precise, suppose that for some $d>0$ the second moment bound 
	$$\Erw[Z(\GG)^2\mid\vm]\leq O(\Erw[Z(\GG)\mid\vm]^2)$$
holds with high probability over the choice of $\vm$.
(For second moment calculations it is vital to condition on $\vm$.)
Then the Paley-Zygmund inequality shows that there exists a constant $\delta>0$ such that \whp\ over the choice of $\vm$,
	$$\pr[Z(\GG)\geq \delta q\xi^{\vm}\mid\vm]\geq \Omega(1).$$
Hence, (\ref{eqcond2}) implies that $d\leq\dc$.
In fact, in most examples of random CSPs 
(\ref{eqdcdsat}) is strictly better than any previously known lower bound on the satisfiability threshold.

\subsection{The Kesten-Stigum bound}
While exact, the formula for $\dc$ from \Thm~\ref{Thm_cond} may not be easy to evaluate. 
However, there is an important upper bound that is.
For a function $\psi\in\Psi$ let  $\Phi_{\psi}\in\RR^{\Omega\times\Omega}$ be the matrix with entries
	\begin{equation}\label{eqPhiMatrices}
	 \Phi_{\psi} (\omega,\omega') =q^{1-k}\xi^{-1}\sum_{\tau\in\Omega^k}\vecone\{\tau_1=\omega,\tau_{2}=\omega'\}
 		\psi(\tau)\qquad\qquad\qquad (\omega,\omega' \in \Omega).
	\end{equation}
Further, let $\Xi$ be the linear operator on the $q^2$-dimensional space $\RR^\Omega\tensor\RR^\Omega$ defined by
	\begin{align}\label{eqXi}
	\Xi&=\Erw[\Phi_{\PSI}\tensor\Phi_{\PSI}].
	\end{align}
Additionally, with $\vecone$ denoting the vector with all entries equal to one, let
	\begin{align}\label{eqSpaceE}
	\cE&=\cbc{z\in\RR^q\tensor\RR^q:\forall y\in\RR^q:\scal{z}{\vecone\tensor y}=\scal{z}{y\tensor\vecone}=0}&&\mbox{and}&
	\dKS&=\bc{(k-1)\max_{x\in\cE:\|x\|=1}\scal{\Xi x}x}^{-1},
	\end{align}
with the convention that $\dKS=\infty$ if $\max_{x\in\cE:\|x\|=1}\scal{\Xi x}x=0$.

\begin{theorem}[\SYM, \BAL]\label{Thm_KS}
We have $\dc\leq\dKS$.
\end{theorem}

In the case of the random graph $q$-colouring problem (see \Sec~\ref{Sec_intro_col} and Example~\ref{Ex_hyp}) we calculate $\dKS=(q-1)^2$.
This expression matches  the {\em Kesten-Stigum bound} that plays a role in broadcasting processes on random trees~\cite{KSBoundBroadcasting}.
Moreover, for the graph colouring problem it was shown in~\cite{CKPZ} that $\dc\leq (q-1)^2$.
Thus, \Thm~\ref{Thm_KS} extends the Kesten-Stigum bound to general CSPs and shows that it always gives an upper bound on the condensation threshold.
While the Kesten-Stigum bound is conjectured to be tight in a few cases (such as random graph $3$-colouring), the bound fails to be tight in others  (such as random graph 5-coloring5-colouring)~\cite{Sly}.
Generally the tightness of the Kesten-Stigum bound has implications on algorithmic problems, a point on which we elaborate below.

\subsection{The number of solutions}
\Thm~\ref{Thm_cond} determines the leading exponential order of the partition function for $d<\dc$.
The following theorem, which is the main result of the paper, takes a closer look and determines the precise limiting distribution of $Z(\GG)$ for $d<\dc$.
Let
	\begin{equation}\label{eqPhi}
	\Phi=\Erw[ \Phi_{\PSI}]\in\RR^{\Omega\times\Omega}
	\end{equation}
and let $\eig(\Phi)$ be the multiset that contains the eigenvalues of $\Phi$ according to their geometric multiplicities.

\begin{theorem}[\SYM, \BAL, \MIN, \UNI]\label{Thm_SSC}
 Suppose that  $0<d<\dcond$.
Let $(K_{\ell})_{\ell\geq1}$ be Poisson variables  with means $\Erw[K_{\ell}]=\frac1{2\ell}(d(k-1))^\ell$ and let $(\PSI_{\ell,i,j})_{\ell,i,j\geq1}$ be a sequence of samples from $P$, all mutually independent.
Then 
	\begin{align*}
	\cK&=\exp\bc{\frac{d(k-1)(1-\Tr(\Phi))}{2}+\vecone\{k=2\}\frac{d^2(1-\Tr(\Phi^2))}4}
		\prod_{\ell=2+\vecone\{k=2\}}^\infty{\exp\bc{\frac{(d(k-1))^\ell}{2\ell}\bc{1-\Tr(\Phi^\ell)}}\prod_{i=1}^{K_{\ell}}\Tr\prod_{j=1}^\ell\Phi_{\PSI_{\ell,i,j}}  }
	\end{align*}
{satisfies $\cK>0$ almost surely}.
Moreover, $\eig(\Phi)\subset(-\infty,0]\cup\cbc 1$ and
	\begin{align}\label{eqThm_SSC}
	\frac{ Z(\GG) }{q^{n+\frac12}\xi^{\vec m}} 
		&\quad\stacksign{$n\to\infty$}\longrightarrow\quad \cK\prod_{\lambda\in\eig(\Phi)\setminus\cbc 1}\sqrt{1-d(k-1)\lambda}
	\end{align}
in distribution.
\end{theorem}

Thus, \Thm~\ref{Thm_SSC} shows that $Z(\GG)$ is remarkably concentrated for $d<\dc$.
Indeed, while one might a priori expect that fluctuations of variables such as the order and size of the giant component of  $\GG$ have a significant knock on effect on $Z(\GG)$ and cause multiplicative fluctuations of order at least $\exp(\Omega(\sqrt n))$, \Thm~\ref{Thm_SSC} shows that $Z(\GG)$ merely  has bounded multiplicative fluctuations.
We are not aware of a general physics prediction as to the limiting distribution of the partition function of random CSPs, although there is a paper on the diluted version of the
Sherrington-Kirkpatrick model~\cite{Fede1} (which does not have hard constraints).

\subsection{The overlap}
One of the main predictions of the physics paper~\cite{pnas} is that for densities $d<\dc$ the Boltzmann distribution $\mu_{\GG}$ does not exhibit extensive long-range correlations.
The next theorem verifies this conjecture.
Define the {\em overlap} of assignments $\sigma,\tau\in\Omega^{V_n}$ as the $\Omega\times\Omega$-matrix $\rho_{\sigma,\tau}=(\rho_{\sigma,\tau}(\omega,\omega'))_{s,t\in\Omega}$ with
	$$\rho_{\sigma,\tau}(\omega,\omega')=|\sigma^{-1}(\omega)\cap\tau^{-1}(\omega')|/n.$$
Since $\sum_{\omega,\omega'}\rho_{\sigma,\tau}(\omega,\omega')=1$, we can view $\rho_{\sigma,\tau}$ as a probability distribution on $\Omega\times\Omega$, namely the empirical distribution of the value combinations $(\sigma(x_i),\tau(x_i))_{i=1,\ldots,n}$.
Let $\bar\rho$ be the uniform distribution on $\Omega\times\Omega$.
Moreover, write $\SIGMA,\TAU$ for two independent samples chosen from $\mu_{\GG}$, $\bck{\nix}_{\GG}$ for the
expectation with respect to $\SIGMA,\TAU$ and $\Erw\brk\nix$ for the expectation with respect to the choice of $\GG$.

\begin{theorem}[\SYM, \BAL, \MIN, \UNI]\label{Thm_overlap}
For all  $0<d<\dcond$ we have
	\begin{align}\label{eqThm_overlap}
	\lim_{n\to\infty}\Erw\brk{\bck{\|\rho_{\SIGMA,\TAU}-\bar\rho\|_{\mathrm{TV}}}_{\GG}\mid Z(\GG)>0}=0.
	\end{align}
\end{theorem}

\noindent
For $d<\dc$ the event $Z(\GG)>0$ occurs \whp\ due to~\eqref{eqdcdsat}.

\Thm~\ref{Thm_overlap} shows that for $d<\dc$ the overlap of two random satisfying assignments $\SIGMA,\TAU$ is about uniform, i.e., there is no extensive 
correlation between $\SIGMA,\TAU$.
Using the general results from~\cite{Victor} regarding probability measures on discrete cubes, we can express this result in terms of pairwise correlations between variables.
Specifically, for $1\leq i<j\leq n$ let $\mu_{\GG,x_i,x_j}$ be the joint distribution of the values $\SIGMA(x_i),\SIGMA(x_j)$.
Thus, $\mu_{\GG,x_i,x_j}$ is a probability distribution on $\Omega\times\Omega$.
Then (\ref{eqThm_overlap}) can be rephrased equivalently as
	\begin{align}\label{eqThm_overlap2}
	\lim_{n\to\infty}\frac1{n^2}\sum_{1\leq i<j\leq n}\Erw\brk{\TV{\mu_{\GG,x_i,x_j}-\bar\rho}\,\big|\, Z(\GG)>0}=0
	\end{align}
(see Appendix~\ref{Apx_epssymm} for a proof).
In other words, for most pairs $i,j$ the values $\SIGMA(x_i),\SIGMA(x_j)$ are asymptotically independent.
Equation~(\ref{eqThm_overlap2}) matches the precise definition of ``static replica symmetry'' from~\cite{pnas,MM}.

\subsection{Local weak convergence}
Since the expected distance between two uniform variables of $\GG$ is $\Omega(\ln n)$,
the correlation decay property~(\ref{eqThm_overlap2}) mostly concerns pairs of variables that are far apart.
Complementing this result, the following theorem deals with the joint distribution of the values of variables in the vicinity of a specific reference variable.
Formally, for a variable $x$ of a CSP instance $G$ let $\neigh_{2\ell}(G,x)$ be the CSP obtained from $G$ by deleting all variables and constraints at a distance greater than $2\ell$ from $x$.
Of course, $\mu_{\neigh_{2\ell}(G,x)}$ denotes the Boltzmann distribution of this CSP.
For comparison, let $\mu_{G,\neigh_{2\ell}(G,x)}$ denote the joint distribution of the variables in $\neigh_{2\ell}(G,x)$ under the Boltzmann distribution $\mu_G$ of the entire CSP $G$.
Thus,
if all functions $\psi$ are $\{0,1\}$-valued, then
 $\mu_{G,\neigh_{2\ell}(G,x)}(\sigma)$ is proportional to the number of possible ways of extending a satisfying assignment $\sigma$ of $\neigh_{2\ell}(G,x)$ to a satisfying assignment of $G$.

A priori the two distributions $\mu_{\GG,\neigh_{2\ell}(\GG,x_i)}$ and $\mu_{\neigh_{2\ell}(\GG,x_i)}$ might be rather different.
Indeed, under $\mu_{\neigh_{2\ell}(\GG,x_i)}$ the boundary variables at distance precisely $2\ell$ from $x_i$ are subject to the sub-CSP $\neigh_{2\ell}(\GG,x_i)$ only, whereas in $\mu_{\GG,\neigh_{2\ell}(\GG,x_i)}$ they are connected to further constraints.
These further constraints are apt to form longish chains (of a typical length of about $\Theta(\ln n)$) through which the boundary variables are connected with each other, at least if $d>1/(k-1)$ exceeds the giant component threshold.
Nevertheless, the following theorem shows that the correlations along these chains decay quickly enough so that the two distributions are close to each other for most variables $x_i$.

\begin{theorem}[\SYM, \BAL, \MIN]\label{Thm_lwc}
Let $0<d<\dc$.
Then for any $\ell\geq1$,
	\begin{align}\label{eqThm_lwc}
	\lim_{n\to\infty}\frac1n\sum_{i=1}^n
		\Erw\brk{\TV{\mu_{\GG,\neigh_{2\ell}(\GG,x_i)}-\mu_{\neigh_{2\ell}(\GG,x_i)}}\,\big|\,Z(\GG)>0}&=0.
	\end{align}

\end{theorem}

\subsection{Reconstruction}\label{Sec_reconstruction}
\Thm~\ref{Thm_lwc} allows us to prove a  prediction from~\cite{pnas} regarding a ``point-to-set'' decorrelation property called non-reconstruction.
Recall that we denote by $\bck\nix_{\GG}$ the expectation with respect to samples $\SIGMA$ from $\mu_{\GG}$.
Let us further denote by $\langle\nix\mid\overline{\neigh_{2\ell}(\GG,x_i)}\rangle_{\GG}$ the conditional expectation given the values $\SIGMA(x)$ of all variables $x$ at a distance greater than $2\ell$ from $x_i$.
Then we define
	\begin{align}\label{eqGcorr}
	\corr(d)&=\limsup_{\ell\to\infty}\limsup_{n\to\infty}\frac1n\sum_{i=1}^n\sum_{\omega\in\Omega}
		\Erw\brk{\vecone\{Z(\GG)>0\}
			\bck{\abs{\bck{\vecone\{\SIGMA(x_i)=\omega\}\big|\overline{\nabla_{2\ell}(\GG,x_i)}}_{\GG}-1/q}}_{\GG}+\vecone\{Z(\GG)=0\}}.
	\end{align}
In words, we choose a random variable $x_i$ and a value $\omega\in\Omega$.
Then we choose a random CSP $\GG$ and check whether $\GG$ is satisfiable.
If so, we draw a sample $\TAU$ from the Boltzmann distribution $\mu_{\GG}$ and fix the variables at a distance greater than $2\ell$ from $x_i$ to the values observed under $\TAU$ (the outer $\bck\nix_{\GG}$).
Subsequently we draw a further sample $\SIGMA$ from the Boltzmann distribution $\mu_{\GG}$ given the boundary condition induced by $\TAU$ (the inner $\bck\nix_{\GG}$). 
The value that we record is by how much the conditional marginal probability differs from $1/q$.
Additionally, unsatisfiable $\GG$ contribute a value of one.
Thus, if $\corr(d)=0$ then typically the value of $x_i$ is independent of {\em all} the values at a large enough distance $\ell$.
The {\em reconstruction threshold}
	\begin{equation}\label{eqdr}
	\dr=\inf\{d>0:\mathrm{corr}(d)>0\}\wedge\dc
	\end{equation}
is defined as the smallest density where this decorrelation property fails (or at most $\dc$).

A priori the reconstruction threshold seems extremely difficult to analyse because the definition of $\corr(d)$ involves the Boltzmann distribution induced by the random graph $\GG$.
However, verifying a prediction from~\cite{pnas}, we prove that the Boltzmann distribution of the random graph can be replaced by that of a random Galton-Watson tree, which is conceptually far simpler.
This multi-type Galton-Watson tree $\TT(d,P)$ mimics the local structure of $\GG$.
Its types are either variables or constraints, which come with a weight function $\psi\in\Psi$.
The root is a variable $r$, and the offspring of a variable is a $\Po(d)$ number of constraints whose weight functions are chosen from $P$ independently.
The parent variable occurs in a random position from $\{1,\ldots,k\}$ in each of these constraints; the positions are also chosen independently for each constraint.
Moreover, each constraint has precisely $k-1$ children, which are variables.
For an integer $\ell\geq0$ we denote by $\TT^{2\ell}(d,P)$ the top $2\ell$ layers of this tree and we define
\begin{equation}\label{def:PlantedCorr}
	\mathrm{corr}^\star(d) = \lim_{\ell\to\infty}\sum_{\omega\in\Omega}
		\Erw\bck{\abs{\bck{\vecone\{\SIGMA(r)=\omega\}\big|\overline{\nabla_{2\ell}(\TT^{2\ell}(d,P),r)}}_{\TT^{2\ell}(d,P)}-1/q}}_{\TT^{2\ell}(d,P)}
\end{equation}
Of course, the outer expectation $\Erw\brk\nix$ refers to the Galton-Watson process, the outer $\bck\nix_{\TT^{2\ell}(d,P)}$ represents the choice of a random boundary condition (i.e., the values of all variables at distance precisely $2\ell$ from $r$), and the inner $\bck\nix_{\TT^{2\ell}(d,P)}$ stands for the conditional distribution of the value $\SIGMA(r)$ given the boundary condition.
The {\em tree reconstruction threshold} is defined as
	$$\dr^\star=\inf\{d>0:\mathrm{corr}^\star(d)>0\}.$$

\begin{theorem}[\SYM, \BAL, \MIN, \POS, \UNI]\label{thrm:TreeGraphEquivalence}
We have $\dr^\star=\dr$.
\end{theorem}

\noindent
Thus, \Thm~\ref{thrm:TreeGraphEquivalence} reduces the study of the reconstruction problem on $\GG$ to the same problem on the random tree $\TT(d,P)$, a task that can be tackled via a number of techniques (such as the `contraction method'~\cite{Bandyopadhyay}).

\subsection{Quiet planting}
A random CSP organically gives rise to an associated distribution on inference problems called the {\em planted model}.
This is a random CSP instance  built around a given `planted' solution.
The algorithmic task is to detect and infer the planted solution from the CSP instance.
This computational challenge, which has a remarkably long history, has been harnessed as a benchmark for algorithms based on a broad variety of paradigms, ranging from combinatorial to spectral methods to semidefinite programming (e.g.,~\cite{AlonKahale,DyerFrieze,KrivelevichVilenchik}).
In addition, planted models have been put forward as one-way function candidates in cryptography~\cite{Goldreich}.

To define the planted model, first draw an assignment $\SIGMA^*\in\Omega^{V_n}$ uniformly at random.
Given $\SIGMA^*$ let $\GG^*(n,m,P,\SIGMA^*)$ be the random CSP instance drawn from the distribution
	\begin{align}\label{eqGGplanted}
	\pr\brk{\GG^*(n,m,P,\SIGMA^*)=G\mid\SIGMA^*}&=\frac{\psi_G(\SIGMA^*)\pr\brk{\GG(n,m,P)=G}}{\Erw[\psi_{\GG(n,m,P)}(\SIGMA^*)]}.
	\end{align}
Thus, we reweigh the prior $\GG(n,m,P)$ according to the weight $\psi_G(\SIGMA^*)$ of the planted assignment.
In the most common case where all functions $\psi\in\Psi$ are $\{0,1\}$-valued, \eqref{eqGGplanted} can be stated equivalently as follows.
\begin{quote}Draw $\GG^*(n,m,P,\SIGMA^*)$ from the conditional distribution of $\G(n,m,P)$ given the event $\{\SIGMA^*\models\G(n,m,P)\}$.
\end{quote}
In other words, $\GG^*(n,m,P,\SIGMA^*)$ is chosen uniformly from the set of all CSP instances for which $\SIGMA^*$ is satisfying.

In the event that $\Erw[\psi_{\GG(n,m,P)}(\SIGMA^*)]=0$, the distribution $\GG^*(n,m,P,\SIGMA^*)$ is undefined.
To deal with this technicality we let $\GG^*$ be the conditional distribution of $\GG^*(n,\vm,P,\SIGMA^*)$ given $\Erw[\psi_{\GG(n,\vm,P)}(\SIGMA^*)]>0$,
where we recall that $\vm$ has distribution $\Po(dn/k)$.
Because in a random assignment $\SIGMA^*$ each value $\omega\in\Omega$ very likely occurs about $n/q$ times, condition {\bf SYM} ensures that the event $\Erw[\psi_{\GG(n,\vm,P)}(\SIGMA^*)]>0$ has probability $1-\exp(-\Omega(n))$ for any fixed $d>0$.

The most modest algorithmic question associated with the planted model is the {\em detection problem} (cf.~\cite{Banks,Decelle,Cris}).
It asks for an algorithm that can distinguish the planted model $\GG^*$ from the null model $\GG$.
Formally, with probability $1/2$ the algorithm is given an input from the distribution $\GG$, and with probability $1/2$ the input is drawn from $\GG^*$.
The task is to discern correctly with high probability from which distribution the input was chosen.
The following theorem shows that $\dc$ marks the threshold from where such an algorithm exists.
Recall that the two random graph models $\GG,\GG^*$ are {\em mutually contiguous} if for any sequence $(\cE_n)_{n\geq1}$ of events we have the equivalence
	\begin{align*}
	\lim_{n\to\infty}\pr\brk{\GG\in\cE_n}&=0&\Leftrightarrow&&\lim_{n\to\infty}\pr\brk{\GG^*\in\cE_n}&=0.
	\end{align*}
By contrast, we call the models {\em mutually orthogonal} if there exists $(\cE_n)_{n\geq1}$ such that
	\begin{align*}
	\lim_{n\to\infty}\pr\brk{\GG\in\cE_n}&=1&\mbox{while}&&\lim_{n\to\infty}\pr\brk{\GG^*\in\cE_n}&=0.
	\end{align*}

\begin{theorem}[\SYM, \BAL, \MIN, \UNI]\label{thmContiquity}
 For all $d<\dcond$ the models $\GG$ and $\GG^\ast$ are mutually contiguous.
If {\bf POS} is satisfied as well, then $\GG$ and $\GG^\ast$ are mutually mutually orthogonal  for all $d>\dc$.
\end{theorem}

In particular, for $d<\dc$ no algorithm can tell with high probability whether its input stems from $\GG$ or $\GG^*$, regardless of the running time.
By contrast, the proof of \Thm~\ref{thmContiquity} yields an (exponential time) algorithm that distinguishes the two distributions \whp\ for $d>\dc$.

The first part of \Thm~\ref{thmContiquity} can be sharpened in an important way.
Namely, the contiguity statement extends to the graph/satisfying assignment pairs $(\GG^*,\SIGMA^*)$ and $(\GG,\SIGMA)$, where we recall that $\SIGMA$ denotes a satisfying assignment drawn from the Boltzmann distribution $\mu_{\GG}$.

\begin{corollary}[\SYM, \BAL, \MIN, \UNI]\label{Cor_thmContiquity}
For every $d<\dcond$ the pairs $(\GG,\SIGMA)$ and $(\GG^\ast,\SIGMA^*)$ are mutually contiguous.
\end{corollary}

\noindent
\Cor~\ref{Cor_thmContiquity} enables us to study typical properties of the pair $(\GG,\SIGMA)$ by way of the planted model $(\GG^\ast,\SIGMA^*)$, a technique known as {\em quiet planting}~\cite{Barriers,quiet}.
This method has proved vital for the analysis of many properties of specific examples of random CSPs (e.g.,~\cite{Anastos,Molloy}).
\Cor~\ref{Cor_thmContiquity} shows that quiet planting is a universal technique and establishes $\dc$ as the precise threshold up to which the method is applicable.

\subsection{Discussion and related work}\label{Sec_discussion}
The results presented in this section vindicate and go in some ways beyond the predictions made in~\cite{pnas} on the basis of the non-rigorous cavity method for a broad class of random constraint satisfaction problems.
In a word, we obtain a very accurate description of the ``replica symmetric'' phase of random CSPs, i.e., of the regime of densities up to the condensation threshold.
Since in many prominent examples the condensation threshold is known to be quite close to the satisfiability threshold, these results typically cover most of the satisfiable regime.
Furthermore, we expect that the `quiet planting' result (\Thm~\ref{thmContiquity} and \Cor~\ref{Cor_thmContiquity}) will pave the way for further detailed results on the evolution of random CSPs.

That said, a number of questions remain open.
Specifically, we know very little about the regime $d>\dc$, i.e., beyond the replica symmetric phase.
For instance, neither the decorrelation property (\ref{eqThm_overlap2}) nor the local convergence property (\ref{eqThm_lwc})
are conjectured to extend beyond $\dc$, but we do not currently have a proof.
Furthermore, in~\cite{pnas} the reconstruction threshold is simply defined as $\dr=\inf\{d>0:\mathrm{corr}(d)>0\}$, without taking the min with $\dc$ as in (\ref{eqdr}).
We conjecture that these two definitions are equivalent, which would follow  immediately if we knew that (\ref{eqThm_overlap}) does not hold for $d>\dc$.
Also apart from the example of the regular $k$-NAESAT problem for large $k$~\cite{SSZ} the limit of $\sqrt[n]{Z(\GG)}$ is not known for $d>\dc$ for any random CSP.

An important feature of the results presented here is that they apply to CSPs with very small average degrees.
In most previous work, particularly in work based on combinatorial second moment arguments~\cite{Cond,Catching,KostaSAT,DSS1,DSS3}, the assumption that the average variable degree be sufficiently big is 
endemic.
The assumption is usually made implicitly by requiring, e.g., that the number $q$ of colours in the graph colouring problem or the clause length $k$ in a random $k$-NAESAT problem be sufficiently big.
Roughly speaking, these combinatorial arguments effectively use the notion that a sufficiently dense \Erdos-\Renyi\ graph is not very far from regular.
By contrast, since here we avoid such asymptotic arguments, we are in a position to do away with implicit or explicit density assumptions.

One of the guiding themes in the theory of random CSPs is the quest for satisfiability thresholds.
Despite considerable efforts to this day the {\em exact} thresholds are known in only a handful of cases such as random 2-SAT,
random $1$-in-$k$-SAT, random $k$-XORSAT and random linear equations~\cite{ACIM,Ayre,mick,Cuckoo,DM,Goerdt,PS}.
Additionally, a line of work on the second moment method~\cite{nae,yuval,Catching,KostaSAT,DSS1} culminated in the exact computation of the $k$-SAT threshold for large $k$~\cite{DSS3}.
In other cases such as (hyper)graph colouring upper and lower bounds are known that differ by a small additive constant in the limit of large $k$ and/or $q$~\cite{AchNaor,Ayre2,Cond,Covers,COZ,Greenhill}.
We observed that as a byproduct \Thm~\ref{Thm_SSC} yields lower bounds on the satisfiability thresholds of several problems, particularly hypergraph colouring and random $k$-SAT for small $k$, which are at least as good (and likely better) than the ones obtained in prior work~\cite{yuval,Ayre2,Greenhill}.

While in \Sec~\ref{Sec_examples} we will see many examples of random CSPs that satisfy the assumptions {\bf SYM}, {\bf BAL}, etc., there are  a few interesting ones that don't.
For instance, the random $k$-SAT problem fails to satisfy {\bf SYM}.
At the same time, it is easy to prove that in random $k$-SAT the number of solutions is not as tightly concentrated as \Thm~\ref{Thm_SSC} shows it is in the case of problems that satisfy our assumptions.
In fact, the random $k$-SAT partition function has multiplicative fluctuations of order $\exp(\Omega(\sqrt n))$.
Thus, 
random $k$-SAT is materially different.

\Thm s~\ref{Thm_KS} and~\ref{thmContiquity} can be seen as generalisations of results obtained in~\cite{Banks,SoftCon} for the stochastic block model, a planted version of the Potts model that has become a prominent benchmark for Bayesian inference~\cite{AbbeSurvey,Cris}.
In the stochastic block model the Kesten-Stigum bound marks the point from where an efficient algorithm is known to solve the detection problem~\cite{abbe2015detection}.
But generally the Kesten-Stigum bound is strictly greater than the condensation threshold, and it has been conjectured that in the intermediate regime the detection problem can be solved in exponential but not in polynomial time~\cite{Decelle}.
In light of \Thm s~\ref{Thm_KS} and~\ref{thmContiquity} it would be interesting to see if the detection problem can be solved efficiently for general random CSPs if $d>\dKS$, and in fact if there are examples of (in the worst case NP-hard) random CSPs where efficient algorithms succeed for $\dc<d<\dKS$.

With respect to proof techniques the present work builds strongly upon the methods developed in~\cite{SoftCon,CKPZ}.
The additional technical challenge that we need to confront is the presence of {\em hard} constraints that strictly forbid certain value combinations.
In other words, we allow constraint functions $\psi$ that may take the value $0$, whereas
\cite{SoftCon} deals with soft constraints only, as does~\cite{CKPZ}, apart from an ad-hoc limiting result about the condensation threshold in the random graph colouring problem.
We will discuss the difficulties that hard constraints cause in more detail as we proceed, but roughly speaking the matter is as follows.
One of the main proof steps is to quantify precisely the evolution of the partition function of the random CSP if we add one random constraint after the other.
While we can use the techniques from~\cite{SoftCon,CKPZ} directly to analyse the {\em typical} effect of adding a hard constraint, there is an error probability that these estimates are off.
In the case of soft constraints, this is not a very serious issue because the impact of a single soft constraint cannot be catastrophic.
But in the presence of hard constraints it can.
In fact, a single awkward constraint can wipe out all satisfying assignments in one stroke.
In summary, we will still follow the strategy developed in~\cite{SoftCon,CKPZ}, but we have to come up with new ideas to cope with `exceptional' cases more accurately.
Hence, throughout \Sec s~\ref{Sec_beta} and~\ref{sec:ProofPreCond} we repeatedly adapt or apply arguments from~\cite{SoftCon,CKPZ}.
To avoid repetitions we put off those bits of the arguments that required only minute amendments to the appendix.
Additionally, we will be able to extend several of the results from~\cite{SoftCon,CKPZ} to the case of hard constraints directly by a limiting argument.
More details can be found in \Sec~\ref{Sec_Proofs}, which contains a proof outline.

The proofs of \Thm s~\ref{Thm_lwc} and~\ref{thrm:TreeGraphEquivalence} about local weak convergence and the reconstruction problem are based on a new argument that is somewhat more straightforward than prior ones from~\cite{SoftCon,Nor,montanari2011reconstruction}.
The basic proof idea, which goes back to the work of Gerschenfeld and Montanari~\cite{GM}, is to derive the desired properties of the Boltzmann distribution from the overlap result, \Thm~\ref{Thm_overlap} in our case.
But the new insight here is that this implication can be obtained fairly directly from a key statement called the Nishimori identity (\Lem~\ref{lem:nishimori} below).
A similar observation was made in~\cite[\Sec~11]{SoftCon}, but there the idea was applied directly to deduce the reconstruction threshold, without considering local weak convergence explicitly.
Here we first establish the local weak convergence result, from which we then derive the reconstruction statement.
As it turns out, this line of argument allows for a shorter, more transparent proof.
The details can be found in \Sec~\ref{Sec_lwc}.

\section{Examples}\label{Sec_examples}

\noindent
In the following we present several examples of well-studied CSPs that satisfy the assumptions of the main results.

\subsection{Random $k$-NAESAT}
In Example~\ref{Ex_NAE} we saw how the random $k$-NAESAT can be stated as a random CSP over $\Omega=\{\pm1\}$ with $P_{k-\mathrm{NAE}}$ being the uniform distribution on the $2^k$ functions $\psi_\tau:\sigma\in\Omega^k\mapsto1-\vecone\{\sigma=\tau\}-\vecone\{\sigma=-\tau\}$ for $\tau\in\Omega^k$.

\begin{lemma}\label{Lemma_NAE}
For any $k\geq3$ the distribution $P_{k-\mathrm{NAE}}$ satisfies {\bf SYM}, {\bf BAL}, {\bf MIN}, {\bf POS} and {\bf UNI}.
\end{lemma}
\begin{proof}
Clearly, $q=2$ and $\xi=1-2^{1-k}$ and it is immediate that $P_{k-\mathrm{NAE}}$ is permutation-invariant.
Further, for either $\omega\in\Omega$ and any $\tau\in\Omega^k$ and any $i\in[k]$ the number of assignments $\sigma\in\Omega^k$ with $\sigma_i=\omega$ with $\psi_\tau(\sigma)=1$ is equal to $2^{k-1}-1$, which shows {\bf SYM}.
For {\bf BAL} we observe that
	\begin{align}\label{eqLemma_NAE}
	\phi(\mu)=\sum_{\sigma\in\Omega^k}\Erw[\PSI(\sigma)]\prod_{i=1}^k\mu(\sigma_i)
		=1-2^{-k}\sum_{\sigma,\tau\in\Omega^k}\bc{\vecone\{\sigma=\tau\}+\vecone\{\sigma=-\tau\}}\prod_{i=1}^k\mu(\sigma_i)=1-2^{1-k}
	\end{align}
is constant.
Further, regarding {\bf MIN}, fix a probability distribution $\rho$ on $\Omega\times\Omega$ such that $\rho(1,1)+\rho(1,-1)=\rho(1,1)+\rho(-1,1)=1/2$
and let $r=\rho(1,1)+\rho(-1,-1)$.
Then by (\ref{eqLemma_NAE}),
	\begin{align*}
	\varphi(\rho)&=\sum_{\sigma, \sigma'\in \Omega^k}\Erw[\PSI(\sigma)\PSI(\sigma')]  \prod_{i=1}^k\rho(\sigma_i,\sigma'_i)
		=1-2^{2-k}+2^{-k}\sum_{\sigma,\sigma',\tau\in\Omega^k}\vecone\{\sigma=\pm\tau,\,\sigma'=\pm\tau\}\prod_{i=1}^k\rho(\sigma_i,\sigma'_i)\\
		&=1-2^{2-k}+2^{1-k}\bc{r^k+(1-r)^k}.
	\end{align*}
This function is convex and attains its minimum at $r=1/2$, corresponding to $\rho=\bar\rho$.
Hence, $P_{k-\mathrm{NAE}}$ satisfies {\bf MIN}.

Moving on to {\bf POS}, fix two distributions $\pi,\pi'\in\cP_*^2(\Omega)$ and an integer $\ell\geq2$.
Then
	\begin{align}\nonumber
	\Erw\brk{\left(1-\sum_{\sigma\in\Omega^k}\PSI(\sigma)\prod_{i=1}^ k\RHO_i(\sigma_i)\right)^\ell}
		&=2^{-k}\sum_{\tau\in\Omega^k}\Erw\brk{\bc{\prod_{i=1}^k\RHO_i(\tau_i)+\prod_{i=1}^k\RHO_i(-\tau_i)}^\ell}\\
		&=2^{\ell-k}\prod_{i=1}^k\Erw\brk{\RHO_i(1)^\ell+\RHO_i(-1)^\ell}=2^{\ell-k}\Erw\brk{\RHO_1(1)^\ell+\RHO_1(-1)^\ell}^k.
			\label{eqLemma_NAE1}
	\end{align}
Analogously,
	\begin{align}		\label{eqLemma_NAE2}
	\Erw\brk{\left(1-\sum_{\sigma\in\Omega^k}\PSI(\sigma)	\prod_{i=1}^k \RHO_i'(\sigma_i)\right)^\ell}&=
		2^{\ell-k}\Erw\brk{\RHO_1'(1)^\ell+\RHO_1'(-1)^\ell}^k,\\
	\Erw\brk{\left(1-\sum_{\tau\in\Omega^k}\PSI(\tau)\RHO_1(\tau_1)\prod_{i=2}^k\RHO_i'(\tau_i)\right)^\ell}&=
		2^{\ell-k}\Erw\brk{\RHO_1(1)^\ell+\RHO_1(-1)^\ell}\Erw\brk{\RHO_1'(1)^\ell+\RHO_1'(-1)^\ell}^{k-1}.
					\label{eqLemma_NAE3}
	\end{align}
Due to the elementary inequality $X^k+(k-1)Y^k-kXY^{k-1}\geq0$ for all $X,Y\geq0$, {\bf POS} follows from (\ref{eqLemma_NAE1})--(\ref{eqLemma_NAE3}).
Finally, condition {\bf UNI} is satisfied for $k\geq3$ because every $k$-clause contains a variable that does not belong to the cycle.
\end{proof}

\Thm~\ref{Thm_NAEcond} follows immediately by combining \Lem~\ref{Lemma_NAE} with \Thm~\ref{Thm_cond}.
Similarly, \Thm~\ref{Thm_NAEoverlap} follows from \Lem~\ref{Lemma_NAE} and \Thm~\ref{Thm_overlap}.

\subsection{Random (hyper)graph colouring}
The random hypergraph colouring problem was defined as a CSP in Example~\ref{Ex_hyp}.
The following lemma shows that the problem satisfies all of our assumptions.
Hence, \Thm~\ref{Thm_cond} yields the exact condensation threshold of this problem for all values of the uniformity parameter $k$ and the number $q$ of colours, except naturally the trivial case $q=k=2$.
Additionally, \Thm~\ref{Thm_SSC} yields the limiting distribution of the number of colourings and \Cor~\ref{Cor_thmContiquity} establishes quiet planting.
{An asymptotically tight quiet planting result was obtained prior to the present work by Ayre and Greenhill~\cite{Ayre3}.
Specifically, for any fixed $k\geq3$ they proved quiet planting for degrees $d<\dc-\eps_k(q)$, where $\eps_k(q)\to0$ in the limit of large $q$.}
Additionally, Ayre and Greenhill obtain the precise rigidity threshold in the random hypergraph problem, a question that we do not deal with in the present work.
Finally, for $k=2$ we obtain \Thm s~\ref{Thm_qcol} and~\ref{Thm_qcol_lwc} from \Sec~\ref{Sec_intro_col}.

\begin{lemma}\label{Lemma_kq}
For any $k\geq2$, $q\geq2$ with $k+q>4$ the random hypergraph colouring problem satisfies {\bf SYM}, {\bf BAL}, {\bf MIN}, {\bf POS} and {\bf UNI}.
\end{lemma}
\begin{proof}
We have $\Omega=[q]$ and $\xi=1-q^{1-k}$ and the single constraint function $\psi_{k,q}$ is invariant under permutations of its coordinates.
Furthermore, if we fix the colour of one vertex in a hyperedge, then there are $q^{k-1}-1$ possible ways to colour the others so that the hyperedge is bichromatic.
Hence, {\bf SYM} is satisfied.
With respect to {\bf BAL} we have
	\begin{align}\label{eqLemma_kq}
	\phi(\mu)=\sum_{\sigma\in\Omega^k}\psi_{k,q}(\sigma)\prod_{i=1}^k\mu(\sigma_i)=1-\sum_{\sigma\in\Omega}\mu(\sigma)^k.
	\end{align}
This function is concave with its maximum attained at the uniform distribution, whence {\bf BAL} follows.
Coming to {\bf MIN}, we fix a probability distribution $\rho$ on $\Omega$ with uniform marginals.
Then (\ref{eqLemma_kq}) implies that
	\begin{align*}
	\varphi(\rho)&=\sum_{\sigma,\tau\in\Omega^k}\psi_{k,q}(\sigma)\psi_{k,q}(\tau)\prod_{i=1}^k\rho(\sigma_i,\tau_i)
		=1-2q^{1-k}+\sum_{\sigma,\tau\in\Omega}\rho(\sigma,\tau)^k.
	\end{align*}
Clearly, the right hand side is a convex function that attains its minimum at the uniform distribution, whence we obtain {\bf MIN}.

To show {\bf POS}, fix two $\pi,\pi'\in\cP_*^2(\Omega)$ and $\ell\geq2$.
Then
	\begin{align}\label{eqLemma_kq1}
	\Erw\brk{\left(1-\sum_{\tau\in\Omega^k}\psi_{k,q}(\tau)\prod_{i=1}^ k\RHO_i(\tau_i)\right)^\ell}&=
		\sum_{\sigma_1,\ldots,\sigma_\ell\in\Omega}\Erw\brk{\prod_{i=1}^k\prod_{j=1}^\ell\RHO_i(\sigma_j)}=
		\sum_{\sigma_1,\ldots,\sigma_\ell\in\Omega}\Erw\brk{\prod_{j=1}^\ell\RHO_1(\sigma_j)}^k.
	\end{align}
Similarly,
	\begin{align}\label{eqLemma_kq2}
	\Erw\brk{\left(1-\sum_{\tau\in\Omega^k}\psi_{k,q}(\tau)	\prod_{i=1}^k \RHO_i'(\tau_i)\right)^\ell}&=
		\sum_{\sigma_1,\ldots,\sigma_\ell\in\Omega}\Erw\brk{\prod_{j=1}^\ell\RHO_1'(\sigma_j)}^k,\\
	\Erw\brk{\left(1-\sum_{\tau\in\Omega^k}\psi_{k,q}(\tau)\RHO_1(\tau_1)\prod_{i=2}^k\RHO_i'(\tau_i)\right)^\ell}
		&=\sum_{\sigma_1,\ldots,\sigma_\ell\in\Omega^k}\Erw\brk{\prod_{j=1}^\ell\RHO_1(\sigma_j)}\Erw\brk{\prod_{j=1}^\ell\RHO_1'(\sigma_j)}
			^{k-1}.\label{eqLemma_kq3}
	\end{align}
Thus, {\bf POS} follows from (\ref{eqLemma_kq1})--(\ref{eqLemma_kq3}) and the elementary inequality $X^k+(k-1)Y^k-kXY^{k-1}\geq 0$  for $X,Y\geq0$.
Finally, it is well known that condition {\bf UNI} is satisfied for all $k,q\geq2$ except $k=q=2$.
\end{proof}

\subsection{Balanced satisfiability}
The following CSP was introduced in \cite{yuval} to derive a lower bound on the satisfiability threshold for random $k$-SAT. Let $\Omega=\{\pm 1\}, k \geq 3$ and let $\lambda = \lambda(k) \in (0,1)$ be the unique root of
\begin{align} \label{def_lambda}
(1-\lambda)(1+\lambda)^{k-1}-1 = 0.
\end{align}
Further, for $\tau \in \Omega^k$ let 
	\begin{align}\label{eqbalSAT}
	\psi_{\tau}(\sigma)=\lambda^{\sum_{j=1}^k \vecone\{\sigma_j = \tau_j\}}\left(1- \prod_{i=1}^k \vecone\{\sigma_i = -\tau_i\}\right)
	\end{align}
and let $P_{k-\mathrm{BAL}}$ be the uniform distribution on these $2^k$ functions.

If we omit the $\lambda$-factor in (\ref{eqbalSAT}), then we recover the classical random $k$-SAT problem.
Indeed, if we identify the Boolean values true and false with $-1$ and $+1$, then
a constraint endowed with the function 
	\begin{align}\label{eqbalSAT2}
	\sigma\in\Omega^k\mapsto1- \prod_{i=1}^k \vecone\{\sigma_i = -\tau_i\}\in\{0,1\}
	\end{align}
represents a $k$-clause in which the $i$th variable appears positively if $\tau_i=1$ and negatively if $\tau_i=-1$.
However, as we explained in \Sec~\ref{Sec_discussion}, the $k$-SAT problem fails to satisfy condition {\bf SYM}, and thus the results of the present paper do not cover this example.
In fact, for the same reason it is not possible to lower bound the satisfiability threshold of random $k$-SAT by applying the second moment method to the number of satisfying assignments (cf.~\cite{nae,yuval}).
Therefore, in order to lower bound the $k$-SAT threshold Achlioptas and Peres~\cite{yuval} introduced the weighted constraint functions (\ref{eqbalSAT}).
The $\lambda$-factor weighs each $\sigma$ according to the number of true literals; more specifically, since $\lambda\in(0,1)$ there is a penalty for `over-satisfying' clauses.
This penalty factor guarantees that {\bf SYM} is satisfied in the resulting weighted CSP, which we call the {\em balanced satisfiability problem}.
Achlioptas and Peres applied the second moment method to the corresponding partition function $Z(\GG(n,m,P_{k-\mathrm{BAL}}))$, which yields a lower bound on the number of satisfying assignments as $\lambda\in(0,1)$.

The following lemma shows that the balanced satisfiability problem meets all our conditions bar {\bf POS}.

\begin{lemma}\label{Lemma_BalSAT}
For any $k \geq 3$ the distribution $P_{k-\mathrm{BAL}}$ satisfies {\bf SYM}, {\bf BAL}, {\bf MIN} and {\bf UNI}.
\end{lemma}

\noindent
\Thm~\ref{Thm_cond} and~(\ref{eqdcdsat}) therefore show that $\dc$ is a lower bound on the satisfiability threshold of the {\em balanced} satisfiability problem.
In fact, because the $\psi_\tau$ are upper bounded by the unweighted (\ref{eqbalSAT2}), $\dc$ also is a lower bound on the actual $k$-SAT threshold for every $k\geq3$.
This lower bound, although difficult to evaluate numerically, improves over the one that can be obtained via the second moment method.
Furthermore, the contiguity result provided by \Thm~\ref{thmContiquity} proves a statistical physics conjecture of Krzakala, M\'ezard and Zdeborov\'a~\cite{quietSAT}.

\begin{proof}[Proof of \Lem~\ref{Lemma_BalSAT}]
For {\bf SYM}, note that for any $\sigma \in \Omega^k$,
\begin{align}\label{eqLemma_BalSat1}
\Erw\brk{\PSI(\sigma)} = 2^{-k} \sum_{\tau \in \Omega^k} \lambda^{\sum_{j=1}^k \vecone\{\sigma_j = \tau_j\}}\left(1- \prod_{i=1}^k \vecone\{\sigma_i = -\tau_i\}\right) = 2^{-k}\left((1+\lambda)^k-1 \right) = 2^{1-k} \frac{\lambda}{1-\lambda},
\end{align}
which directly implies $\xi = 2^{1-k} \frac{\lambda}{1-\lambda}$.
Hence, for all $i \in \{1, \ldots, k\}, \omega \in \Omega$ and $\tau \in \Omega^k$ we have
\begin{align*}
\sum_{\sigma \in \Omega^k} \vecone\{\sigma_i = \omega\} \psi_{\tau}(\sigma)&=  \sum_{\sigma \in \Omega^k} \vecone\{\sigma_i = \omega\}\prod_{j=1}^k  \lambda^{\vecone\{\sigma_j = \tau_j\}} - \sum_{\sigma \in \Omega^k} \prod_{j=1}^k \vecone\{\sigma_j = -\tau_j\}\vecone\{\sigma_i = \omega\}\\
&= (1+\lambda)^{k-1} \lambda^{\vecone\{\tau_i = \omega\}} - \vecone\{\tau_i \not= \omega\} = \frac{\lambda}{1-\lambda} = 2^{k-1} \xi.
\end{align*}
Thus, 
{\bf SYM} is satisfied.
As (\ref{eqLemma_BalSat1}) implies that for any $\mu \in \mathcal{P}(\Omega)$, $\phi_{k-\mathrm{BAL}}(\mu) = \xi,$ {\bf BAL} is also satisfied.

We next turn to condition {\bf MIN}. Fix a probability distribution $\rho$ on $\Omega\times\Omega$ such that $\rho(1,1)+\rho(1,-1)=\rho(1,1)+\rho(-1,1)=1/2$
and let $r=\rho(1,1)+\rho(-1,-1)$. Then $\varphi(\rho)=\varphi(r)=2^{-k} f(r)$ with $f$ from \cite[Equation (8)]{yuval}
and thus
\begin{align}\label{eqLemma_BalSat2}
\varphi(r)
&= \bc{\frac r 2 \bc{1-\lambda}^2+\lambda}^k - 2 \bc{\frac r 2 \bc{1-\lambda}+\frac{\lambda}{2}}^k + \bc{\frac r 2}^k.
\end{align}
Using the definition of $\lambda$, we obtain
\begin{align}\label{eqLemma_BalSat3}
\varphi'(r) &
= \frac{k}{2} \sum_{j=1}^{k-1}\binom{k-1}{j}\bc{\frac{r}{2}-\frac{1}{4}}^j\bc{\frac{1}{4}}^{k-1-j}\bc{1-\bc{\frac{1-\lambda}{1+\lambda}}^j}^2.
\end{align}
It is immediate from (\ref{eqLemma_BalSat3}) that $\varphi'(1/2)=0$, while $\varphi'(r)>0$ for $r \in (1/2, 1]$. For $r<1/2$, all terms corresponding to odd $j$ in (\ref{eqLemma_BalSat3}) are negative, while those corresponding to even $j$ are positive. Let 
\begin{align*}
c_j &= \binom{k-1}{j}\bc{\frac{1}{4}}^{k-1-j}\bc{1-\bc{\frac{1-\lambda}{1+\lambda}}^j}^2&&\mbox{such that}&
\varphi'(r) &= \frac{k}{2} \sum_{j=1}^{k-1} c_j \bc{\frac{r}{2}-\frac{1}{4}}^j.
\end{align*}
The ratio of an odd coefficient $j$ and its even successor $j+1$ works out to be
\begin{align*}
\frac{c_j}{c_{j+1}} = \frac{(j+1)}{4(k-(j+1))}\bc{1-\bc{\frac{1-\lambda}{1+\lambda}}^j}^2\bc{1-\bc{\frac{1-\lambda}{1+\lambda}}^{j+1}}^{-2}, 
\end{align*}
which is increasing in $j$. Thus, $\varphi'(r)$ is negative for all $r\in (0,1/2)$ such that $\frac{1}{4}-\frac{r}{2} < \frac{c_1}{c_2}$,  which is the case for
\begin{align*}
r > \frac{1}{2} - \frac{(1+\lambda)^2}{4(k-2)}=r^\ast_k.
\end{align*}
Unfortunately, only $r^\ast_3\leq 0$ and for $k\geq 4$ we upper bound $\varphi'(r)$ by hand for all $r \in [0,r^\ast_k)$ to show that is is negative. 
By \cite[Lemma~7]{yuval} for all $k\geq 3$ the following bounds on $\lambda$ hold:
\begin{align}\label{eqlambdaBound}
2^{1-k} + k4^{-k} < 1-\lambda < 2^{1-k} + 3k4^{-k}.
\end{align} 
Let $g(r) = (1-\lambda)2^{k-1} -2\lambda^{k-1}+\frac{r^{k-1}}{1-\lambda}.$ Using $\frac r2 (1-\lambda)^2 < (1-\lambda)$ and $\frac{r}{2} (1-\lambda) \geq 0 $ we obtain
\begin{align*}
\varphi'(r) &= \frac k2 \bc{ \bc{ 1-\lambda}^2 \bc{ \frac r2 (1-\lambda)^2 +\lambda}^{k-1} -2(1-\lambda) \bc{\frac r2(1- \lambda) + \frac \lambda 2}^{k-1} + \bc{ \frac r2 }^{k-1} }  \\
&< k2^{-k}(1-\lambda) \bc{ (1-\lambda)2^{k-1} -2\lambda^{k-1}+\frac{r^{k-1}}{1-\lambda} } = k2^{-k}(1-\lambda) g(r).
\end{align*} 
As $g(r)$ is strictly increasing in $r,$ finding a $\bar r_k \in [ r^\ast_k, 1/2)$ such that $g(\bar r_k) \leq 0 $ for all $k\geq 3$ suffices to establish {\bf MIN}. To this regard, for all $k\geq 5$, set \begin{align*}
\bar r_k = \lambda \bc{ (1-\lambda) \bc{ 2- \frac{1+3\cdot 2^{-(k+1)}k}{1-(k-1)2^{1-k}-3k(k-1)4^{-k}}}}^{1/(k-1)}.
\end{align*}
Using \eqref{eqlambdaBound} yields that 
\begin{align*}
g(\bar r_k)&= (1-\lambda)2^{k-1} -\lambda^{k-1}\bc{\frac{1+3\cdot 2^{-(k+1)}k}{1-(k-1)2^{1-k}-3k(k-1)4^{-k}}} \\
&\leq 1- \frac{\lambda^{k-1}}{1-(k-1)2^{1-k}-3k(k-1)4^{-k}} + 3 k 2^{-(k+1)} \bc{ 1- \frac{\lambda^{k-1}}{1-(k-1)2^{1-k}-3k(k-1)4^{-k}}} \leq 0 
\end{align*}
because $\lambda^{k-1} \geq \bc{1-2^{1-k}-3k4^{-k} }^{k-1} \geq 1-(k-1)2^{1-k}-3k(k-1)4^{-k}.$ To verify that $\bar r_k \geq r^\ast_k$ we  have by \eqref{eqlambdaBound} that for all $k\geq 6$ 
\begin{align}\label{eqbarrk1}
 \bc{1+k2^{-(k+1)}} \bc {2-\frac{1+3\cdot 2^{-(k+1)}k}{1-(k-1)2^{1-k}-3k(k-1)4^{-k}}}\geq \frac 1 2 , \; (k-1)\bc{2^{1-k}+3k4^{-k}} \leq 0.2,\; \text{ and } (1+\lambda)^2 \geq 3.8. 
\end{align}
Thus, combining \eqref{eqlambdaBound} and \eqref{eqbarrk1} yields 
\begin{align*}
\bar r_k &\geq \frac \lambda 2 \bc{ \exp \bc{ -\frac{\ln 2 }{k-1}} } \geq \frac 12 \bc{ 1-2^{1-k}-3k4^{-k} } \bc {1-\frac{\ln 2}{k-1}}
\geq \frac 12 - \frac{ 2\ln 2 +4(k-1)\bc{2^{1-k}+3k4^{-k}}}{4(k-1)} 
\geq \frac 12 - \frac{2.2}{4(k-1)} \geq r^\ast_k.
\end{align*}
For $k=5$ one calculates that $r^\ast_5 <0.19$ whereas $\bar r_5 > 0.32$. It remains the case $k=4$ where $r^\ast_4 < 0.083$, but using $g$ as an upper bound turns out to be too crude. We have $0.14 < 1-\lambda < 0.18$ and thus for all $r \in [0,0.1]$ we calculate \begin{align*}
\varphi'(r) \leq 2(0.18^2(0.1\cdot 0.18^2 + 0.18)^3-2\cdot 0.14 \cdot 0.7^3+0.1^3)\leq -0.18.
\end{align*}
Finally, {\bf UNI} is satisfied for $k\geq3$ because every $k$-clause contains a variable that does not belong to the cycle.
\end{proof}

\subsection{Parity-Majority}\label{Sec_Parity-Majority}
We consider the following compound CSP, which has been suggested as a device for constructing one-way functions in cryptography \cite{applebaum2018algebraic}.\footnote{This problem was brought to our attention by Chris Brzuska.} Each constraint function evaluates the XOR of two structurally different parts, namely a parity check and a majority function. Formally, let $\Omega=\{\pm 1\}$ and $k\geq3$ be an odd integer. For $\tau \in \Omega^{2k}$ and a permutation $\theta$ of $[2k]$ define the constraint function $\psi_{\tau, \theta}: \Omega^{2k} \to \{0,1\}$,
\begin{align*}
\psi_{\tau,\theta}(\sigma)&= \vecone\cbc{\prod_{i=1}^k \sigma_{\theta(i)} \tau_{i} = 1}\vecone\cbc{\sum_{i=k+1}^{2k}\sigma_{\theta(i)}\tau_{i} < 0} +  \vecone\cbc{\prod_{i=1}^k \sigma_{\theta(i)} \tau_{i} = -1}\vecone\cbc{\sum_{i=k+1}^{2k}\sigma_{\theta(i)}\tau_{i} > 0}.
\end{align*}
Let $\Psi= \cbc{\psi_{\tau, \theta}: \tau \in \Omega^{2k}, \theta \text{ permutation of } [2k]} $ and $P_{\mathrm{MAJ}}$ be the uniform distribution over $\Psi.$ In words, a sample from $P_{\mathrm{MAJ}}$ is generated by uniformly choosing a vector of `signs' (determining for each position whether the corresponding input is negated) and $k$ positions participating in the parity check and the majority function, respectively.
Now, an assignment $\sigma$ satisfies $\psi_{\tau,\theta}$ if either the parity of the literals $(\sigma_{\theta(i)}\tau_i)_{i=1, \dots , k}$ equals $1$ and the majority of literals $(\sigma_{\theta(i)}\tau_i)_{i= k+1, \dots , 2k}$ votes for $-1$, or vice versa.

\begin{lemma}\label{Lemma_PajMaj}
For any $k \geq 3$, the parity-majority problem satisfies {\bf SYM}, {\bf BAL}, {\bf MIN} and {\bf UNI}.
\end{lemma}

Permuting the inputs of the constraint functions is necessary for the second part of {\bf SYM} to hold. However, as for the rest of the arguments the particular choice of $\theta$ does not make a structural difference, we may work with the identity map and lighten the notation to $\psi_{\tau, \text{id} } = \psi_{\tau}.$ 

\begin{claim}\label{Claim_PajMaj}
For any $k \geq 3$, the parity-majority problem satisfies {\bf SYM}, {\bf BAL} and {\bf UNI}.
\end{claim}

\begin{proof}
Let $\sigma \in \Omega^{2k}$ be arbitrary. The number of $\tau \in \Omega^{2k}$ with $\psi_\tau(\sigma) = 1$ equals $2^{2k-1}$, as for any $(\tau_2, \ldots, \tau_{2k})$ there is exactly one choice of $\tau_1$ which leads to $\psi_\tau(\sigma) = 1$. As each $\tau \in \Omega^{2k}$ is chosen with equal probability, this implies that $\Erw\brk{\PSI(\sigma)}= 1/2$, irrespective of $\sigma \in \Omega^{2k}$. Thus, $\xi= \frac{1}{2}$.

Similarly, for each $\tau \in \Omega^{2k}, i \in [2k], \omega \in \Omega$, the number of $\sigma \in \Omega^{2k}$ with $\psi_{\tau}(\sigma)=1$ and $\sigma_i = \omega$ is $2^{2k-2}$, as $k \geq 3$ and \textit{any} choice of $2k-1$ components which satisfies $\sigma_i = \omega$ and does not fix one of the first $k$ parity components ($\sigma_j$, say) can be extended to a satisfying assignment by choosing this variable $\sigma_j$ in a unique way. Thus,
\begin{align*}
\sum_{\sigma \in \Omega^{2k}} \vecone\{\sigma_i = \omega\} \psi_{\tau}(\sigma) = 2^{2k-2} = 2^{2k-1} \xi
\end{align*}
and due to the construction of $\Psi$ and the uniformity of $P_{\mathrm{MAJ}},$  {\bf SYM} is satisfied.
Further, the above calculation shows that $\phi_{\mathrm{MAJ}}(\mu) = \xi$ for any $\mu \in \mathcal{P}(\Omega)$,
and thus {\bf BAL} is also satisfied as well.
Finally, {\bf UNI} is satisfied because again, $k \geq 3$ and every $k$-clause contains a variable that does not belong to the cycle.
\end{proof}

To prove {\bf MIN} we need to do a bit of calculus.
Fix a probability distribution $\rho$ on $\Omega\times\Omega$ such that $\rho(1,1)+\rho(1,-1)=\rho(1,1)+\rho(-1,1)=1/2$
and let $r=\rho(1,1)+\rho(-1,-1)$.

\begin{claim}\label{Claim_Obs1}
We have
\begin{align}\label{eqLemma_PajMaj1}
\varphi_{\mathrm{MAJ}}(\rho) &= \sum_{\sigma, \sigma' \in \Omega^{2k}}\Erw[\PSI(\sigma)\PSI(\sigma')]  \prod_{i=1}^{2k}\rho(\sigma_i,\sigma'_i) = \sum_{\sigma, \sigma' \in \Omega^{2k}}\psi_{(1, \ldots, 1)}(\sigma)\psi_{(1, \ldots, 1)}(\sigma') \prod_{i=1}^{2k}\rho(\sigma_i,\sigma'_i).
\end{align} 
\end{claim}
\begin{proof}
Indeed, something stronger is true: for any $\tau, \tau' \in \Omega^{2k}$,
\begin{align*}
\sum_{\sigma, \sigma' \in \Omega^{2k}}\psi_\tau(\sigma)\psi_{\tau}(\sigma') \prod_{i=1}^k\rho(\sigma_i,\sigma'_i)
&= \sum_{\sigma, \sigma' \in \Omega^{2k}}\psi_{\tau'}(\sigma)\psi_{\tau'}(\sigma') \prod_{i=1}^k\rho(\tau_i'\tau_i\sigma_i,\tau_i'\tau_i\sigma'_i) =  \sum_{\sigma, \sigma' \in \Omega^{2k}}\psi_{\tau'}(\sigma)\psi_{\tau'}(\sigma')\prod_{i=1}^k\rho(\sigma_i,\sigma'_i),
\end{align*}
and the claim follows by applying the above to $\tau'=(1,\ldots,1)$.
\end{proof}

\noindent
Define 
\begin{align}\label{Def_f}
f:&[0,1] \to \RR,&r\mapsto& 2^{-k}\sum_{\sigma, \sigma' \in \Omega^{k}}\vecone\cbc{\prod_{i=1}^k \sigma_i = 1} \vecone\cbc{\prod_{i=1}^k \sigma_i' = 1}r^{\sum_{i=1}^k \vecone\{\sigma_i = \sigma_i'\}}\bc{1-r}^{k-\sum_{i=1}^k \vecone\{\sigma_i = \sigma_i'\}},\\
\label{Def_g}
g:&[0,1] \to \RR,&r\mapsto& 2^{-k}\sum_{\sigma, \sigma' \in \Omega^{k}}\vecone\cbc{\sum_{i=1}^{k}\sigma_{i}< 0} \vecone\cbc{\sum_{i=1}^{k}\sigma_{i}'< 0}r^{\sum_{i=1}^k \vecone\{\sigma_i = \sigma_i'\}}\bc{1-r}^{k-\sum_{i=1}^k \vecone\{\sigma_i = \sigma_i'\}} .
\end{align}
\begin{claim}\label{Claim_Obs2} With $f$ and $g$ defined in (\ref{Def_f}), (\ref{Def_g}), we have
\begin{align}
\varphi_{\mathrm{MAJ}}(r) = 2 \bc{f(r)g(r) + f(1-r)g(1-r)}.
\end{align}

\end{claim}
\begin{proof}
Using Claim \ref{Claim_Obs2}, we rewrite
\begin{align}\label{eqLemma_PajMaj2}
\varphi_{\mathrm{MAJ}}(r) &= \sum_{\sigma, \sigma' \in \Omega^{2k}}\bc{\vecone\cbc{\prod_{i=1}^k \sigma_i = 1}\vecone\cbc{\sum_{i=k+1}^{2k}\sigma_{i}< 0} +  \vecone\cbc{\prod_{i=1}^k \sigma_i = -1}\vecone\cbc{\sum_{i=k+1}^{2k}\sigma_{i} > 0}} \nonumber \\ 
& \qquad \cdot \bc{\vecone\cbc{\prod_{i=1}^k \sigma_i' = 1}\vecone\cbc{\sum_{i=k+1}^{2k}\sigma'_{i}< 0} +  \vecone\cbc{\prod_{i=1}^k \sigma_i' = -1}\vecone\cbc{\sum_{i=k+1}^{2k}\sigma_{i}' > 0}} \prod_{i=1}^{2k}\rho(\sigma_i,\sigma'_i) \nonumber \\
& = 2 \sum_{\sigma, \sigma' \in \Omega^{2k}}\vecone\cbc{\prod_{i=1}^k \sigma_i = 1}\vecone\cbc{\sum_{i=k+1}^{2k}\sigma_{i}< 0} \vecone\cbc{\prod_{i=1}^k \sigma_i' = 1}\vecone\cbc{\sum_{i=k+1}^{2k}\sigma'_{i}< 0}\bc{\prod_{i=1}^{2k}\rho(\sigma_i,\sigma'_i)+ \prod_{i=1}^{2k}\rho(\sigma_i,-\sigma'_i)} \nonumber \\
&=  2 \bc{f(r)g(r) + f(1-r)g(1-r)},
\end{align}
as desired.
\end{proof}

We can easily write down an explicit expression for the parity component.

\begin{claim}\label{Claim_Obsf}
For all $r\in[0,1]$  we have $f(r) = \frac 14 \bc{1 +\bc{1-2r}^k}.$
\end{claim}

\begin{proof}
For odd $k$ a pair $(\sigma, \sigma')\in \Omega^{2k}$  with exactly $i$ common positions has the same parity, if and only if $i$ is odd, thus 
\begin{align*}
f(r) = 2^{-k} \sum_{\sigma \in \Omega^k }  \vecone\cbc{\prod_{i=1}^k \sigma_i = 1} \sum_{i\in [k]: i \text{ is odd }} {k \choose i} r^i(1-r)^{k-i} = \frac12 \sum_{i\in [k]: i \text{ is odd }} {k \choose i} r^i(1-r)^{k-i}.
\end{align*}
Now, since $1+(1-2r)^k= (r+ (1-r))^k - (1-(1-r))^k = 2 \sum_{i\in [k]: i \text{ is odd }} {k \choose i} r^i(1-r)^{k-i}$ the assertion follows.
\end{proof}

\begin{claim}\label{Claim_Obs3}
For all $r\in[0,1]$  we have $2f(r) + 2f(1-r)=2g(r) + 2g(1-r)=1$.
\end{claim}
\begin{proof}
Let $\bar f(r) = 1/2 -f(r)$ and $\bar g(r) = 1/2 - g(r),$ respectively. Rewriting $\varphi_{\mathrm{MAJ}}(r)$ in a slightly different fashion than before yields
\begin{align*}
\varphi_{\mathrm{MAJ}}(r) &= 2 \sum_{\sigma, \sigma' \in \Omega^{2k}}\left( \vecone\cbc{\prod_{i=1}^k \sigma_i = 1}\vecone\cbc{\sum_{i=k+1}^{2k}\sigma_{i}< 0} \vecone\cbc{\prod_{i=1}^k \sigma_i' = 1}\vecone\cbc{\sum_{i=k+1}^{2k}\sigma'_{i}< 0} \right.  \\
&\qquad \left. + \, \vecone\cbc{\prod_{i=1}^k \sigma_i = 1}\vecone\cbc{\sum_{i=k+1}^{2k}\sigma_{i}> 0} \vecone\cbc{\prod_{i=1}^k \sigma_i' = -1}\vecone\cbc{\sum_{i=k+1}^{2k}\sigma'_{i}< 0} \right)\prod_{i=1}^{2k}\rho(\sigma_i,\sigma'_i)  \\
&=  2 \bc{f(r)g(r) + \bar f(r) \bar g(r)}.
\end{align*}
Thus, combining this with \ref{eqLemma_PajMaj2} we obtain
\begin{align}\label{eqfbargbar}
f(1-r)g(1-r) = \bar f(r) \bar g(r)
\end{align}
Since $k$ is odd, Claim \ref{Claim_Obsf} yields
\begin{align}\label{eqfbar}
 2\bar f(r) &= 1-\frac 12 \left( 1+(1-2r)^k \right)= \frac 12 \left( 1 + (1-2(1-r))^k \right)=  2f(1-r).
\end{align} 
The claim now readily follows from \eqref{eqfbargbar}, \eqref{eqfbar} and the definitions of $\bar f, \bar g.$
\end{proof}

\begin{claim}\label{Claim_Obs4} The function $f$ is strictly increasing on $(0,1)\setminus \{1/2\}$, while $g$ is increasing on $[0,1)$. 
\end{claim}
\begin{proof}
Given Claim \ref{Claim_Obsf} and recalling that $k$ is odd, we see that $f$ is strictly increasing on $[0,1/2)$ and $(1/2,1]$ with a saddle point at $r=1/2$.

The function $g$, which corresponds to the majority part, is more complicated. For $j \in \{1, \ldots, k\}$, let $\mathcal{S}_j$ be the set of pairs of assignments with majority vote $-1$ which agree on exactly the first $j$ components and let $g_j = |\mathcal{S}_j|$ be the number of such pairs.
Then
\begin{align*}
g(r)& = 2^{-k} \sum_{j=1}^k \binom{k}{j} g_j r^j (1-r)^{k-j},&
2^k g'(r)& = g_1(1-r)^{k-1} + \sum_{j=1}^{k-1} \binom{k-1}{j} \bc{g_{j+1}-g_{j}} r^j (1-r)^{k-(j+1)}.
\end{align*}
It is therefore sufficient to show that $g_{j+1}\geq g_{j}$  for all $j \in \{1, \ldots, k-1\}$. To this end, we consider the following  injective map $h$ from $\mathcal{S}_j$ to  $\mathcal{S}_{j+1}.$ Given a pair of solutions $(s^{(1)}, s^{(2)}) \in \mathcal{S}_j$, denote by $(\bar{s}^{(1)}, \bar{s}^{(2)}) \in \Omega^{2k}$ the assignment pair obtained from $(s^{(1)}, s^{(2)})$ by swapping their $(j+1)$st component. There are two possible cases. If $(\bar{s}^{(1)}, \bar{s}^{(2)})$ is not in $\mathcal{S}_j$, we set the $(j+1)$st component of both $s^{(1)}$ and $s^{(2)}$ to $-1$ and obtain a valid solution pair in $\mathcal{S}_{j+1}.$
On the other hand, if both $(s^{(1)}, s^{(2)}),(\bar{s}^{(1)}, \bar{s}^{(2)}) \in \mathcal{S}_j$ then in order for $h$ to be injective, we assign $1$ to the $(j+1)$st component of $\bar{s}^{(1)}, \bar{s}^{(2)}.$ This gives a valid solution in $\mathcal{S}_{j+1},$ because the fact that both $(s^{(1)}, s^{(2)}) $ and $(\bar{s}^{(1)}, \bar{s}^{(2)}) $  are solutions implies that they have a majority vote of $-1$ irrespective of the value of their $(j+1)$st component. 
Thus $g$ is increasing on $[0,1)$.
\end{proof}

\begin{proof}[Proof of \Lem~\ref{Lemma_PajMaj}]
Claim~\ref{Claim_PajMaj} establishes {\bf SYM}, {\bf BAL} and {\bf UNI}.
With respect to \MIN,
Claims \ref{Claim_Obs1}--\ref{Claim_Obs4} show that $\varphi_{\mathrm{MAJ}}$ has a unique minimum at $\frac{1}{2}$, as $\varphi'_{\mathrm{MAJ}}(r)=2\bc{f'(r)(2g(r)-1/2)+g'(r)(2f(r)-1/2)}.$
\end{proof}

\section{Preliminaries and notation}

\subsection{Basics}
Throughout the paper we continue to use the notation introduced in \Sec s~\ref{Sec_results} and~\ref{Sec_Proofs}.
In particular, we write $V_n=\{x_1,\ldots,x_n\}$ for a set of $n$ variable nodes and $F_m=\{a_1,\ldots,a_m\}$ for a set of $m$ constraint nodes.
Further, $\vm(d,n)$ is a random variable with distribution $\Po(dn/k)$ and we just write $\vm(d)$ or $\vm$ if $n$ and/or $d$ are apparent.
Additionally, we let $\cM(d)$ be the set of all sequences $m=m(n)$ such that $|m(n)-dn/k|\leq n^{3/5}$ for all $n$.

We write $\cP(\cX)$ for the set of probability measures on a finite set $\cX$.
We identify $\cP(\cX)$ with the standard simplex in $\RR^{\cX}$, thereby turning $\cP(\cX)$ into a Polish space.
Further, for $\sigma_1,\ldots,\sigma_l:V_n\to\Omega$ let $\rho_{\sigma_1,\ldots,\sigma_l}\in\cP(\Omega^l)$ denote the {\em $l$-wise overlap}, defined by 
	\begin{equation}\label{eqMultiOverlap}
	\rho_{\sigma_1,\ldots,\sigma_l}(\omega_1,\ldots,\omega_l)=|\sigma_1^{-1}(\omega_1)\cap\cdots\cap\sigma_l^{-1}(\omega_l)|/n.
	\end{equation}
We use this notation also in the case $l=1$, and then $\rho_{\sigma_1}\in\cP(\Omega)$ is just the empirical distribution of the spins under $\sigma_1$.
Further, we let $\bar\rho_l$ signify the uniform distribution on $\Omega^l$.
In particular, $\bar\rho_1$ is the uniform distribution on $\Omega$.
We usually omit the index $l$ to ease the notation.
An assignment $\sigma:V_n\to\Omega$ is {\em nearly balanced} if $\TV{\rho_\sigma-\bar\rho}\leq n^{-2/5}$.
In addition, for two spin assignments $\sigma,\tau:V\to\Omega$ we let $\sigma\triangle\tau=\{v\in V:\sigma(v)\neq\tau(v)\}.$

The entropy of a probability distribution $\mu\in\cP(\cX)$ is always denoted by $\cH(\mu)$.
Thus, recalling that $\Lambda(z)=z\ln z$ for $z>0$ and setting $\Lambda(0)=0$, we have $\cH(\mu)=-\sum_{x\in\cX}\Lambda(\mu(x)).$

By default we use $O$-notation to refer to the limit $n\to\infty$.
On the few occasions where we refer to a different limit we say so.

\subsection{Constraint satisfaction problems}
In a few places we will need to look at a slightly more general class of constraint satisfaction problems than introduced in \Sec~\ref{Sec_results1}.
Namely, let $\Omega$ be a finite set.
By extension of \Def~\ref{Def_CSP}, a general {\em constraint satisfaction problem} $G=(V,F,(\partial a)_{a\in F},(\psi_a)_{a\in F})$
consists of a finite set $V$ of variables, a finite set $F$ of constraints, a function $\psi_a:\Omega^{k_a}\to[0,1]$ for some integer $k_a\geq1$, and a tuple $\partial a\in V^{k_a}$.
The difference here is that the $\psi_a$ are not required to belong to a fixed finite set, and that the arities $k_a$ of the constraints can be different.
Similarly as before, we introduce
	\begin{align*}
	\psi_G(\sigma)&=\prod_{a\in F}\psi_{a}(\sigma(\partial_1a,\ldots,\partial_{k_a}a))&(\sigma\in\Omega^V),\\
	Z(G)&=\sum_{\sigma\in\Omega^V}\psi_{G}(\sigma).
	\end{align*}
Further, if $Z(G)>0$ we introduce the {\em Boltzmann distribution} by letting
	$\mu_{G}(\sigma)=\psi_{G}(\sigma)/Z(G)$ for $\sigma\in\Omega^V$.

We will need the following general observation about random CSPs.

\begin{lemma}[\SYM] \label{define_phi}
The function
$
\phi: \RR^{\Omega} \to \RR,$ $\rho \mapsto \sum_{\tau \in \Omega^k}\Erw\brk{\vec \psi(\tau)} \prod_{i=1}^k \rho(\tau_i)
$
satisfies $D\phi(\bar{\rho}) = k \xi \vecone$ and $D^2\phi(\bar{\rho}) = qk(k-1)\xi\Phi$. Moreover,
$\phi$ is strictly positive on the interior of $\cP(\Omega)$.
\end{lemma}
\begin{proof}
The first and second derivatives can be computed along the lines of the proof of Lemma $4.4$ in \cite{SoftCon}.
The positivity bit is immediate as the product $ \prod_{i=1}^k \rho(\tau_i)$ is uniformly bounded below and $\sum_{\tau \in \Omega^k}\Erw\brk{\vec \psi(\tau)} = q^k \xi > 0$.
\end{proof}

\subsection{Boltzmann distributions}
Suppose that $\cX,V$ are finite sets and let $N=|V|$.
For a measure $\mu\in\cP(\cX^V)$, a subset $U\subset V$ and $\sigma\in\cX^U$ we let 
	$\mu_U(\sigma)=\sum_{\tau\in\cX^V}\vecone\{\forall i\in U:\tau_i=\sigma_i\}\mu(\tau)$.
Thus, $\mu_U$ is the marginal distribution that $\mu$ induces on $U$.
Where the reference to $U$ is evident we just write $\mu(\sigma)$.
Additionally, we use the shorthand $\mu_{i_1,\ldots,i_h}$ for $\mu_{\{i_1,\ldots,i_h\}}$ if $i_1,\ldots,i_h\in V$.

If $\mu\in\cP(\cX^{V})$, then  $\SIGMA_{\mu},\TAU_{\mu},\SIGMA_{1,\mu},\SIGMA_{2,\mu},\ldots\in\cX^{V}$ denote mutually independent samples from $\mu$.
Where $\mu$ is apparent from the context we omit the index and just write $\SIGMA,\TAU$, etc.
If $X:(\cX^{V})^l\to\RR$ is a random variable, then we write
	\begin{align*}
	\bck{X}_\mu=\bck{X(\SIGMA_1,\ldots,\SIGMA_l)}_\mu&
		=\sum_{\sigma_1,\ldots,\sigma_l\in\Omega^{V_n}}X(\sigma_1,\ldots,\sigma_l)\prod_{j=1}^l\mu(\sigma_j).
	\end{align*}
Thus, $\bck{X}_\mu$ is the mean of $X$ over independent samples from $\mu$.

If $\mu=\mu_G$ is the Boltzmann distribution induced by a CSP instance $G$, we write $\SIGMA_G$ etc.\ instead of $\SIGMA_{\mu_G}$ and we also write $\bck\nix_G$ rather than $\bck\nix_{\mu_G}$.
We use this notation to distinguish averages over $\mu_G$ from other sources of randomness (e.g., the choice of  the random CSP),
for which we reserve the symbols $\Erw\brk\nix$ and $\pr\brk\nix$.

Let $\eps>0$ and $\ell\geq2$.
Following~\cite{Victor}, we say that the probability measure $\mu\in\cP(\cX^V)$ is {\em $(\eps,\ell)$-symmetric} if
	\begin{align*}
	\sum_{1\leq i_1<\cdots<i_\ell\leq N}\TV{\mu_{i_1,\ldots,i_\ell}-\mu_{i_1}\tensor\cdots\tensor\mu_{i_\ell}}
		&<\eps N^\ell.
	\end{align*}
(The idea is to express that the joint distribution of $\ell$ randomly chosen coordinates is likely to be close to a product distribution.)
Further, an $(\eps,2)$-symmetric measure is simply called {\em $\eps$-symmetric}.
We need the following two results from~\cite{Victor}.

\begin{lemma}[{\cite[Corollaries~2.3 and~2.4]{Victor}}]\label{Lemma_lwise}
For any $\cX\neq\emptyset$, $l\geq3$, $\delta>0$ there is $\eps>0$ such that for all $N>1/\eps$ the following is true.
	\begin{quote}
	If $\mu\in\cP(\cX^I)$ is $\eps$-symmetric, then $\mu$ is $(\delta,l)$-symmetric.
	\end{quote}
\end{lemma}	

\noindent
Let $\mu^{\tensor \ell}\in\cP((\cX^V)^\ell)$ be the distribution $\mu^{\tensor \ell}(\sigma_1,\ldots,\sigma_\ell)=\prod_{j=1}^\ell\mu(\sigma_j)$.

\begin{lemma}[{\cite[\Prop~2.5]{Victor}}]\label{Lemma_prodsym}
For any $\eps>0$, $\ell\geq1$, $\cX\neq\emptyset$ there exists $\delta>0$ such that for all $N>1/\delta$ the following is true.
\begin{quote}
If $\mu\in\cP(\cX^V)$ is $\delta$-symmetric, then $\mu^{\tensor\ell}$ is $\eps$-symmetric.
\end{quote}
\end{lemma}

\noindent
The following lemma relates $\eps$-symmetry and the overlap.

\begin{lemma}\label{Lemma_ovsym}
For any $\eps>0$, $\cX\neq\emptyset$ there exist $\delta>0$, $n_0>0$ such that for all $n>n_0$ and all $\mu\in\cP(\cX^n)$ the following is true.
	\begin{quote}
	If $\bck{\tv{\rho_{\SIGMA,\TAU}-\bar\rho}}_{\mu}<\delta$, then $\mu$ is $\eps$-symmetric and
	$\sum_{i=1}^n\TV{\mu_i-\bar\rho}<\eps n.$
	\end{quote}
Conversely, for any $\eps>0$, $\cX\neq\emptyset$ there exist $\delta>0$, $n_0>0$ such that for all $n>n_0$ and all $\mu\in\cP(\cX^n)$ the following is true.
	\begin{quote}
	If $\mu$ is $\delta$-symmetric and
	$\sum_{i=1}^n\TV{\mu_i-\bar\rho}<\delta n$, then $\bck{\tv{\rho_{\SIGMA,\TAU}-\bar\rho}}_{\mu}<\eps.$
	\end{quote}
\end{lemma}

\noindent
Although \Lem~\ref{Lemma_ovsym} was known (and used) before, we are not aware of a convenient reference.
We therefore prove the lemma in Appendix \ref{Apx_epssymm}.

\begin{corollary}\label{Lemma_multiOverlap}
For any finite set $\cX$, any $\eps>0$ and any $l\geq3$ there exist $\delta=\delta(\cX,\eps,l)$ and $n_0=n_0(\cX,\eps,l)$
	such that for all $n>n_0$ and all $\mu\in\cP(\cX^{V_n})$ the following is true:
	if $\bck{\TV{\rho_{\SIGMA_1,\SIGMA_2}-\bar\rho}}<\delta$, then $\bck{\TV{\rho_{\SIGMA_1,\ldots,\SIGMA_l}-\bar\rho_l}}<\eps$.
\end{corollary}
\begin{proof}
This follows from \Lem s~\ref{Lemma_prodsym} and~\ref{Lemma_ovsym}.
\end{proof}

The following lemma shows that there is a generic (randomised) way of perturbing a given measure in such a way that the outcome is likely $\eps$-symmetric.

\begin{lemma}[{\cite[ Lemma 5.3]{CKPZ}}]\label{Lemma_pinning}
For any $\eps>0$ and any $\cX\neq\emptyset$ there exists a bounded integer random variable $\THETA_\eps\geq0$ such that for all
$\mu\in\cP(\cX^V)$ for sufficiently large $N$ the following is true.
Obtain a random probability measure $\check\MU\in\cP(\cX^V)$ as follows.
	\begin{itemize}
	\item Choose a set $\vU\subset V$ of size $\THETA_\eps$ uniformly at random.
	\item Independently draw  $\check\SIGMA\in\cX^V$ from $\mu$.
	\item Define the (random) probability measure
		\begin{align*}
		\check\MU(\sigma)&=\frac{\mu(\sigma)\vecone\{\forall i\in\vU:\sigma_i=\check\SIGMA_i\}}
			{\mu(\{\tau\in\cX^V:\forall i\in\vU:\tau_i=\check\SIGMA_i\})}&&(\sigma\in\cX^V).
		\end{align*}
	\end{itemize}
Then $\cMU$ is $\eps$-symmetric with probability at least $1-\eps$.
\end{lemma}

\noindent
Thus, in order to obtain an $\eps$-symmetric measure it suffices to peg a bounded number of randomly chosen coordinates to a `reference configuration' $\check\SIGMA$.
Throughout the paper we denote by $\THETA_\eps$ the random variable from \Lem~\ref{Lemma_pinning}.
It will be convenient to use the convention that $\THETA_1=0$.

Finally, we need the following fact.

\begin{lemma}[{\cite[\Lem~4.7]{CKPZ}}]\label{Lemma_nbalanced_CKPZ}
For any $\eps>0$ there is $\delta>0$ such that for all sufficiently large $N$ the following is true.
If $\mu\in\cP(\cX^V)$ satisfies
	$\bck{\TV{\rho_{\SIGMA,\TAU}-\bar\rho}}_{\mu}<\delta$, 
then for all nearly balanced $\tau$ we have $\bck{\TV{\rho_{\SIGMA,\tau}-\bar\rho}}_{\mu}<\eps$.
\end{lemma}

\noindent
Thus, if the overlap of two samples $\SIGMA,\TAU$ is typically close to the uniform overlap $\bar\rho$, then in fact the overlap of a random $\SIGMA$ with an {\em arbitrary} nearly balanced $\tau\in\cX^V$ is likely close to uniform.

\section{Proof strategy}\label{Sec_Proofs}

\noindent
In this section we outline the proofs of the main results presented in \Sec~\ref{Sec_results}, deferring some of the details to the later sections.
Following~\cite{SoftCon} we approach the proofs of the main results by way of analysing the partition function of the planted model $\GG^*$.
This will enable us to construct a suitable random variable to which we can apply the small subgraph conditioning technique, originally developed by Robinson and Wormald~\cite{RobinsonWormald} to count Hamilton cycles in random regular graphs, to prove \Thm~\ref{Thm_SSC}.
The other results then derive from \Thm~\ref{Thm_SSC}.

\subsection{The planted model revisited}\label{Sec_Nishimori}
Before we begin let us get a technical issue out of the way.
The constraints of the random CSP $\GG$ are not quite independent because we require that the hypergraph underlying $\GG$ be simple.
However, in the proofs this slight dependency becomes a nuisance.
We therefore introduce a tweaked model $\G(n,m,P)$ with variables $V_n=\{x_1,\ldots,x_n\}$ whose constraints $a_1,\ldots,a_m$ are chosen independently from the following distribution:
for each $a_i$, the $k$-tuple $\partial a_i\in V_n^k$ is chosen uniformly at random, and the function $\psi_{a_i}\in\Psi$ is chosen from $P$ independently of $\partial a_i$.
Thus, it is possible that the same variable occurs  twice in the constraint $a_i$.

Recall the function $\phi$ which appeared in {\bf BAL} and Lemma \ref{define_phi}. Due to independence of the constraints in $\G(n,m,P)$, we have the identity
\begin{align} \label{psi_and_phi}
\Erw\brk{\psi_{\vec G(n,m)}(\sigma)} = \phi(\rho_\sigma)^m,
\end{align}
which will be used in various places below. 

Naturally, there is a planted model that goes with $\G(n,m,P)$.
Namely, let $\Sigma_n$ be the set of all $\sigma\in\Omega^{V_n}$ such that $\Erw[\psi_{\G(n,m,P)}(\sigma)]>0$.
In other words, $\Sigma_n$ is the set of assignments that may occur as satisfying assignments of some random CSP instance.
By adaptation of~\eqref{eqGGplanted},
for $\sigma\in\Sigma_n$ we define the planted model $\G^*(n,m,P,\sigma)$ by letting
	\begin{align}\label{eq:planted}
	\pr[\G^*(n,m,P,\sigma)=G]&=\frac{\psi_G(\sigma)\pr\brk{\G(n,m,P)=G}}{\Erw[\psi_{\G(n,m,P)}(\sigma)]},
	\end{align}
for any possible CSP instance $G$.
Equivalently, because the $m$ constraints of $\G(n,m,P)$ are drawn independently, \eqref{eq:planted} can be stated as follows:
the constraints $a_1,\ldots,a_m$ are drawn independently from the distribution
	\begin{align}\label{eq:myplanted}
	\pr\brk{\partial a_i=(x_{i_1},\ldots,x_{i_k}),\psi_{a_i}=\psi}&=\frac{\psi(\sigma(x_{i_1}),\ldots,\sigma(x_{i_k}))P(\psi)}
		{\sum_{j_1,\ldots,j_k=1}^n\Erw[\PSI(\sigma(x_{j_1}),\ldots,\sigma(x_{j_k}))]}.
	\end{align}

We continue to denote by $\SIGMA^*=\SIGMA^*_n$ a uniformly random assignment $V_n\to\Omega$.
Suppose we first choose a random assignment $\SIGMA^*\in\Sigma_n$ uniformly and then draw $\G^*(n,m,P,\SIGMA^*)$ from the planted model.
What will be the resulting distribution on CSP instances?
If we assume that all $\psi\in\Psi$ take values in $\{0,1\}$, then this distribution on CSPs should roughly weigh each possible instance $G$ according to its number $Z(G)$ of satisfying assignment; for $G$ has one chance to come up as $\G^*(n,m,P,\sigma)$ for each of its satisfying assignments $\sigma$.
But of course this is only approximately right because the denominator in (\ref{eq:planted}) may depend on $\sigma$.
To correct for this, we introduce a distribution on assignments by letting
	\begin{align}\label{eq:NishimoriS}
	\pr[\hat\SIGMA_{n,m,P}=\sigma]&=\frac{\Erw[\psi_{\G(n,m,P)}(\sigma)]}{\Erw[Z(\G(n,m,P))]}&\mbox{for }\sigma\in\Omega^{V_n}.
	\end{align}
Condition {\bf SYM} guarantees that the denominator $\Erw[Z(\G(n,m,P))]$ is non-zero for all {$n\geq q$}.
It will emerge in due course that the distributions $\SIGMA^*$ and $\hat\SIGMA_{n,m,P}$ are 
 mutually contiguous (see \Lem~\ref{Prop_contig} below).
 From now on we tacitly assume that $n\geq q$.

We claim that the random CSP $\G^*(n,m,P,\hat\SIGMA_{n,m,P})$ is distributed {\em exactly} as the distribution $\G(n,m)$ reweighed according to the partition function.
Formally, let $\hat\G(n,m,P)$ be the random CSP with distribution
	\begin{align}\label{eq:NishimoriG}
	\pr\brk{\hat\G(n,m,P)=G}&=\frac{Z(G)\pr[\G(n,m,P)=G]}{\Erw[Z(\G(n,m,P))]}.
	\end{align}
Then we have the following.

\begin{lemma}\label{lem:nishimori}
For all $\sigma,G$ we have
	\begin{align}\label{eq:nishimori}
	\pr\brk{\hat\SIGMA_{n,m,P}=\sigma}\cdot\pr\brk{\G^*(n,m,P,\sigma)=G}&=\mu_G(\sigma)\cdot\pr\brk{\hat\G(n,m,P)=G}.
	\end{align}
\end{lemma}
\begin{proof}
From the definitions (\ref{eq:planted}), (\ref{eq:NishimoriS}) and (\ref{eq:NishimoriG}) it is immediate that
	\begin{align*}
	\pr\brk{\hat\SIGMA_{n,m,P}=\sigma}\cdot\pr\brk{\G^*(n,m,P,\sigma)=G}&=
		\frac{\Erw[\psi_{\G(n,m,P)}(\sigma)]}{\Erw[Z(\G(n,m,P))]}\cdot\frac{\psi_G(\sigma)\pr\brk{\G(n,m,P)=G}}{\Erw[\psi_{\G(n,m,P)}(\sigma)]}\\
		&=\mu_G(\sigma)\cdot\frac{Z(G)\pr\brk{\G(n,m,P)=G}}{\Erw[Z(\G(n,m,P))]}=
			\mu_G(\sigma)\cdot\pr\brk{\hat\G(n,m,P)=G},
	\end{align*}
as claimed.
\end{proof}

\noindent
Borrowing a term from the statistical physics literature~\cite{LF}, we call \eqref{eq:nishimori} the {\em Nishimori identity}.
This identity will play a fundamental role because it allows us to analyse the partition function by way of the planted model.
The definitions of the models $\hat\G(n,m,P)$, $\hat\SIGMA(n,m,P)$ and \Lem~\ref{lem:nishimori} already appeared in \cite{CKPZ} for the case that all $\psi\in\Psi$ are strictly positive (soft constraints).

To unclutter the notation we will skip the reference to $P$ where possible and just write $\G(n,m)$, $\hat\G(n,m)$, etc.
Further,  recalling that $\vm=\vm_d(n)$ is a random variable with distribution $\Po(dn/k)$, we let introduce  $\hat\G=\hat\G(n,\vm,P)$,
$\G^*=\G^*(n,\vm,P,\SIGMA^*)$ and $\hat\SIGMA=\hat\SIGMA_{n,\vm,P}$.

\subsection{The heat is on}
As mentioned earlier the point of the present work is that we manage to accommodate hard constraints, i.e., functions $\psi$ that may take the value $0$.
A natural first idea might be to deal with this case by softening the constraints so that the results from~\cite{SoftCon} apply and to deal with hard constraints by taking the `softening parameter' to $0$.
Unfortunately, matters are not quite so simple. But we can still get some milage out of this idea.

To be precise,  for a parameter $\beta\geq0$ and a function $\psi:\Omega^k\to[0,1]$ define
	\begin{equation}\label{eqsoft}
	\psi_\beta(\sigma)
		=\eul^{-\beta}+(1-\eul^{-\beta})\psi(\sigma).
	\end{equation}
Thus, $\psi_\beta\geq\eul^{-\beta}$ is a softened version of $\psi$, and we think of $\eul^{-\beta}$ as the softening parameter.
In physics jargon, (\ref{eqsoft}) corresponds to a `positive temperature' variant of the CSP, and $\beta$ might be called the `inverse temperature'.
We let $\Psi_\beta=\{\psi_\beta:\psi\in\Psi\}$.
Further, let $P_\beta$ be the distribution of $\PSI_\beta$ and define
	$$\xi_\beta=q^{-k}\sum_{\sigma\in\Omega^k}\Erw[\PSI_\beta(\sigma)]=\eul^{-\beta}+(1-\eul^{-\beta})\xi.$$
Accordingly, we introduce the symbols $\G_\beta(n,m)=\G(n,m,P_\beta)$, $\hat\G_\beta(n,m)=\hat\G(n,m,P_\beta)$, etc.

In order to apply the results from~\cite{SoftCon} to the `softened' CSP we observe that $P_\beta$ satisfies our main assumptions;
	condition {\bf UNI} is obsolete because all $\psi_\beta$ are strictly positive.

\begin{lemma}\label{Lemma_conditions}
If $P$ satisfies any of the conditions {\bf SYM}, {\bf BAL}, {\bf MIN} and {\bf POS}, then so does $P_\beta$ for any $\beta>0$.
\end{lemma}

\begin{proof}
Assuming that $P$ satisfies {\bf SYM}, we find
\begin{align*}
\sum_{\tau\in\Omega^k}\vecone\{\tau_i=\omega\}\psi_\beta(\tau)&=\eul^{-\beta}q^{k-1} + 
	(1-\eul^{-\beta})\sum_{\tau\in\Omega^k}\vecone\{\tau_i=\omega\}\psi(\tau) =q^{k-1}(\eul^{-\beta}+(1-\eul^{-\beta})\xi)
		=q^{k-1}\xi_\beta.
\end{align*} 
Similarly, if $P$ satisfies {\bf BAL}, then
\begin{align*}
\sum_{\tau\in\Omega^k}\Erw[\PSI_\beta(\tau)]\prod_{i=1}^k \mu(\tau_i)&=\eul^{-\beta}+(1-\eul^{-\beta})\sum_{\tau\in\Omega^k}\Erw[\PSI(\tau)]\prod_{i=1}^k \mu(\tau_i) 
\end{align*} 
is a concave function of $\mu$ that attains its maximum at the uniform distribution.
Moving on to condition {\bf MIN}, we observe that for any $\rho \in \cR(\Omega)$,
	$
	\sum_{\sigma, \tau\in \Omega^k}\Erw[\PSI(\sigma)]  \prod_{i=1}^k\rho(\sigma_i,\tau_i) = q^{-k} \sum_{\sigma\in \Omega^k}	\Erw[\PSI(\sigma)] = \xi.  
	$
Hence,
	\begin{align*}
	\sum_{\sigma, \tau\in \Omega^k}\Erw[\PSI_\beta(\sigma)\PSI_\beta(\tau)]  \prod_{i=1}^k\rho(\sigma_i,\tau_i) = 
		\eul^{-2\beta}+2\eul^{-\beta}(1-\eul^{-\beta})\xi + (1-\eul^{-\beta})^2
			\sum_{\sigma, \tau\in \Omega^k}\Erw[\PSI(\sigma)\PSI(\tau)]  \prod_{i=1}^k\rho(\sigma_i,\tau_i).
	\end{align*}
Clearly, if $P$ satisfies {\bf MIN}, then the uniform distribution on $\Omega\times\Omega$ will be the unique global minimiser $\rho\in\cR(\Omega)$ of the last expression.
Finally, regarding {\bf POS} we calculate
	\begin{align*}
	\Erw\brk{\left(1-\sum_{\tau\in\Omega^k}\PSI_\beta(\tau)\prod_{i=1}^ k\RHO_i(\tau_i)\right)^\ell}&=
		(1-\eul^{-\beta})^\ell\cdot\Erw\brk{\left(1-\sum_{\tau\in\Omega^k}\PSI(\tau)\prod_{i=1}^ k\RHO_i(\tau_i)\right)^\ell},\\
	\Erw\left[\left(1-\sum_{\tau\in\Omega^k}\PSI_\beta(\tau)	\prod_{i=1}^k \RHO_i'(\tau_i)\right)^\ell\right]&=
		(1-\eul^{-\beta})^\ell\cdot\Erw\left[\left(1-\sum_{\tau\in\Omega^k}\PSI(\tau)	\prod_{i=1}^k \RHO_i'(\tau_i)\right)^\ell\right],\\
	\Erw\left[\left(1-\sum_{\tau\in\Omega^k}\PSI_\beta(\tau)\RHO_1(\tau_1)\prod_{i=2}^k\RHO_i'(\tau_i)\right)^\ell\right]&=
		(1-\eul^{-\beta})^\ell\cdot
		\Erw\left[\left(1-\sum_{\tau\in\Omega^k}\PSI(\tau)\RHO_1(\tau_1)\prod_{i=2}^k\RHO_i'(\tau_i)\right)^\ell\right].
	\end{align*}
Hence, if $P$ satisfies {\bf POS}, then so does $P_\beta$.
\end{proof}

Can we use the softened CSP directly to, say, prove \Thm~\ref{Thm_cond} about the condensation phase transition?
Suppose we fix a CSP $G$ with (hard) constraints from $\Psi$ and denote by $Z_\beta(G)$ the partition function of the CSP with soft constraints obtained by replacing each $\psi$ by the corresponding $\psi_\beta$.
Then we verify immediately that $\lim_{\beta\to\infty}Z_\beta(G)=Z(G)$.
In other words, $G$ comes out as the `zero temperature' limit of $G_\beta$.
Consequently, we obtain 
	\begin{align}\nonumber
	\lim_{\beta\to\infty}\Erw\sqrt[n]{Z(\G_\beta)}&=\Erw\sqrt[n]{Z(\G)}&\mbox{and therefore}\\
	\lim_{n\to\infty}\lim_{\beta\to\infty}\Erw\sqrt[n]{Z(\G_\beta)}&=\lim_{n\to\infty}\Erw\sqrt[n]{Z(\G)},\label{eqlimlim1}
	\end{align}
where the second line is conditional on the existence of limits.
Furthermore, using the results from~\cite{SoftCon,CKPZ}, we can determine the condensation threshold of the softened CSP  $Z(\G(n,m,P_\beta))$.
Hence, we should be able to compute
	\begin{align}\label{eqlimlim2}
	\lim_{\beta\to\infty}\lim_{n\to\infty}\Erw\sqrt[n]{Z(\G(n,m,P_\beta))},
	\end{align}
at least for $d<\dc$.

Alas, the order of the limits in (\ref{eqlimlim1}) and (\ref{eqlimlim2}) is reversed. 
Whether the limits commute is arguably one of the most challenging open problems in the theory of random CSPs (cf.\ the discussion in~\cite{Daniel}).
The following result, which constitutes one of the main technical contributions of this paper, proves that in planted models the limits do indeed commute.
Recall the expression $\cB(d,P,\pi)$ from (\ref{eqMyBethe}).

\begin{theorem}[\SYM, \BAL]\label{Thm_planted}
For every $d>0$ we have
	\begin{align}\label{eqThm_planted_SB}
	\limsup_{n\to\infty}\frac 1n\Erw[\ln Z(\hat\G)]&
		\leq\lim_{\beta \to \infty}\sup_{\pi\in\cP_*^2(\Omega)} \cB(d, P_\beta, \pi) =
	 \sup_{\pi\in\cP_*^2(\Omega)}\cB (d, P,\pi).
	\end{align}
Furthermore, if \,\POS\ is satisfied as well, then
	\begin{align}\label{eqThm_planted_SBP}
	\lim_{n\to\infty}\frac 1n \Erw[\ln Z(\hat\G)]&=
	\lim_{\beta\to\infty}\lim_{n\to\infty}\frac 1n\Erw[\ln Z(\hat\G_\beta)]=
		\lim_{\beta \to \infty}\sup_{\pi\in\cP_*^2(\Omega)} \cB(d, P_\beta, \pi)= \sup_{\pi\in\cP_*^2(\Omega)}\cB (d, P,\pi).
	\end{align}
\end{theorem}

\noindent
Apart from being a vital step toward the proofs of the main results, we believe that \Thm~\ref{Thm_planted} may be of independent interest for the study of planted instances of CSPs.
The proof of \Thm~\ref{Thm_planted}, which we carry out in \Sec~\ref{Sec_beta}, combines techniques from~\cite{SoftCon,CKPZ} with new arguments required to deal with hard constraints.

\subsection{The Kesten-Stigum bound}\label{Sec_outline_KS}
We are going to combine \Thm~\ref{Thm_planted} with small subgraph conditioning to prove \Thm~\ref{Thm_SSC}.
To pave the way for this argument we need two preparations.
First, because the eigenvalues of the operator $\Xi$ from (\ref{eqXi}) will come up a lot, we need to 
investigate the spectrum of $\Xi$.
Also recall the matrix $\Phi$ from \eqref{eqPhi} and the space $\cE$ from (\ref{eqSpaceE}).
Additionally, let
	\begin{equation}\label{eqcE'}
	\cE'=\{x\in\RR^q\tensor\RR^q:\scal x{\vecone\tensor\vecone}=0\}\supset\cE.
	\end{equation}
Finally, let us introduce the matrices
	\begin{align*}
	 \Phi_{\psi_\beta} (\omega,\omega') &=q^{1-k}\xi_\beta^{-1}\sum_{\tau\in\Omega^k}\vecone\{\tau_1=\omega,\tau_{2}=\omega'\}
 		\psi_\beta(\tau)\quad\mbox{for }\omega,\omega' \in \Omega,&
	\Phi_\beta&=\Erw[ \Phi_{\PSI_\beta}],&
	\Xi_\beta&=\Erw[ \Phi_{\PSI_\beta}\tensor\Phi_{\PSI_\beta}].
	\end{align*}

\begin{lemma}[\SYM, \BAL]\label{Lemma_Phi}\label{Lemma_Xi}
The matrices $\Phi,\Xi$ enjoy the following properties.
\begin{enumerate}[(i)]
\item $\Phi$ is symmetric and doubly-stochastic and $\max_{x\perp\vecone}\scal{\Phi x}x\leq0$.
\item $\Xi$ is self-adjoint, $\Xi(\vecone\tensor\vecone)=\vecone\tensor\vecone$ and for every $x$ we have
$\Xi (x\tensor\vecone)=(\Phi x)\tensor\vecone$, $\Xi (\vecone\tensor x)=\vecone\tensor(\Phi x)$ and
	\begin{align}\label{eqLemma_Xi}
	\scal{\Xi (x\tensor\vecone)}{x\tensor\vecone}&\leq0,&
		\scal{\Xi (\vecone\tensor x)}{\vecone\tensor x}&\leq0&\mbox{if }x\perp\vecone.
	\end{align}
Furthermore, $\Xi\cE\subset\cE$ and $\Xi\cE'\subset\cE'$.
\end{enumerate}
\end{lemma}
\begin{proof}
\Lem~\ref{Lemma_conditions} shows together with \cite[\Lem s~3.5 and 3.6]{SoftCon} that statements (i) and (ii) hold for $\Phi_\beta$ and $\Xi_\beta$ for any $\beta>0$.
Since $\lim_{\beta\to\infty}\Phi_\beta=\Phi$ and $\lim_{\beta\to\infty}\Xi_\beta=\Xi$, the assertion follows.
\end{proof}

Since the self-adjoint operator $\Xi$ induces an endomorphism of the subsapce $\cE$, we define the multi-set
	\begin{align}
	\Eig[\Xi] = \left\{\lambda \in \RR : \exists x \in \cE\setminus\cbc0:\Xi x = \lambda x\right\}\label{eq:IntroSetEigenvals}
	\end{align}
that contains each eigenvalue according to its geometric multiplicity.
To apply small subgraph conditioning we need the following bound on the spectral radius.

\begin{proposition}[\SYM, \BAL]\label{prop_KS}
We have  $\dc(k-1)\max_{\lambda\in\Eig[\Xi]}|\lambda| \leq 1.$
\end{proposition}

\noindent
\Prop~\ref{prop_KS} is almost immediate from the following statement about the softened version of the random CSP.
By extension of (\ref{eq:dcond}) and (\ref{eqSpaceE}) we define
	\begin{align*}
	\dcond(\beta) &=  \inf\left\{d>0\,:\, \sup_{\pi\in\Pomast} \cB(d,P_\beta,\pi) > \ln q + \frac{d}{k}\ln \xi\right\},&
	\dKS(\beta)&=\bc{(k-1)\max_{x\in\cE:\|x\|=1}\scal{\Xi_\beta x}x}^{-1}.
	\end{align*}
The following lemma paraphrases several results from {\cite[\Sec~5]{SoftCon}}.

\begin{lemma}[\SYM, \BAL] \label{prop:TaylorBdpi}
We have
	$$\dc(\beta)(k-1)\max_{\lambda\in\Eig[\Xi_\beta]}|\lambda| \leq 1\qquad\mbox{for all }\beta>0.$$
Moreover, if $d>0$, $\beta_0>0$ are such that $d>\dc(\beta)$ for all $\beta>\beta_0$, then
 there exists $\eps>0$ such that
	\begin{equation}\label{eqprop:TaylorBdpi}
	\sup_{\pi\in\cP_2^*(\Omega)}\cB(d,P_\beta,\pi)>\ln q+\frac dk\ln \xi_\beta+\eps\qquad\mbox{ for all $\beta>\beta_0$}.
	\end{equation}
\end{lemma}

\begin{proof}[Proof of Proposition \ref{prop_KS}]
Suppose that $d$ is such that $d(k-1)\max_{\lambda\in\Eig[\Xi]}|\lambda| > 1.$
Then for all sufficiently large $\beta$ we have $d(k-1)\max_{\lambda\in\Eig[\Xi_\beta]}|\lambda| > 1$, because $\lim_{\beta\to\infty}\Xi_\beta=\Xi$.
Therefore, \Lem~\ref{prop:TaylorBdpi} yields $\eps>0$ such that \eqref{eqprop:TaylorBdpi} is satisfied for all large enough $\beta$.
Finally, since $\lim_{\beta\to\infty}\xi_\beta=\xi$, \eqref{eqThm_planted_SB} yields
$\sup_{\pi\in\cP_*^2(\Omega)}\cB (d, P,\pi)>\ln q+d\ln \xi/k$.
Hence, $d>\dc$.
\end{proof}

\noindent
 \Thm~\ref{Thm_KS} drops out as an immediate consequence of \Lem~\ref{Lemma_Xi} and \Prop~\ref{prop_KS}.

\begin{proof}[Proof of \Thm~\ref{Thm_KS}]
We have $\max_{x\in\cE:\|x\|=1}\scal{\Xi x}x=\max_{\lambda\in\Eig[\Xi]}|\lambda|$
because \Lem\ \ref{Lemma_Xi} shows that $\Xi$ is self-adjoint.
Therefore, \Thm~\ref{Thm_KS} follows from \Prop~\ref{prop_KS}.
\end{proof}

\subsection{The overlap}\label{Sec_outline_cond}
As a second preparation for the small subgraph conditioning we need to investigate the overlap of two randomly chosen satisfying assignments in the planted model.

\begin{proposition}[\SYM, \BAL, \MIN]\label{prop:belowcond-unif}
\begin{enumerate}
\item  Suppose that $d<\dc$.
There exists a sequence $\zeta=\zeta(n)=o(1)$ 
such that for all $m\in\cM(d)$ we have
	\begin{equation}\label{eq:overlapunif}
	\Erw\bck{\TV{\rho_{\SIGMA_1,\SIGMA_2}-\bar\rho}}_{\hat\G(n,m)}\leq\zeta.
	\end{equation}
\item Conversely, let $D>0$ and assume that {\bf POS} is satisfied as well. If  for all $d<D$ we have
	\begin{equation}\label{eq:overlapunif}
	\Erw\bck{\TV{\rho_{\SIGMA_1,\SIGMA_2}-\bar\rho}}_{\hat\G(n,m)}=o(1).
	\end{equation}
	 then  $\dc\geq D$.
\end{enumerate}
\end{proposition}

\noindent
We defer the proof of \Prop~\ref{prop:belowcond-unif} to \Sec~\ref{sec:ProofPreCond}.
With $\zeta$ from \Prop~\ref{prop:belowcond-unif} we define
	\begin{align}\label{eqcZeps}
	\cZ(G)&=Z(G)\vecone\cbc{\bck{\TV{\rho_{\SIGMA_1,\SIGMA_2}-\bar\rho}}_G\leq\zeta}.
	\end{align}
Thus, $\cZ(G)$ is a truncated version of the partition function $Z(G)$, where an instance $G$ contributes only if its overlaps concentrate about $\bar\rho$.
A similar truncated variable was used in~\cite{SoftCon} in the case of soft constraints and in~\cite{CKPZ} in the special case of the random graph colouring problem.

\begin{corollary}[\SYM, \BAL, \MIN]\label{cor:belowcond-unif}
If $d<\dc$, then $\Erw[\cZ(\G(n,m))]\sim\Erw[Z(\G(n,m))]$ uniformly for all $m\in\cM(d)$.
\end{corollary}
\begin{proof}
This is immediate from the first part of \Prop~\ref{prop:belowcond-unif} and the definition (\ref{eq:NishimoriG}) of $\hat\G(n,m)$.
\end{proof}

\subsection{Small subgraph conditioning}
We are ready to conduct small subgraph conditioning for the random variable $\cZ(\G(n,m))$.
We begin by computing the first and the second moment.

\begin{proposition}[\SYM, \BAL]\label{lem:FirstMoment}
Let $d>0$. Then uniformly for all $m\in\cM(d)$,
	\begin{equation}\label{eq:FirstMoment}
	\Erw[Z(\G(n,m))] \sim \frac{q^{n+\frac12}\xi^m} { \prod_{\lambda\in\eig(\Phi)\setminus\cbc 1}\sqrt{1-d(k-1)\lambda}}. 
	\end{equation}
\end{proposition}

\begin{proposition}[\SYM, \BAL]\label{lem:SecondMoment}
Let $0<d<\dc$. Then {uniformly} for all $m\in\cM(d)$,
	\begin{equation}\label{eq:SecondMoment}
	\Erw[\cZ(\G(n,m))^2] \leq \frac{(1+o(1))q^{2n+1}{\xi^{2m}}}
 	{ \prod_{{\lambda\in\eig'(\Xi)}}\sqrt{1-d(k-1)\lambda}}. 
	\end{equation}
\end{proposition}

\noindent
The expression on the r.h.s.\ of (\ref{eq:FirstMoment}) makes sense because $\eig(\Phi)\setminus\cbc 1\subset\RR_{\leq 0}$ by \Lem~\ref{Lemma_Phi}.
Similarly, \Lem~\ref{Lemma_Phi} and \Prop~\ref{prop_KS} show that in (\ref{eq:SecondMoment}) we only take square roots of positive numbers if $d<\dc$.

The proofs of \Prop s~\ref{lem:FirstMoment} and~\ref{lem:SecondMoment} are virtually identical to the moment calculations performed in~\cite[\Sec~7]{SoftCon}; we included them in Appendix~\ref{Apx_moments}.
Both are fairly straightforward, but the calculation of the second moment hinges on the fact that only CSP instances whose overlap concentrates about $\bar\rho$ contribute to $\cZ(\G(n,m))$.
In fact, the second moment of the original random variable $Z(\G(n,m))$ is generally {\em much} bigger (by an exponential factor).
In effect, we could not possibly base our small subgraph conditioning argument on the plain random variable $Z(\G(n,m))$.
Note, however, that up to $\dc$ the first moments of $\cZ(\G(n,m))$ and $Z(\G(n,m))$ are asymptotically the same by \Cor~\ref{cor:belowcond-unif}.

Combining \Cor~\ref{cor:belowcond-unif} with \Prop s~\ref{lem:FirstMoment} and~\ref{lem:SecondMoment} and applying \Lem~\ref{Lemma_Xi}, we obtain
	\begin{align}\label{eqAtTheEndOfTheDay}
	\frac{\Erw[\cZ(\G(n,m))^2]}{\Erw[\cZ(\G(n,m))]^2}&\sim\prod_{\lambda\in\Eig[\Xi]}\frac1{\sqrt{1-d(k-1)\lambda}}
		&\mbox{if }d<\dc,\ m\in\cM(d).
	\end{align}
Thus, \Prop~\ref{prop_KS} shows that the ratio of the second moment and the square of the first is bounded.
However, the quotient does not generally converge to $1$ as $n\to\infty$.
Following the general small subgraph paradigm as set out in~\cite{Janson,RobinsonWormald}, we will `explain' the remaining variance in terms of the bounded-length cycles of the bipartite graph induced by the random CSP instance.

A similar strategy was used in~\cite{SoftCon} for problems with soft constraints, and we can reuse some of the terminology introduced there.
A {\em signature of order $\ell$} is a family
	$$Y=(\psi_1,s_1,t_1,\psi_2,s_2,t_2,\ldots,\psi_\ell,s_\ell,t_\ell)$$
such that $\psi_1,\ldots,\psi_\ell\subset\Psi$, $s_1,t_1,\ldots,s_\ell,t_\ell\in[k]$ and $s_i\neq t_i$ for all $i\in[\ell]$ and $s_1<t_1$ if $\ell=1$.
Let $\cY_\ell$ be the set of all signatures of order $\ell$, let $\cY_{\leq\ell}=\bigcup_{l\leq\ell}\cY_l$ and let
	$\cY=\bigcup_{\ell\geq1}\cY_\ell$.
For a CSP $G$ with variables $V_n$ and constraints $F_m$ we call a family
$(x_{i_1},a_{h_1},\ldots,x_{i_\ell},a_{h_\ell})$ a \emph{cycle of signature $Y$ in $G$} if
	\begin{description}
	\item[CYC1] $i_1,\ldots,i_\ell\in[n]$ are pairwise distinct and $i_1=\min\{i_1,\ldots,i_\ell\}$,
	\item[CYC2] $h_1,\ldots,h_\ell\in[m]$ are pairwise distinct and $h_{1}<h_\ell$ if $\ell >1$,
	\item[CYC3] $\psi_{a_{h_j}}=\psi_j$ and $\partial_{s_j} a_{h_j}=x_{i_j}$ for all $j\in\{1,\ldots,\ell\}$,
		 $\partial_{t_j} a_{h_j}=x_{i_{j+1}}$ for all $j<\ell$ and $\partial_{t_\ell} a_{h_\ell}=x_{i_{1}}$.
	\end{description}
Thus, the cycle, which, of course, alternates between variables and constraints, begins with the variable with the smallest index ({\bf CYC1}).
From there it is directed toward the constraint with the smaller index ({\bf CYC2}).
Furthermore, the constraint functions along the cycle are the ones prescribed by the signature, the cycle enters the $j$th constraint through its $s_j$th position and leaves through position number $t_j$ ({\bf CYC3}).

Let $C_Y(G)$ be the number of cycles of signature $Y$.
Moreover, for an event $\psi\in\Psi$ and $h,h'\in\{1,\ldots,k\}$ define the $q\times q$ matrix
$\Phi_{\psi,h,h'}$ by letting
	\begin{equation}\label{eqCyclePhi}
	\Phi_{\psi,h,h'} (\omega,\omega') = q^{1-k}\xi^{-1}\sum_{\tau\in\Omega^k}\vecone\{\tau_h=\omega,\tau_{h'}=\omega'\}
 		\psi(\tau)\qquad(\omega,\omega' \in \Omega).\end{equation}
In addition, for a signature $Y=(\psi_1,s_1,t_1,\ldots,\psi_\ell,s_\ell,t_\ell)$ define
\begin{align}\label{eqCycles}
	\kappa_{Y}&=\frac{1}{2\ell}\bcfr dk^\ell\prod_{i=1}^\ell P(\psi_i),&
		\Phi_Y&=\prod_{i=1}^\ell\Phi_{\psi_i,s_i,t_i},&\hat\kappa_Y&=\kappa_Y\Tr(\Phi_Y).
\end{align}
Finally, let $\fS$ be the event that the factor graph $\G(n,m)$ is simple, i.e., that $\partial_1a_i,\ldots,\partial_ka_i$ are pairwise distinct for every $i\in[m]$ and that
$\{\partial_1a_i,\ldots,\partial_ka_i\}\neq\{\partial_1a_j,\ldots,\partial_ka_j\}$ for all $1\leq i<j\leq m$.
The following proposition, whose proof we put off to \Sec~\ref{sec:ProofPreCond}, characterises the joint distributions of the cycle counts in $\G(n,m)$ and $\hat\G(n,m)$.

\begin{proposition}[\SYM, \BAL, \UNI]\label{prop:FirstCondOverFirst}
We have $\kappa_Y>0$ for all $Y\in\cY$ and if $\hat\kappa_Y=0$, then $Y$ has order one and $C_Y(\hat\G(n,m))=0$ deterministically for all $n,m$.
Further, if  $Y_1, Y_2, \ldots Y_l\in\cY$ are pairwise distinct and $y_{1}, \ldots,y_l\geq0$, then for any $d>0$,
	\begin{align}\label{eqCyclePoisson}
	\pr\brk{\forall t\le l:\ C_{Y_t}(\G(n,m))=y_{t}}&\sim\prod_{t=1}^l\pr\brk{\Po(\kappa_{Y_t})=y_t}&
	\end{align}
 uniformly for all $m\in\cM(d)$.
If, in addition, $\hat\kappa_{Y_1},\ldots,\hat\kappa_{Y_l}>0$, then uniformly for all $m\in\cM(d)$,
	\begin{align}\label{eqCyclePoissonHat}
	\pr\brk{\forall t\le l:\ C_{Y_t}(\hat\G(n,m))=y_{t}}&\sim\prod_{t=1}^l\pr\brk{\Po(\hat\kappa_{Y_t})=y_t}.
	\end{align}
Finally,
	\begin{align*}
	\pr\brk{\G(n,m)\in\fS}&\sim\exp\bc{-\frac{d(k-1)}2-\frac{\vecone\{k=2\}d^2}4},&
	\pr\brk{\hat\G(n,m)\in\fS}&		\sim\exp\bc{-\frac{d(k-1)}2\Tr(\Phi)-\frac{\vecone\{k=2\}d^2}4\Tr(\Phi^2)}.
	\end{align*}
\end{proposition}

Based on \Prop s~\ref{lem:FirstMoment}, \ref{lem:SecondMoment} and~\ref{prop:FirstCondOverFirst} the proof of \Thm~\ref{Thm_SSC} is fairly standard.
We will carry out the details in \Sec~\ref{Sec_Thm_SSC}.
Then in \Sec~\ref{Sec_beta} we will prove \Thm~\ref{Thm_planted}.
Several of the proof ingredients will be reused later in \Sec~\ref{sec:ProofPreCond}, where we establish \Prop s~\ref{lem:FirstMoment}, \ref{lem:SecondMoment} and~\ref{prop:FirstCondOverFirst}.
With all the tools in place, in \Sec~\ref{sec:ProofPreCond} we also complete the proofs of \Thm s~\ref{Thm_cond}, \ref{Thm_overlap} and~\ref{thmContiquity}.
Finally, in \Sec~\ref{Sec_lwc} we prove \Thm s~\ref{Thm_lwc} and~\ref{thrm:TreeGraphEquivalence}.

\subsection{Proof of \Thm~\ref{Thm_SSC}}\label{Sec_Thm_SSC}
Fix $0<d<\dc$ and let $m\in\cM(d)$.
Let $\fF_{\ell}=\fF_{\ell}{(n,m)}$ be the $\sigma$-algebra generated by the cycle counts $(C_Y)_{Y\in\cY_{\leq\ell}}$.
The proof of \Thm~\ref{Thm_SSC} follows the original strategy from~\cite{RobinsonWormald} by studying the conditional variance of $\cZ(\G(n,m))$ given $\fF_\ell$.
Janson~\cite{Janson} stated a relatively general results that covers many applications of this strategy, but unfortunately not ours.
The issue is that the number $\vm$ of constraints in the statement of \Thm~\ref{Thm_SSC} is random.
Therefore, we use a combinatorial argument that goes back to~\cite{COW}, which was also used in~\cite{SoftCon}.
The proof here is similar to the one in~\cite{SoftCon}, and actually considerably simpler because in the present paper the set $\Psi$ of constraint functions is finite.
Only the very last part of the proof requires a new argument to accommodate hard constraints.

We aim to prove that $\Erw[\Var(\cZ(\G(n,m))|\fF_{\ell})]$ is much smaller than $\Erw[\cZ(\G(n,m))]$ for large enough $\ell$.
Then we will apply Chebyshev's inequality to $\Erw[\cZ(\G(n,m))|\fF_{\ell}]$ to derive that $\cZ(\G(n,m))\sim\Erw[\cZ(\G(n,m))|\fF_{\ell}]$ \whp\ in the limit of large $\ell,n$.
Formally, we will prove

\begin{lemma}[\SYM, \BAL, \MIN, \UNI]\label{Claim_Tower2}
For any $\eta>0$ there exists $\ell_0(\eta)$ such that for every $\ell>\ell_0(\eta)$ uniformly for all $m\in\cM(d)$,
	$$\lim_{n\to\infty}\pr\brk{|\cZ(\G(n,m))-\Erw[\cZ(\G(n,m))|\fF_{\ell}]|>\eta \Erw[Z(\G(n,m))]}=0.$$
\end{lemma}

We prove \Lem~\ref{Claim_Tower2} by way of the basic identity
	\begin{align}\label{eqTower}
	\Var[\cZ(\G(n,m))]=\Var(\Erw[\cZ(\G(n,m))|\fF_{\ell}])+\Erw[\Var(\cZ(\G(n,m))|\fF_{\ell})].
	\end{align}
Due to (\ref{eqTower}), to prove that $\Erw[\Var(\cZ(\G(n,m))|\fF_{\ell})]$ is small it suffices to show that
	\begin{equation}\label{eqTower'}
	\Var(\Erw[\cZ(\G(n,m))|\fF_{\ell}])=\Erw[\Erw[\cZ(\G(n,m))|\fF_{\ell}]^2]-\Erw[\cZ(\G(n,m))]^2
	\end{equation}
is nearly as big as $\Var[\cZ(\G(n,m))]$, so that will be our first intermediate goal.
We begin with the following little calculation.
Let 
	$\delta_Y=\Tr(\Phi_Y)-1=(\hat\kappa_Y-\kappa_Y)/\kappa_Y.$

\begin{lemma}[\SYM, \BAL]\label{Lemkyhatky2}
We have
	$\displaystyle\sum_{\ell\geq1}\sum_{Y\in\cY_{\leq\ell}}\delta_Y^2\kappa_Y =-\frac12\sum_{\lambda\in\Eig[\Xi]}\ln(1-d(k-1)\lambda).$
\end{lemma}
\begin{proof}
The proof is essentially identical to that of \cite[\Lem~9.1]{SoftCon}.
Let $\PHI_\ell=\prod_{i=1}^\ell\Phi_{\PSI_{i}}$.
Then 
	\begin{align}\label{eq:SumYTrace}
	\sum_{Y\in\cY_{\leq\ell}}\delta_Y^2\kappa_Y=
		\sum_{Y\in\cY_{\leq\ell}} \frac{(\hat\kappa_Y-\kappa_Y)^2}{\kappa_Y}& = \sum_{j=1}^\ell \frac{(d(k-1))^j}{2j} \Erw \brk{ \bc{\Tr\PHI_j-1}^2 }
	\end{align}
Hence, applying (\ref{eqXi}), (\ref{eqPhi}) and \Lem~\ref{Lemma_Phi}, we obtain
	\begin{align}
	\Erw \brk{(\Tr\PHI_j-1)^2}&= \Tr \Erw\brk{\PHI_j\tensor\PHI_j}-2\Tr\Erw\brk{\PHI_j}+1 =	\Tr(\Xi^j)-2\Tr(\Phi^j)+1. \label{eq:TraceEigen}
	\end{align}
Finally, since
	$\Tr(\Xi^j)=\sum_{\lambda\in\eig(\Xi)}\lambda^j=1+2\sum_{\lambda\in\eig(\Phi)\setminus\{1\}}\lambda^j+\sum_{\lambda\in\Eig[\Xi]}			
				\lambda^j=-1+2\Tr(\Phi^j)+\sum_{\lambda\in\Eig[\Xi]}\lambda^j,$
combining \eqref{eq:SumYTrace} and \eqref{eq:TraceEigen} gives
	\begin{align}\label{eqSumYell}
	\sum_{Y\in\cY_{\leq\ell}} \frac{(\hat\kappa_Y-\kappa_Y)^2}{\kappa_Y}&=
			\sum_{j=1}^\ell \sum_{\lambda\in\Eig[\Xi]}\frac{\bc{d(k-1)\lambda}^j}{2j}.
	\end{align}
\Prop~\ref{prop_KS} shows $d(k-1)\max_{\lambda\in\Eig[\Xi]}|\lambda|<1$ for $d<\dc$, and thus we may take $\ell$ to infinity in (\ref{eqSumYell}).
\end{proof}

\begin{lemma}[\SYM, \BAL, \MIN, \UNI]\label{Claim_Tower1}
Suppose that $0<d<\dc$, $\ell>0$.
Then uniformly for
all $m\in\cM(d)$,
	\begin{align*}
	\Erw[\Erw[\cZ(\G(n,m))|\fF_{\ell}]^2]\geq
		(1+o(1))\Erw[Z(\G(n,m))]^2\cdot\exp\sum_{Y\in\cY_{\leq\ell}}\delta_Y^2\kappa_Y.
	\end{align*}
\end{lemma}
\begin{proof}
Fix a number $\alpha>0$, pick $B=B(\alpha,\ell)>0$ large, let 
	$C=\{(c_Y)_{Y\in\cY_{\leq\ell}}\in\ZZ^{\cY_{\leq\ell}}:
	0\leq c_Y\leq B\mbox{ for all }Y\in\cY_{\leq\ell}\}$
and let $\cC=\{(C_Y(\G(n,m)))_{Y\in\cY_{\leq\ell}}\in C\}$.
Then~\eqref{eq:NishimoriG} yield
	\begin{align}\label{eqTower0}
	\frac{\Erw[\vecone\cC\cdot\Erw[Z(\G(n,m))|\fF_{\ell}]^2]}{\Erw[Z(\G(n,m))]^2}&
			=\sum_{c\in\cC}\frac{\pr[\forall Y\in\cY_{\leq\ell}:C_Y(\hat\G(n,m))=c_Y]^2}
					{\pr\brk{\forall Y\in\cY_{\leq\ell}:C_Y(\G(n,m))=c_Y}}.
	\end{align}
\Prop~\ref{prop:FirstCondOverFirst} yields $\pr\brk{\forall Y\in\cY_{\leq\ell}:C_Y(\G(n,m))=c_Y}\sim\prod_Y\pr[\Po(\kappa_Y)=c_Y]$ uniformly for all $c\in\cC$.
Similarly, if $c_Y=0$ for all $Y$ with $\hat\kappa_Y=0$, then \Prop~\ref{prop:FirstCondOverFirst} yields 
	$\pr\brk{\forall Y\in\cY_{\leq\ell}:C_Y(\hat\G(n,m))=c_Y}\sim\prod_Y\pr[\Po(\hat\kappa_Y)=c_Y]$.
By contrast, if $c_Y>0$ for some $Y$ with 	$\hat\kappa_Y=0$, then $\pr\brk{\forall Y\in\cY_{\leq\ell}:C_Y(\hat\G(n,m))=c_Y}=0$.
Thus, (\ref{eqTower0}) gives
	\begin{align}
	\frac{\Erw[\vecone\cC\cdot\Erw[Z(\G(n,m))|\fF_{\ell}]^2]}{\Erw[Z(\G(n,m))]^2}
			&\sim\sum_{c\in\cC}\prod_{Y\in\cY_{\leq\ell}}\frac{\pr\brk{\Po((1+\delta_Y)\kappa_Y)=c_Y}^2}
					{\pr\brk{\Po(\kappa_Y)=c_Y}}\nonumber\\
			&=\exp\brk{-\sum_{Y\in\cY_{\leq\ell}}(1+2\delta_Y)\kappa_Y}
				\sum_{c\in\cC}\prod_{Y\in\cY_{\leq\ell}}\frac{((1+\delta_Y)^2\kappa_Y)^{c_Y}}{c_Y!}.		\label{eqTower00}
	\end{align}
Choosing $B$ sufficiently big, we can ensure that 
	$\sum_{c\in\cC}\prod_{Y\in\cY_{\leq\ell}}[((1+\delta_Y)^2\kappa_Y)^{c_Y}/c_Y!]\geq
		\exp(-\alpha/2+\sum_{Y\in\cY_{\ell}}(1+\delta_Y)^2\kappa_Y)$.
Hence, (\ref{eqTower00}) implies that for large $n$,
	\begin{align}\label{eqeqTower2}
	\frac{\Erw[\vecone\cC\cdot\Erw[Z(\G(n,m))|\fF_{\ell}]^2]}{\Erw[Z(\G(n,m))]^2}&\geq
		\exp\brk{-\alpha+\sum_{Y\in\cY_{\ell}}\delta_Y^2\kappa_Y}.
	\end{align}
Further, as  $0\leq \cZ(\G(n,m))\leq Z(\G(n,m))$,
	\begin{align}\nonumber
	\Erw&\brk{\vecone\cC\cdot(\Erw[Z(\G(n,m))|\fF_{\ell}]^2-\Erw[\cZ(\G(n,m))|\fF_{\ell}]^2)}\\
	&=\Erw\brk{\vecone\cC\cdot(\Erw[Z(\G(n,m))|\fF_{\ell}]+\Erw[\cZ(\G(n,m))|\fF_{\ell}])(\Erw[Z(\G(n,m))|\fF_{\ell}]-\Erw[\cZ(\G(n,m))|\fF_{\ell}])}\nonumber\\
	&\leq	2\|\vecone\cC\cdot\Erw[Z(\G(n,m))|\fF_{\ell}]\|_\infty\Erw\brk{\Erw[Z(\G(n,m))|\fF_{\ell}]-\Erw[\cZ(\G(n,m))|\fF_{\ell}]}.
	\label{eqTower3}
	\end{align}
Since $B$ is (large but) fixed, \Prop~\ref{prop:FirstCondOverFirst} yields 
	$\|\vecone\cC\cdot\Erw[Z(\G(n,m))|\fF_{\ell}]\|_\infty\leq O(\Erw[Z(\G(n,m))])$,
whereas \Cor~\ref{cor:belowcond-unif} shows $\Erw\brk{\Erw[Z(\G(n,m))|\fF_{\ell}]-\Erw[\cZ(\G(n,m))|\fF_{\ell}]}=o(\Erw[Z(\G(n,m))])$.
Plugging these estimates into (\ref{eqTower3}), we get
	$\Erw\brk{\vecone\cC\cdot(\Erw[Z(\G(n,m))|\fF_{\ell}]^2-\Erw[\cZ(\G(n,m))|\fF_{\ell}]^2)}=o(\Erw[Z(\G(n,m))])$.
Thus, the lemma follows from (\ref{eqeqTower2}).
\end{proof}

\begin{proof}[Proof of \Lem~\ref{Claim_Tower2}]
Given $\eta>0$ choose $\alpha=\alpha(\eta)>0$ small enough.
We introduce the auxiliary random variable
	$$X(\G(n,m))=|\cZ(\G(n,m))- \Erw[\cZ(\G(n,m))|\fF_{\ell}]|\cdot\vecone\cbc{
		|\cZ(\G(n,m))- \Erw[\cZ(\G(n,m))|\fF_{\ell}]|>\alpha^{1/3}\Erw[Z(\G(n,m))]}$$
so that
	\begin{align}\label{eqXNotTooSmall}
	X(\G(n,m))<\alpha^{1/3}\Erw[Z(\G(n,m))]&\qquad\Rightarrow\qquad\abs{\cZ(\G(n,m))-\Erw[\cZ(\G(n,m))|\fF_{\ell}]}
		\leq\alpha^{1/3}\Erw[Z(\G(n,m))].
	\end{align}
Combining \eqref{eqAtTheEndOfTheDay},  (\ref{eqTower'}) and Lemmas \ref{Lemkyhatky2} 
and \ref{Claim_Tower1}, we obtain
	$\Erw\brk{\Var[\cZ(\G(n,m))|\fF_{\ell}]}<\alpha{\Erw[Z(\G(n,m))]^2},$
providing $\ell,n$ are large enough.
Therefore, Chebyshev's inequality yields
	\begin{align}\nonumber
	\Erw[X(\G(n,m))]&\leq\alpha^{1/3}\Erw[Z(\G(n,m))]\sum_{j\geq0}2^{j+1}\pr\brk{X(\G(n,m))>2^j\alpha^{1/3}\Erw[Z(\G(n,m))]}\\
		&\leq\alpha^{1/3}\Erw[Z(\G(n,m))]\sum_{j\geq0}2^{j+1}\pr\brk{|\cZ(\G(n,m))- \Erw[\cZ(\G(n,m))|\fF_{\ell}]|>2^j\alpha^{1/3}\Erw[Z(\G(n,m))]}\nonumber\\
		& \leq4\alpha^{-1/3}\Erw[Z(\G(n,m))]\cdot\Erw\brk{\frac{\Var[\cZ(\G(n,m))|\fF_{\ell}]}{\Erw[Z(\G(n,m))]^2}}
			\leq4\alpha^{2/3}\Erw[Z(\G(n,m))].\label{eqErwX}
	\end{align}
Finally, the assertion  follows from (\ref{eqXNotTooSmall}), (\ref{eqErwX}) 
and Markov's inequality.
\end{proof}

\begin{proof}[Proof of \Thm~\ref{Thm_SSC}]
Let $(K_Y)_{Y\geq1}$ be a family of mutually independent Poisson variables with means $\Erw[K_Y]=\kappa_Y$, let
 $(K_j)_{j\geq1}$ be mutually independent Poisson variables with means $\Erw[K_j]=(d(k-1))^j/(2j)$ and let $(\PSI_{h,i,j})_{h,i,j\geq1}$ be a family of samples from $P$, mutually independent and independent of the $K_j$.
We first use an argument from~\cite{Janson} to show that the random variable $\cK$ from \Thm~\ref{Thm_SSC} is well-defined.
Let $\ell\geq1$.
Then (\ref{eqCycles}) shows that the random variables
	\begin{align*}
	\cK_\ell'&=
		\prod_{Y\in\cY_{\ell}}\frac{(\Tr\Phi_Y)^{K_Y}}{\exp(\kappa_Y\delta_Y)},&
	\cK_{\ell}&=\exp\brk{\frac{(d(k-1))^\ell}{2\ell}(1-\Tr(\Phi^\ell))}
		\prod_{i=1}^{K_\ell}\Tr\prod_{j=1}^\ell\Phi_{\PSI_{\ell,i,j}}
	\end{align*}
are identically distributed.
Further, since $\Erw[(\Tr\Phi_Y)^{K_Y}]=\exp(\kappa_Y\delta_Y)$ and because the $K_Y$ are mutually independent, 
we have $\Erw[\cK_\ell]=\Erw[\cK_\ell']=1$.
Therefore, the random variables $\cK_{\leq\ell}=\prod_{l\leq\ell}\cK_l$ form a martingale.
Additionally, since $\Erw[(\Tr\Phi_Y)^{2K_Y}]=\exp(2\kappa_Y\delta_Y+\kappa_Y\delta_Y^2)$, Lemma \ref{Lemkyhatky2}
shows that the martingale is $L_2$-bounded.
Therefore, $(\cK_{\leq\ell})_{\ell\geq1}$ converges to a limit $\cK_*$ almost surely and in $L_2$.
The random variable $\cK$ is obtained from $\cK_*$ by disregarding the factors $\ell=1$ and $\ell=2$ if $k=2$.

As a next step we show that $\cK>0$ almost surely (this is where there is a significant difference between hard constraints and soft ones).
There are two cases to consider.
First, assume that $d<\dc\leq(k-1)^{-1}$.
Then $\sum_{\ell\geq1}\Erw[K_\ell]=O(1)$.
Consequently, for any $\eps>0$ we can find $L>0$ such that $\pr\brk{\forall\ell>L:K_\ell=0}>1-\eps$.
But given that $K_\ell=0$ for all $\ell>L$, $\cK$ is a finite product of positive terms, and thus $\cK$ is positive.
Next, suppose that $\dc>(k-1)^{-1}$.
Then \Lem~\ref{Lemkyhatky2} implies that $\sum_{Y\in\cY}\delta_Y^2<\infty$.
Hence, there exists $\ell_0>1$ such that for all $\ell>\ell_0$ and all $Y\in\cY_\ell$ we have $|\delta_Y|\leq1/2$.
Thus, for $\ell>\ell_0$ we obtain
	\begin{align*}
	\Erw[\cK_\ell^{-1}]&=\prod_{Y\in\cY_{\ell}}\frac{\exp(\kappa_Y\delta_Y)}{(1+\delta_Y)^{K_Y}}
		=\exp\brk{\sum_{Y\in\cY_{\ell}}\frac{\kappa_Y\delta_Y^2}{1+\delta_Y}}\leq
			\exp\brk{4\sum_{Y\in\cY_{\ell}}\kappa_Y\delta_Y^2}.
	\end{align*}
Consequently, \Lem~\ref{Lemkyhatky2} shows that the expected reciprocals $\Erw[\cK_{\leq\ell}^{-1}]$ remain bounded for all $\ell$, whence $\cK>0$ almost surely.

To complete the proof of \Thm~\ref{Thm_SSC}, we recall that  $\Erw|\cZ(\G(n,m))-Z(\G(n,m))|=o(\Erw[Z(\G(n,m))])$ by  \Cor~\ref{cor:belowcond-unif}.
Hence, \Lem~\ref{Claim_Tower2} yields
	\begin{equation}\label{eq_Claim_Tower2}
	\lim_{n\to\infty}\pr\brk{|Z(\G(n,m))-\Erw[Z(\G(n,m))|\fF_{\ell}]|>\eta \Erw[Z(\G(n,m))]}=0\qquad\mbox{for any }\eta>0.
	\end{equation}
Further, by \Prop~\ref{prop:FirstCondOverFirst} the conditional expectation $\Erw[Z(\G(n,m))|\fF_{\ell}]$ is distributed as follows: for any non-negative integer vector $(c_Y)_{Y\in\cY_{\le\ell}}$ such that $c_Y=0$ if $\hat\kappa_Y=0$ we have
	\begin{align}\nonumber
	\frac{\Erw[Z(\G(n,m))|\forall Y\in\cY_{\leq\ell}:C_Y(\G(n,m))=c_Y]}{\Erw[Z(\G(n,m))]}&
		=\frac{\pr[\forall Y\in\cY_{\leq\ell}:C_Y(\hat\G(n,m))=c_Y]}{\pr\brk{\forall Y\in\cY_{\leq\ell}:C_Y(\G(n,m))=c_Y}}
			\qquad\mbox{[due to (\ref{eq:NishimoriG})]}\\
		&\sim\prod_{Y\in\cY_{\leq\ell}}\frac{\pr\brk{\Po(\hat\kappa_Y)=c_Y}}{\pr\brk{\Po(\kappa_Y)=c_Y}}
		=\prod_{Y\in\cY_{\leq\ell}}\frac{(\Tr\Phi_Y)^{c_Y}}{\exp(\hat\kappa_Y-\kappa_Y)},\label{eqTraceFormula}
	\end{align}
while $\Erw[Z(\G(n,m))|\forall Y\in\cY_{\leq\ell}:C_Y(\G(n,m))=c_Y]=0$ if $c_Y>0$ for some signature $Y$ with $\hat\kappa_Y=0$.
Indeed, \Prop~\ref{prop:FirstCondOverFirst} shows that $\hat\kappa_Y=0$ can only occur for signatures of order one, and for such signatures we obtain $\Tr\Phi_Y=0$.
Consequently, the conditional expectation is given by (\ref{eqTraceFormula}) in all cases.
In order words, letting $Q_{\ell}(\G(n,m))=\Erw[Z(\G(n,m))|\fF_{\ell}]/\Erw[Z(\G(n,m))]$,
we conclude that
	\begin{align}
	Q_{\ell}(\G(n,m))
		\quad\stacksign{$n\to\infty$}\to\quad W_{\leq\ell}(\G(n,m))=\prod_{Y\in\cY_{\leq\ell}}\frac{(\Tr\Phi_Y)^{C_Y(\G(n,m))}}{\exp(\hat\kappa_Y-\kappa_Y)}\label{eqProofThm_SSC1}
	\end{align}
in probability.
Therefore, 
\Prop~\ref{prop:FirstCondOverFirst} implies that $Q_{\ell}(\G(n,m))$ converges to $\cK_{\leq\ell}$ in distribution for every $\ell\geq1$.
Since $(\cK_{\leq\ell})_\ell$ converges to $\cK_*$ almost surely and in $L_2$, \eqref{eq_Claim_Tower2} shows that for any bounded continuous $g:\RR\to\RR$,
	\begin{align*}
	\forall\eps>0\exists\ell_0(\eps)\forall\ell\geq\ell_0(\eps)&:
		\limsup_{n\to\infty}\Erw[g(\cK_*)]-\Erw[g(\cK_{\leq\ell})]<\eps,\\
	\forall\eps>0\exists\ell_0'(\eps)\forall\ell\geq\ell_0'(\eps)&:
		\limsup_{n\to\infty}\Erw[g(\cK_{\leq\ell})]-\Erw\brk{g\bc{\frac{Z(\G(n,m))}{\Erw[Z(\G(n,m))]}}}<\eps.
	\end{align*}
Combining these two statements, we conclude that 
$Z(\G(n,m))/\Erw[Z(\G(n,m))]$ converges to $\cK_*$ in distribution.
Further, as  $\pr\brk{\G(n,m)\in\fS\triangle\{C_1(\G(n,m))+\vecone\{k=2\}C_2(\G(n,m))=0\}}=O(1/n),$
we see that $Z(\GG(n,m))/\Erw[Z(\GG(n,m))]$ converges to $\cK$ in distribution.
Finally, plugging in the formula for the first moment from (\ref{eq:FirstMoment}) yields (\ref{eqThm_SSC}).
\end{proof}

\section{The planted model}\label{Sec_beta}
\noindent
In this section we prove \Thm~\ref{Thm_planted}.
Specifically, in \Sec~\ref{Sec_interpolation}--\ref{Sec_Lemma_interpolation} we prove via an adaptation of the interpolation argument from~\cite{CKPZ} that the functional $\cB$ provides a lower bound on $\Erw[\ln Z(\hat\G)]$.
Some of the intermediate steps of this proof will be reused in \Sec~\ref{sec:ProofPreCond}.
Subsequently, in \Sec~\ref{Sec_upperBound} we show how the results from~\cite{SoftCon} can be combined with a limiting argument to derive a matching upper bound on $\Erw[\ln Z(\hat\G)]$.

\subsection{The interpolation method}\label{Sec_interpolation}
We are going to prove the following lower bound on $\Erw[\ln Z(\hat\G)]$.

\begin{proposition}[\SYM, \BAL, \POS]\label{Prop_interpolation}
If $\pi\in\cP_*^2(\Omega)$  is supported on a finite set, then
	$\liminf_{n\to\infty}\frac1n\Erw[\ln Z(\hat\G)]\geq\cB(d,P,\pi).$
\end{proposition}

We prove \Prop~\ref{Prop_interpolation} via the interpolation method.
Specifically, we adapt the interpolation argument developed in~\cite{CKPZ} for the case of soft constraints.
The basic idea is to construct a family of random CSPs, parametrised by $t\in[0,1]$, such that for $t=1$ the model coincides with $\hat\G$, while for $t=0$ the CSP is so simple that we can calculate the partition function easily.
Indeed, we will see that the logarithm of the partition function at $t=0$ is asymptotically equal to $n\cB(d,P,\pi)$ \whp\
To obtain the desired lower bound on $\Erw[\ln Z(\hat\G)]$ we will prove that the mean of the logarithm of the partition function is a monotonically increasing function of $t$.

The intermediate models parametrised by $t\in[0,1]$ comprise a blend of unary and $k$-ary constraints, and $t$ governs the proportion of $k$-ary constraints.
Thus, at $t=0$ all constraints are unary, whereas at $t=1$ there are $k$-ary constraints only.
This interpolating family is best introduced by way of the following generalised random CSP.
Suppose that $\pi\in\cP_*^2(\Omega)$ has a finite support.
Moreover, let $\gamma=(\gamma_v)_{v\in[n]}$ be a sequence of integers, let $\theta\geq0$ be an integer and let $U\subset[n]$.
Define a random CSP $\G(n,m,\gamma,\pi,U)$ with variables $V_n=\{x_1,\ldots,x_n\}$, $k$-ary constraints
	$a_{1},\ldots,a_{m}$ and unary constraints $(b_{i,j})_{i\in[n],j\in[\gamma_i]}$, $(c_i)_{i\in U}$, all chosen mutually independently, as follows.
\begin{description}
\item[INT1] For $i\in[m]$ choose $\partial a_i\in V_n^k$ uniformly and independently pick $\psi_{a_i}\in\Psi$ from the distribution $P$.
\item[INT2] For $i\in[n]$ and $j\in[\gamma_i]$ the constraint $b_{i,j}$ is adjacent to $x_i$ only.
	The random function $\psi_{b_{i,j}}$ is defined as follows:
		with $(\RHO_{i,j,h})_{h\in[k-1]}$ drawn from $\pi$ and $\PSI_{i,j}$ drawn from $P$ mutually independently, let
				$$\psi_{b_{i,j}}(\sigma)=\sum_{\tau_1,\ldots,\tau_{k-1}\in\Omega}\PSI_{i,j}(\tau_1,\ldots,\tau_{k-1},\sigma)
				\prod_{h=1}^{k-1}\RHO_{i,j,h}(\tau_h)\qquad(\sigma\in\Omega).$$
\item[INT3] For $i\in U$ the unary constraint $c_i$ is adjacent to $x_i$ and
 for a uniformly random $\vec\chi_i\in\Omega$ we let 
 			$$\psi_{c_i}(\sigma)=\vecone\{\sigma=\vec\chi_i\}.$$
\end{description}

\noindent
Thus, $a_1,\ldots,a_m$ are chosen just as the constraints of $\G(n,m)$.
Moreover, the unary constraints $b_{i,j}$ acting on $x_i$ come with random constraint functions $\PSI_{i,j}$ whose other $k-1$ inputs are drawn independently from the distributions $\RHO_{i,j,1},\ldots,\RHO_{i,j,k-1}$.
Finally, the constraints $c_i$ simply peg variable $x_i$ to a specific value $\vec\chi_i$.

Like in \Sec~\ref{Sec_Nishimori} we consider several assorted random CSP models, such as a planted version of  $\G(n,m,\gamma,\pi,U)$.
First, given an integer $0\leq\theta\leq n$ let $\vU$ denote a random subset of $[n]$ of size $\theta$ and let $\G(n,m,\gamma,\pi,\theta)=\G(n,m,\gamma,\pi,\vU)$.
Thus, in $\G(n,m,\gamma,\pi,\theta)$ we peg a random set of $\theta$ variables.
Further, let $\hat\G(n,m,\gamma,\pi,\theta)$ be the random CSP obtained by reweighing $\G(n,m,\gamma,\pi,\theta)$ according to its partition function:
for any possible outcome $G$ of $\G(n,m,\gamma,\pi,\theta)$ let
	\begin{equation}\label{eqIntReweighted}
	\pr\brk{\hat\G(n,m,\gamma,\pi,\theta)=G}=\frac{Z(G)\cdot\pr\brk{\G(n,m,\gamma,\pi,\theta)=G}}
		{\Erw[Z(\G(n,m,\gamma,\pi,\theta))]}.
	\end{equation}
The denominator is positive for all $n\geq q$ because of {\bf SYM} and because $\int_{\cP(\Omega)}\rho\dd\pi(\rho)$ is the uniform distribution on $\Omega$.
Further, by extension of (\ref{eq:NishimoriS}) 
we define a distribution on assignments by
letting 
	\begin{align}\label{eqIntHatSigma} 
	\pr\brk{\hat\SIGMA_{n,m,\gamma,\pi,\theta}=\sigma}&=\frac{\Erw[\psi_{\G(n,m,\gamma,\pi,\theta)}(\sigma)]}{\Erw[Z(\G(n,m,\gamma,\pi,\theta))]}
		&\mbox{for any }\sigma\in\Omega^{V_n}.
	\end{align}
Additionally, let $\Sigma(n,m,\gamma,\pi,\theta)\subset\Omega^{V_n}$ be the support of $\hat\SIGMA_{n,m,\gamma,\pi,\theta}$.
Then for $\sigma\in\Sigma(n,m,\gamma,\pi,\theta)$ we define, by extension of (\ref{eq:planted}), a planted random CSP by letting
	\begin{align}\label{eqintG*}
	\pr\brk{\G^*(n,m,\gamma,\pi,\theta,\sigma)=G}&=
		\frac{\psi_G(\sigma)\pr\brk{\G(n,m,\gamma,\pi,\theta)=G}}
			{\Erw[\psi_{\G(n,m,\gamma,\pi,\theta)}(\sigma)]}
	\end{align}
for any possible outcome $G$ of $\G(n,m,\gamma,\pi,\theta)$.

We obtain the interpolating family of random CSPs by choosing the parameters $m,\gamma,\theta$ as appropriate random variables parametrised by $t$.
Specifically, given $d>0$ and $t\in[0,1]$ the number $\vec m_t$ of $k$-ary constraints has distribution $\Po(tdn/k)$.
Moreover, for each $i\in[n]$ let $\vec\gamma_{t,i}$ have distribution $\Po((1-t)d)$ and let $\vec\gamma_t=(\vec\gamma_{t,i})_{i\in[n]}$.
Additionally, let $\THETA_\eps$ be distributed as the random variable from \Lem~\ref{Lemma_pinning}, with the convention that  $\THETA_1=0$.
All of these random variables are mutually independent.
Finally, we let
	\begin{align*}
	\G_{t,\eps}&=\G(n,\vec m_t,\vec\gamma_t,\pi,\THETA_\eps),&\hat\G_{t,\eps}&=\hat\G(n,\vec m_t,\vec\gamma_t,\pi,\THETA_\eps),&
		\hat\SIGMA_{t,\eps}&=\hat\SIGMA_{n,\vm_t,\GAMMA_t,\pi,\THETA_\eps},&
		\G^*_{t,\eps}&=\G^*(n,\vm_t,\GAMMA_t,\pi,\THETA_\eps,\hat\SIGMA_{t,\eps}).
	\end{align*}

The following proposition provides the monotonicity in $t$ that we alluded to above.

\begin{proposition}[\SYM, \BAL, \POS]\label{Lemma_interpolation}
For every $\delta>0$ there is $\eps>0$ such that for large enough $n$ the following holds.
Let
	$$\Gamma_t=\frac{td(k-1)}{k\xi}\Erw\brk{\Lambda\bc{\sum_{\tau\in\Omega^k}\PSI(\tau)\prod_{j=1}^k\RHO_j^{(\pi)}(\tau_j)}}.$$
and define
 	$\phi_{\eps}(t)=\Erw[\ln Z(\hat\G_{t,\eps})]/n+\Gamma_t$ for $t\in[0,1]$.
Then $\frac{\partial}{\partial t}\phi_{\eps}(t)>-\delta$ for all $t\in(0,1)$.
\end{proposition}

\noindent
We observe that the random CSP $\hat\G_{1,\eps}$ at $t=1$ contains $\Po(dn/k)$ $k$-ary constraints as well as a bounded number $\THETA_\eps$ of unary constraints as per {\bf INT3}.
As we will see shortly, this implies that $\Erw[\ln Z(\hat\G_{1,\eps})]\leq\Erw[\ln Z(\hat\G)]$.
Therefore, \Prop~\ref{Lemma_interpolation} shows that for any fixed $\delta>0$ for large enough $n$,
	\begin{equation}\label{eqLemma_interpolation666}
	\frac1n\Erw[\ln Z(\hat\G)]\geq\frac1n\Erw[\ln Z(\hat\G_{0,\eps})]-\Gamma_1-\delta.
	\end{equation}
Further, $\hat\G_{0,\eps}$ consists of unary constraints only, and thus $\Erw[\ln Z(\hat\G_{0,\eps})]$ is going to be easy to compute.
Hence, we will ultimately obtain \Prop~\ref{Prop_interpolation} from (\ref{eqLemma_interpolation666}).

But first we need 
to prove \Prop~\ref{Lemma_interpolation}.
In the special case of soft constraints (i.e., $\psi>0$ for all $\psi\in\Psi$) the above construction of the interpolating family $\hat\G_{t,\eps}$ is identical to the one from~\cite{CKPZ}, and \Prop~\ref{Lemma_interpolation} comes down to~\cite[\Prop~3.25]{CKPZ}.
In fact, the proof of \Prop~\ref{Lemma_interpolation} reuses several of the steps and arguments from~\cite{CKPZ}.
But the presence of hard constraints causes subtle difficulties.
This is because in order to calculate the derivative of $\phi_{\eps}(t)$ we need to investigate the impact of adding a further random constraint to the random CSP instance $\hat\G_{t,\eps}$ on the logarithm of the partition function.
Clearly, in the case of soft constraints the impact of a single constraint is bounded.
But this need not be true in the case of hard constraints, and new arguments are required to deal with this issue.
We will come to this in \Sec~\ref{Sec_Addconstraint}, just after establishing some basic facts about $\hat\G_{t,\eps}$.
Then we will complete the proofs of \Prop s~\ref{Prop_interpolation} and~\ref{Lemma_interpolation} in \Sec~\ref{Sec_Lemma_interpolation}.

\subsection{Groundwork}
Toward the proof of \Prop~\ref{Lemma_interpolation} we need a few basic observations regarding the probability distributions from the previous section.
All of the following results are straightforward adaptations of the corresponding soft constraint versions from~\cite{CKPZ}.
We begin with the following extension of the Nishimori identity.

\begin{lemma}\label{Lemma_NishimoriInt}
For any $G,\sigma$ we have
	$\pr\brk{\hat\SIGMA_{n,m,\gamma,\pi,\theta}=\sigma}\cdot\pr\brk{\G^*(n,m,\gamma,\pi,\theta,\sigma)=G}
		=\mu_G(\sigma)\pr\brk{\hat\G(n,m,\gamma,\pi,\theta)=G}.$
\end{lemma}
\begin{proof}
The proof is essentially identical to that of \Lem~\ref{lem:nishimori}:
\eqref{eqIntReweighted}, \eqref{eqIntHatSigma} and \eqref{eqintG*} yield
	\begin{align*}
	\pr\brk{\hat\SIGMA_{n,m,\gamma,\pi,\theta}=\sigma}&\cdot\pr\brk{\G^*(n,m,\gamma,\pi,\theta,\sigma)=G}=
	\frac{\psi_{G}(\sigma)\pr\brk{\G(n,m,\gamma,\pi,\theta)=G}}{\Erw[Z(\G(n,m,\gamma,\pi,\theta))]}\\
	&=\mu_G(\sigma)\cdot\frac{Z(G)\pr\brk{\G(n,m,\gamma,\pi,\theta)=G}}{\Erw[Z(\G(n,m,\gamma,\pi,\theta))]}
		=\mu_G(\sigma)\pr\brk{\hat\G(n,m,\gamma,\pi,\theta)=G},
	\end{align*}
as desired.
\end{proof}

\noindent
We are going to apply the Nishimori identity as follows.
Suppose that $F(\sigma_0,\ldots,\sigma_\ell)$ is a function of $\ell+1$ assignments.
Then \Lem~\ref{Lemma_NishimoriInt} yields
	\begin{align}\nonumber
	\Erw\bck{F(\SIGMA_0,\ldots,\SIGMA_\ell)}_{\hat\G(n,m,\gamma,\pi,\theta)}
		&=\sum_{\sigma_0\in\Omega^n}\Erw\brk{\mu_{\hat\G(n,m,\gamma,\pi,\theta)}(\sigma_0)
			\bck{F(\sigma_0,\SIGMA_1,\ldots,\SIGMA_\ell)}_{\hat\G(n,m,\gamma,\pi,\theta)}}\\
		&=
		\Erw\bck{F(\hat\SIGMA_{n,m,\gamma,\pi,\theta},\SIGMA_1,\ldots,\SIGMA_\ell)}_{\G^*(n,m,\gamma,\pi,\theta,\hat\SIGMA_{n,m,\gamma,\pi,\theta})}.
	\label{eqLemma_NishimoriInt}\end{align}

\noindent
Of course, in order to put (\ref{eqLemma_NishimoriInt}) to work we need to get a handle on the distribution of  $\hat\SIGMA_{n,m,\gamma,\pi,\theta}$.

\begin{lemma}[\SYM]\label{Lemma_intFirstMoment}
For any assignment $\sigma\in\Omega^{V_n}$ we have
	\begin{equation}\label{eqLemma_intFirstMoment}
	\Erw[\psi_{\G(n,m,\gamma,\pi,\theta)}(\sigma)]=q^{-\theta}\xi^{\sum_{v\in V}\gamma_v} \phi(\rho_\sigma)^m.
	\end{equation}
In particular,  $\hat\SIGMA_{n,m,\gamma,\pi,\theta}$ and $\hat\SIGMA_{n,m,\gamma',\pi,\theta'}$ are identically distributed for all $\gamma,\gamma',\theta,\theta'$.
\end{lemma}
\begin{proof}
The last factor in (\ref{eqLemma_intFirstMoment}) emerges due to (\ref{psi_and_phi}), because the $k$-ary constraints $a_1,\ldots,a_m$ are mutually independent and also the functions $\psi_{a_i}$ and are independent of the neighbourhoods $\partial a_i$ by {\bf INT1}.
Similarly, step {\bf INT2} of the construction gives rise to the middle factor because the $\PSI_{i,j}$ are chosen independently of the $\RHO_{i,j,h}$ and $\Erw[\RHO_{i,j,h}(\tau)]=1/q$ for every $\tau\in\Omega$.
Hence, {\bf SYM} yields $\Erw[\psi_{b_{i,j}}(\sigma)]=\xi$ for every $\sigma\in\Omega$.
Finally, the factor $q^{-\theta}$ results from {\bf INT3}.
\end{proof}

\begin{corollary}[\SYM, \BAL]\label{Cor_intContig}\label{lem:conc_coloring}
Let $D>0$ and $\theta>0$.
Then uniformly for all $m\leq Dn/k$ and all $\gamma$ we have
	\begin{align*}
	\pr\brk{\TV{\rho_{\hat\SIGMA_{n,m,\gamma,\pi, \theta}}-\bar\rho}> n^{-1/2} \ln n }&\leq O(n^{-\ln n}).
	\end{align*}
Furthermore, for any $\eta>0$ uniformly for all $m\leq Dn/k$ and all $\gamma$ we have
	\begin{align*}
	\pr\brk{\TV{\rho_{\hat\SIGMA_{n,m,\gamma,\pi, \theta}}-\bar\rho}>\eta}&\leq \exp(-\Omega(n)).
	\end{align*}
\end{corollary}
\begin{proof}
Let $\sigma\in\Omega^{V_n}$ and recall that $\rho_\sigma\in\cP(\Omega)$ stands for the empirical distribution of $\sigma$.
\Lem~\ref{Lemma_intFirstMoment}, (\ref{eqIntHatSigma}) and (\ref{psi_and_phi}) yield
	\begin{align*}
	\pr\brk{\hat\SIGMA_{n,m,\gamma,\pi,\theta}=\sigma}=\pr\brk{\hat\SIGMA_{n,m,0,\pi,0}=\sigma}&=
		\frac{\phi(\rho_\sigma)^m}{\Erw[Z(\G(n,m))]}
	\end{align*}
and {\bf BAL} provides that the rightmost expression is concave in $\rho_\sigma$ and attains its maximum at $\bar\rho$.
\end{proof}

\noindent
Finally, we introduced the unary constraints from {\bf INT3} in order to obtain the following.

\begin{lemma}\label{Lemma_tpinning}
For any $\eps>0$ there is $n_0>0$ such that for all $d>0$, $t\in[0,1]$ we have $\pr[\mu_{\hat\G_{t,\eps}}\mbox{ is $\eps$-symmetric}]\geq1-\eps$.
\end{lemma}
\begin{proof}
By \Lem~\ref{Lemma_NishimoriInt} the random factor graph $\hat\G_{t,\eps}$ has the same distribution as $\G^*_{t,\eps}$.
Spelling out (\ref{eqintG*}) and using the second part of \Lem~\ref{Lemma_NishimoriInt},
	we see that $\G^*_{t,\eps}$ is obtained by first drawing
$\G^*(n,\vm_t,\GAMMA_t,\pi,0,\hat\SIGMA_{n,\vm_t,\GAMMA_t,\pi,0})$ without pinning and subsequently pinning a random set $\vU$ of $\THETA_\eps$ variables to their planted values $\hat\SIGMA_{n,\vm_t,\GAMMA_t,\pi,0}$.
Applying the first part of \Lem~\ref{Lemma_NishimoriInt}, we see that this experiment is equivalent to first generating a random factor graph
$\hat\G(n,\vm_t,\GAMMA_t,\pi,0)$, then drawing a sample $\SIGMA$ from its Gibbs measure and subsequently pinning the variables in a random set $\vU$
of size $\THETA_\eps$ to the values $\SIGMA(x_i)$, $i\in\vU$.
This last experiment precisely matches the perturbation from \Lem~\ref{Lemma_pinning}, which therefore implies the assertion.
\end{proof}

\subsection{Adding a constraint}\label{Sec_Addconstraint}
As already mentioned in order to prove \Prop~\ref{Lemma_interpolation} we basically need to study the impact  of adding a single constraint to the random CSP $\hat\G_{t,\eps}$.
The following proposition delivers this analysis.
From here on we denote by $\vx_1,\ldots,\vx_k\in V_n$ a family of uniformly random variables, chosen mutually independently and independently of everything else.

\begin{proposition}[\SYM, \BAL]\label{Prop_Deltat}
Let $D>0$ and $\theta>0$.
Uniformly for all $m\le Dn/k$ and all $\gamma$ we have
	\begin{align*}
	\Erw[\ln Z(\hat\G(n,m+1,\gamma,\pi,\theta))]-\Erw[\ln Z(\hat\G(n,m,\gamma,\pi,\theta))]
		=o(1)
			+\xi^{-1}\Erw\brk{\Lambda\bc{\bck{\PSI(\SIGMA(\vx_{1}),\ldots,\SIGMA(\vx_{k}))}_{\hat\G(n,m,\gamma,\pi,\theta)}}}.
	\end{align*}
\end{proposition}

\noindent
\Prop~\ref{Prop_Deltat} extends \cite[\Prop~3.30]{CKPZ} from soft to hard constraints.
To prove the proposition we need the following statement.
The proof, although essentially identical to \cite[\Cor~3.29]{CKPZ}, is included for the sake of completeness.

\begin{lemma}[\SYM, \BAL]\label{Cor_GenCouple}
Let $D>0$ and $\theta>0$.
Uniformly for all $m\leq Dn/k$ and all $\gamma$ the following is true.
There is a coupling of $\hat\SIGMA_{n,m,\gamma,\pi,\theta}$, $\hat\SIGMA_{n,m+1,\gamma,\pi,\theta}$ such that 
	$$\pr\brk{\hat\SIGMA_{n,m,\gamma,\pi,\theta}\neq\hat\SIGMA_{n,m+1,\gamma,\pi,\theta}}=O(n^{-1}\ln^4n)\quad\mbox{and}\quad
		\pr\brk{|\hat\SIGMA_{n,m,\gamma,\pi,\theta}\triangle\hat\SIGMA_{n,m+1,\gamma,\pi,\theta}|>\sqrt n\ln n}=O(n^{-2}).$$
\end{lemma}
\begin{proof}
The second bound is immediate from \Cor~\ref{Cor_intContig}.
To prove the first we bound the total variation distance of $\hat\SIGMA_{n,m,\gamma,\pi,\theta}$, $\hat\SIGMA_{n,m+1,\gamma,\pi,\theta}$.
By \Lem~\ref{Lemma_intFirstMoment} we may assume that $\theta=0$, $\gamma=0$, and thus
$\hat\SIGMA_{n,m,\gamma,\pi,\theta}=\hat\SIGMA_{n,m}$.
Moreover, due to \Cor~\ref{Cor_intContig} we may condition on the event that
	$\tv{\rho_{\hat\SIGMA_{n,m}}-\bar\rho}+\tv{\rho_{\hat\SIGMA_{n,m+1}}-\bar\rho}
		=O(n^{-1/2}\ln n).$

Hence, consider $\sigma$ such that $\tv{\rho_\sigma-\bar\rho}=O(n^{-1/2}\ln n)$.
By {\bf SYM} and {\bf BAL} the first derivative of the function
	$\phi(\rho)=\sum_{\tau\in\Omega^k}\Erw[\PSI(\tau_1,\ldots,\tau_k)]\prod_{j=1}^k\rho(\tau_j)$
vanishes at $\bar\rho$ and thus
	$\phi(\rho)
		=\xi+O(\tv{\rho-\bar\rho}^2).$
Therefore, by \Lem~\ref{Lemma_intFirstMoment} and (\ref{psi_and_phi}),
	\begin{align}\label{eqCor_GenCouple1}
	\frac{\Erw[\psi_{\G(n,m+1)}(\sigma)]}{\Erw[\psi_{\G(n,m)}(\sigma)]}= \phi(\rho_\sigma)
	=\xi+O(\ln^2n/n).
	\end{align}
Summing (\ref{eqCor_GenCouple1}) on $\sigma$ and applying {\bf BAL} a second time, we obtain
	\begin{align}\label{eqCor_GenCouple2}
	\frac{\Erw[Z(\G(n,m+1))]}{\Erw[Z(\G(n,m))]}=\xi+O(\ln^2n/n).
	\end{align}
Plugging (\ref{eqCor_GenCouple1}) and (\ref{eqCor_GenCouple2}) into (\ref{eqIntHatSigma}), we obtain
$\dTV(\hat\SIGMA_{n,m},\hat\SIGMA_{n,m+1})=O(\ln^4n/n)$, as desired.
\end{proof}

The main difference between soft and hard constraints is that the addition of a single hard constraint can potentially have a dramatic impact on the partition function.
In fact, a single hard constraint can diminish $\log Z$ by a linear amount $\Theta(n)$; one of the main technical challenges of this work is to cope with this possibility.
However, the following crucial lemma shows that  in the planted model such `high impact' constraints are unlikely to be present, and that even the collective impact of $n^{3/4}$ constraints is typically sublinear.

\begin{lemma}[\SYM, \BAL]\label{Lemma_Noela}
For any $D>0$ and $\theta>0$ there is $n_0>0$ such that for all $n>n_0$ for all $m\leq Dn/k$ and all $\gamma$ the following is true.
With probability $1-\exp(-n^{0.8})$ the random CSP $\hat\G(n,m,\gamma,\pi,\theta)$ has the following property:
	\begin{quote}
	if $\G'$ is obtained from $\hat\G(n,m,\gamma,\pi,\theta)$ by deleting any set $U$ of at most $n^{3/4}$ constraints, then
	$\ln Z(\G')-\ln Z(\hat\G(n,m,\gamma,\pi,\theta)) \leq n^{0.9}.$
	\end{quote}
\end{lemma}
\begin{proof}
The proof is based on a double-counting argument; throughout we assume that $n$ is sufficiently large.
Let $\check\G=\hat\G(n,m,\gamma,\pi,\theta)$ for brevity.
For a specific set $U$ let $\cE(U)$ be the event that the factor graph $\G'$ satisfies $\ln Z(\G')-\ln Z(\check\G)> n^{0.9}$.
Also let $\cE$ be the union of all the events $\cE(U)$ with $|U|\leq n^{3/4}$.
Additionally, let $\cI$ be the event that $\check\G$ has at least $n^{0.9}$ isolated variables and that no variable has degree larger than $n^{0.8}$.
A standard balls-into-bins calculation shows that 
	\begin{align}\label{eqLemma_Noela}
	\pr\brk{\cI}\geq1-\exp(-2n^{0.8}).
	\end{align}
Hence, it suffices to bound
	\begin{align}\label{eqLemma_Noela0}
	\pr\brk{\cE\cap\cI}&\leq\sum_{U:|U|\leq n^{3/4}}\pr\brk{\cE(U)\cap\cI}.
	\end{align}

Let $\tilde U$ be the set of all $k$-ary constraints in $U$ together with all the $k$-ary constraints that are adjacent to the unique variable appearing in a unary constraint from $U$.
For a graph $\check\G\in\cE(U)\cap\cI$ obtain $\tilde\G$ by rewiring the constraints $a\in\tilde U$ such that in $\tilde\G$ each is adjacent to distinct variables that are isolated in $\check\G$.
There is a sufficient supply of isolated variables because $\check\G\in\cI$ and $|U|\leq 2n^{0.8}$ on $\cI$;
	the isolated vertices used and the rewiring protocol are deterministic given $\check\G$.
We claim that almost surely (with respect to choice of $\check\G$),
	\begin{align}\label{eqLemma_Noela1}
	Z(\G')&\leq\exp(O(|U|))Z(\tilde\G).
	\end{align}
Indeed, each of the $k$-ary constraint of $\tilde\G$ not present in $\G'$ is connected with $k$ variables that do not have any further neighbours.
Hence, {\bf SYM} ensures that the addition of these constraints decreases the partition function by no more than a factor of $\xi^{|U|}$.
Further, \Lem~\ref{Lemma_NishimoriInt} ensures that each of the unary constraints contained in $U$ is satisfiable (because we can think of $\check\G$ as being obtained by first planting an assignment and then adding constraints that are satisfied under this assignment).
Consequently, being connected in $\tilde\G$ exclusively to variables that are adjacent to unary constraints only, the unary constraints in $U$ have an impact of no more than $\exp(O(|U|))$ on the partition function. Thus, (\ref{eqLemma_Noela1}) follows.

Let $\check\cG=\cE(U)\cap\cI$ and let $\tilde\cG$ be the set of all possible graphs $\tilde\G$ that can be obtained from some $\check\G\in\check\cG$.
We define a bipartite graph structure on the (finite) sets $\check\cG,\tilde\cG$ by connecting each $\check\G$ with the corresponding $\tilde\G$.
Thus, each vertex in $\check\cG$ has degree one, but those in $\tilde\cG$ may have many neighbours.
However, we claim that for every $\tilde G\in\tilde\cG$,
	\begin{align}\label{eqLemma_Noela2}
	\sum_{G\in\partial\tilde G}\pr\brk{\G(n,m,\gamma,\pi,\theta)=G}&\leq\exp(n^{0.81})\pr\brk{\G(n,m,\gamma,\pi,\theta)=\tilde G}.
	\end{align}
Indeed, the only difference between $\tilde G$ and any neighbour $G\in\partial\tilde G$ is that $O(n^{0.8})$ constraints have different neighbours.
Since in $\G(n,m,\gamma,\pi,\theta)$ the neighbours are chosen uniformly, we obtain (\ref{eqLemma_Noela2}) from double counting.

To complete the proof recall that $Z(\check\G)\leq\exp(n^{0.9})Z(\G')$ for $\check\G\in\cE(U)$.
Hence, (\ref{eqLemma_Noela1}) implies 
	$$Z(\tilde\G)\geq Z(\check\G)\exp(n^{0.9}/2).$$
Therefore, (\ref{eqLemma_Noela2}) gives
	\begin{align}\nonumber
	\pr\brk{\check\G\in\cE(U)\cap\cI}&\leq
		\frac{\Erw\brk{Z(\G(n,m,\gamma,\pi,\theta))\vecone\{\G(n,m,\gamma,\pi,\theta)\in\cE(U)\cap\cI\}}}{\Erw\brk{Z(\G(n,m,\gamma,\pi,\theta))}}\\
		&\leq\exp(n^{0.81})\cdot\frac{\sum_{G\in\cG}Z(G)\pr\brk{\G(n,m,\gamma,\pi,\theta)=G}}{\sum_{G\in\cG}Z(\tilde G)\pr\brk{\G(n,m,\gamma,\pi,\theta)=G}}\leq\exp(-n^{0.9}/3).
			\label{eqLemma_Noela3}
	\end{align}
Finally, the assertion follows from (\ref{eqLemma_Noela}), (\ref{eqLemma_Noela0}) and (\ref{eqLemma_Noela3}).
\end{proof}

Equipped with \Lem~\ref{Lemma_Noela} we can complete the proof of \Prop~\ref{Prop_Deltat}.
The argument is similar to the proof of \cite[\Lem~3.32]{CKPZ}, except that we have apply \Lem~\ref{Lemma_Noela} to make to coupling  work.

\begin{proof}[Proof of \Prop~\ref{Prop_Deltat}]
The proof is by way of a coupling of  $\hat\G(n,m,\gamma,\pi,\theta)$, $\hat\G(n,m+1,\gamma,\pi,\theta)$.
By \Lem~\ref{Cor_GenCouple} we can couple $\hat\SIGMA'=\hat\SIGMA_{n,m,\vec\gamma_t,\vec m_t,\THETA_\eps},\hat\SIGMA''=\hat\SIGMA_{n,m,\vec\gamma_t,\vec m_t+1,\THETA_\eps}$
such that 
	\begin{align}\label{eqLemma_Deltat2}
	\pr\brk{\hat\SIGMA'=\hat\SIGMA''}&=1-O(\ln^4n/n),&
		\pr\brk{|\hat\SIGMA'\triangle\hat\SIGMA''|>\sqrt n\ln n}&=O(n^{-2}).
	\end{align}
Further, given $\hat\SIGMA',\hat\SIGMA''$ 
we couple $\G'\stacksign{$\mathrm d$}=\G^*(n,m,\gamma,\pi,\theta,\hat\SIGMA')$ and 
	$\G''\stacksign{$\mathrm d$}=\G^*(n,m+1,\gamma,\pi,\theta,\hat\SIGMA'')$ as follows.
	\begin{description}
	\item[Case 1: $\hat\SIGMA'=\hat\SIGMA''$] 
		we couple so that all of their unary constraints as well as the first $m$ $k$-ary constraints coincide.
		Additionally, $\G''$ contains a single further random $k$-ary constraint $\vec a$ drawn according to (\ref{eqintG*}) with respect to the planted assignment
		$\hat\SIGMA'$.
		Hence,
			\begin{align}\label{eqLemma_Deltat2_case1}
			\Erw\brk{\ln\frac{Z(\G'')}{Z(\G')}\bigg|\hat\SIGMA'=\hat\SIGMA''}
				&=\Erw\brk{\ln\bck{\psi_{\vec a}(\SIGMA_{\G'})}_{\G'}\bigg|\hat\SIGMA'=\hat\SIGMA''}.
			\end{align}
	\item[Case 2: $|\hat\SIGMA'\triangle\hat\SIGMA''|\leq\sqrt n\ln n$]
		the definition (\ref{eqintG*}) of the planted distribution ensures that with probability $1-O(n^{-2})$ the total number $X$ of constraints in either
		$\G^*(n,m,\gamma,\pi,\theta,\hat\SIGMA')$ or 
		$\G^*(n,m+1,\gamma,\pi,\theta,\hat\SIGMA'')$ that are adjacent to a variable in $\hat\SIGMA'\triangle\hat\SIGMA''$ is bounded by $n^{2/3}$.
		Hence, we couple the first $m$ constraints such that $\G',\G''$ coincide on those constraints that are not adjacent to any variable in 
		$\hat\SIGMA'\triangle\hat\SIGMA''$, while the constraints that are adjacent to a variable in $\hat\SIGMA'\triangle\hat\SIGMA''$ are chosen independently.
		Additionally, $\G''$ contains an $(m+1)$st constraint that is chosen independently of the rest.
		Thus, \Lem~\ref{Lemma_Noela} implies that
			\begin{align}\label{eqLemma_Deltat2_case2}
			\Erw\brk{\ln\frac{Z(\G'')}{Z(\G')}\bigg||\hat\SIGMA'\triangle\hat\SIGMA''|\leq\sqrt n\ln n}
				&=O(n^{0.9}).
			\end{align}
	\item[Case 3: $|\hat\SIGMA'\triangle\hat\SIGMA''|>\sqrt n\ln n$]
			in this case we choose $\G'$, $\G''$ independently from their respective distributions.
			The deterministic bound $|\ln Z(\G')|,\,|\ln Z(\G'')|\leq O(n+\vm)$ implies
			\begin{align}\label{eqLemma_Deltat2_case3}
			\Erw\brk{\ln\frac{Z(\G'')}{Z(\G')}\bigg||\hat\SIGMA'\triangle\hat\SIGMA''|>\sqrt n\ln n}&=O(n).
			\end{align}
	\end{description}
Combining (\ref{eqLemma_Deltat2})--(\ref{eqLemma_Deltat2_case3}), Lemma \ref{Lemma_NishimoriInt} and applying \Lem~\ref{Lemma_Noela} a second time, we conclude that
	\begin{align}		\label{eqLemma_Deltat5}
	\Erw\left[\ln \frac{Z(\hat\G(n,m+1,\gamma,\pi,\theta)) }{ Z(\hat\G(n,m,\gamma,\pi,\theta))} \right]&=\Erw\brk{\ln\bck{\psi_{\vec a}(\SIGMA_{\G'})}_{\G'}\bigg|\hat\SIGMA'=\hat\SIGMA''}+o(1)=
		\Erw\brk{\ln\bck{\psi_{\vec a}(\SIGMA_{\G'})}_{\G'}}+o(1).
	\end{align}

To compute $\Erw\brk{\ln\bck{\psi_{\vec a}(\SIGMA_{\G'})}_{\G'}}$ we write $\SIGMA,\SIGMA_1,\SIGMA_2,\ldots$ for independent samples from $\mu_{\G'}$.
Spelling out the definition of $\vec a$, we find
	\begin{align*}
	\Erw\brk{\ln\bck{\psi_{\vec a}(\SIGMA_{\G'})}_{\G'}}&=
		\frac{\Erw\brk{\PSI(\hat\SIGMA'(\vy_1),\ldots,\hat\SIGMA'(\vy_k))
				\ln\bck{\PSI(\SIGMA(\vy_1),\ldots,\SIGMA(\vy_k))}_{\G'}}}
			{\Erw\brk{\PSI(\hat\SIGMA'(\vy_1),\ldots,\hat\SIGMA'(\vy_k))}}.
	\end{align*}
Since by \Cor~\ref{Cor_intContig} the empirical distribution $\rho_{\hat\SIGMA'}$ is asymptotically uniform with very high probability,
the denominator equals $\xi+o(1)$ with probability $1-O(n^{-2})$.
Thus,
	\begin{align}\label{eqMont}
	\Erw\brk{\ln\bck{\psi_{\vec a}(\SIGMA_{\G'})}_{\G'}}&=(\xi^{-1}+o(1))
		\Erw\brk{\PSI(\hat\SIGMA'(\vy_1),\ldots,\hat\SIGMA'(\vy_k))
				\ln\bck{\PSI(\SIGMA(\vy_1),\ldots,\SIGMA(\vy_k))}_{\G'}}.
	\end{align}
To proceed we are going to use the series expansion of the logarithm. 
This expansion applies because we may assume that the argument of the logarithm lies in the interval $(0,1]$.
Indeed, to obtain the lower bound we simply observe that $\PSI(\hat\SIGMA'(\vy_1),\ldots,\hat\SIGMA'(\vy_k))>0$ because otherwise the pre-factor vanishes,
and $\mu_{\G'}(\hat\SIGMA')>0$ by \Lem~\ref{Lemma_NishimoriInt}.
Moreover, $\PSI\leq1$ by the definition of the constraint functions.
Thus, expanding the logarithm we obtain
	\begin{align*}
	\Erw\brk{\ln\bck{\psi_{\vec a}(\SIGMA_{\G'})}_{\G'}}&=-(\xi^{-1}+o(1))\Erw\brk{\sum_{\ell\geq1}\frac{\PSI(\hat\SIGMA'(\vy_1),\ldots,\hat\SIGMA'(\vy_k))}\ell
		\bck{1-\PSI(\SIGMA(\vy_1),\ldots,\SIGMA(\vy_k))}_{\G'}^\ell}.
	\end{align*}
Because the constraint functions are upper bounded by $1$, the sum is absolutely convergent.
Hence, we may swap the sum and the expectation and obtain
	\begin{align*}
	\Erw\brk{\ln\bck{\psi_{\vec a}(\SIGMA_{\G'})}_{\G'}}&=-(\xi^{-1}+o(1))\sum_{\ell\geq1}\frac1\ell
		\Erw\brk{\PSI(\hat\SIGMA'(\vy_1),\ldots,\hat\SIGMA'(\vy_k))
		\bck{1-\PSI(\SIGMA(\vy_1),\ldots,\SIGMA(\vy_k))}_{\G'}^\ell}.
	\end{align*}
Further, applying \Lem~\ref{Lemma_NishimoriInt} once more we obtain
	\begin{align}
	\Erw&\brk{\ln\bck{\psi_{\vec a}(\SIGMA_{\G'})}_{\G'}}=-(\xi^{-1}+o(1))\sum_{\ell\geq1}\frac1\ell
		\Erw\brk{\bc{1-\bc{1-\PSI(\hat\SIGMA'(\vy_1),\ldots,\hat\SIGMA'(\vy_k))}}
		\bck{\prod_{h=1}^\ell1-\PSI(\SIGMA_h(\vy_1),\ldots,\SIGMA_h(\vy_k))}_{\G'}}\nonumber\\
		&=-(\xi^{-1}+o(1))\sum_{\ell\geq1}\frac1\ell
		\Erw\brk{\bck{\prod_{h=1}^\ell1-\PSI(\SIGMA_h(\vy_1),\ldots,\SIGMA_h(\vy_k))}_{\G'}}
		-\frac1\ell\Erw\brk{\bck{\prod_{h=1}^{\ell+1}1-\PSI(\SIGMA_h(\vy_1),\ldots,\SIGMA_h(\vy_k))}_{\G'}}\nonumber\\
		&=-(\xi^{-1}+o(1))\brk{1-
			\Erw\brk{\bck{\PSI(\SIGMA(\vy_1),\ldots,\SIGMA(\vy_k))}_{\G'}}
			-\sum_{\ell\geq2}\frac1{\ell(\ell-1)}\Erw\brk{\bck{1-\PSI(\SIGMA(\vy_1),\ldots,\SIGMA(\vy_k))}_{\G'}^\ell}}.
				\label{eqLemma_Deltat66}
	\end{align}
Due to the series expansion $\Lambda(1-x)+x=\sum_{\ell\geq2}\frac{x^\ell}{\ell(\ell-1)}$, the assertion follows by combining \eqref{eqLemma_Deltat5} and (\ref{eqLemma_Deltat66}).
\end{proof}

\subsection{The lower bound}
	\label{Sec_Lemma_interpolation}
Thanks to \Prop~\ref{Prop_Deltat} the rest of the proof of \Prop~\ref{Lemma_interpolation} is almost identical to the proof of \cite[\Prop~3.30]{CKPZ}, except that we have to pay a bit of attention to some convergence issues.
Write $\bck{\nix}_{t,\eps}$ for the expectation with respect to the Gibbs measure of $\hat\G_{t,\eps}$.
Unless specified otherwise $\SIGMA_1,\SIGMA_2,\ldots$ denote independent samples from  $\mu_{\hat\G_{t,\eps}}$.
Moreover, we write $\PSI$ for a sample from $P$ and $\vec x_1,\ldots,\vec x_k\in V_n$ for independently and uniformly chosen variable.
Toward the proof of \Prop~\ref{Lemma_interpolation} we establish the following formula for the derivative of $\phi_\eps(t)=(\Erw[\ln Z(\hat\G_{t,\eps})]+\Gamma_t)/n$.

\begin{lemma}[\SYM, \BAL]\label{Prop_deriv}
Let $\RHO_1,\ldots,\RHO_k$ be chosen from $\pi$, mutually independently and independently of everything else.
Set
	\begin{align*}
	\Xi_{t,\ell}=\Erw\brk{\bck{1-\PSI(\SIGMA(\vec y_1),\ldots,\SIGMA(\vec y_k))}_{t,\eps}^\ell
			-k\bck{1-\sum_{\tau\in\Omega^{k-1}}\PSI(\tau,\SIGMA(\vec y_1))\prod_{j<k}\RHO_j(\tau_j)}_{t,\eps}^\ell
				+(k-1)\bc{1-\sum_{\tau\in\Omega^k}\PSI(\tau)\prod_{j=1}^k\RHO_j(\tau_j)}^\ell}.
				\end{align*}
Then
	$$\frac{\partial}{\partial t}\phi_\eps(t)=o(1)+\frac d{k\xi}\sum_{\ell\geq2}\frac{\Xi_{t,\ell}}{\ell(\ell-1)}\qquad\mbox{ uniformly for all $\eps,t\in(0,1)$}.$$
\end{lemma}

\noindent
Let
	\begin{align*}
	\Delta_t&=\Erw\brk{\ln Z(\hat\G_{t,\eps}(\vec m_t+1,\vec\gamma_t))}
		-\Erw\brk{\ln Z(\hat\G_{t,\eps}(\vec m_t,\vec\gamma_t))},&
	\Delta_t'&=\Erw\brk{\ln Z(\hat\G_{t,\eps}(\vec m_t,\vec\gamma_t+\vecone_{\vx_1}))}-
			\Erw\brk{\ln Z(\hat\G_{t,\eps}(\vec m_t,\vec\gamma_t))}
	\end{align*}
Thus, $\Delta_t$ is the expected impact of adding one more $k$-ary constraint to $\hat\G_{t,\eps}$.
Similarly, $\Delta_t'$ quantifies the average impact of adding a unary constraint as per {\bf INT2}.
The following standard calculation shows how $\frac\partial{\partial t}\Erw[\ln Z(\hat\G_{t,\eps})]$ can be expressed in terms of $\Delta_t,\Delta_t'$.

\begin{claim}[\SYM, \BAL]\label{Lemma_PoissonDeriv}
We have $\frac1n\frac{\partial}{\partial t}\Erw[\ln Z(\hat\G_{t,\eps})]=\frac dk\Delta_t-d\Delta_t'.$
\end{claim}
\begin{proof}
Let $P_\lambda(j)=\lambda^j\exp(-\lambda)/j!$.
By the construction, the parameter $t$ only affects the distribution of random variables $\vec m_t,\vec\gamma_t$.
Indeed,
	\begin{align}\label{eqLemma_PoissonDeriv}
	\Erw[\ln Z(\hat\G_{t,\eps})]&=\sum_{m,\gamma}\Erw[\ln Z(\hat\G_{t,\eps})|\vec m_t=m,\vec\gamma_t=\gamma]
		P_{tdn/k}(m)\prod_{x\in V}P_{(1-t)d}(\gamma_x).
	\end{align}
Since the derivatives of the Poisson densities come out as
	\begin{align*}
	\frac\partial{\partial t}P_{tdn/k}(m)&
		=\frac{dn}k\brk{\vecone\{m\geq1\}P_{tdn/k}(m-1)-P_{tdn/k}(m)},\\
	\frac\partial{\partial t}P_{(1-t)d}(\gamma_v)&	=-d\brk{\vecone\{\gamma_v\geq1\}P_{(1-t)d}(\gamma_v-1)-P_{(1-t)d}(\gamma_v)},
	\end{align*}
the assertion follows from (\ref{eqLemma_PoissonDeriv}) and the product rule.
\end{proof}

\noindent
We proceed to calculate $\Delta_t,\Delta_t'$.

\begin{claim}[\SYM, \BAL]\label{Lemma_Deltat}
We have
	$
	\Delta_t=o(1)-\frac{1-\xi}\xi+
		\sum_{\ell\geq2}\frac{1}{\ell(\ell-1)\xi}\Erw\brk{\bck{1-\PSI(\SIGMA(\vy_1),\ldots,\SIGMA(\vy_k))}_{t,\eps}^\ell}.$
\end{claim}
\begin{proof}
Recalling the expansion $\Lambda(1-x)+x=\sum_{\ell\geq2}\frac{x^\ell}{\ell(\ell-1)}$, we obtain from
\Prop~\ref{Prop_Deltat} that
	\begin{align*}
	\Delta_t&=o(1)+\xi^{-1}\Erw\brk{\Lambda\bc{\bck{\PSI(\SIGMA(\vy_{1}),\ldots,\SIGMA(\vy_{k}))}_{t,\eps}}}\\
		&=o(1)-\xi^{-1}\bc{1-\Erw\brk{\bck{\PSI(\SIGMA(\vy_{1}),\ldots,\SIGMA(\vy_{k}))}_{t,\eps}}}
			+\sum_{\ell\geq2}\frac1{\ell(\ell-1)\xi}\Erw\brk{\bck{1-\PSI(\SIGMA(\vy_1),\ldots,\SIGMA(\vy_k))}_{t,\eps}^\ell}.
	\end{align*}
Further, \Lem~\ref{Lemma_NishimoriInt}, \Cor~\ref{Cor_intContig} and {\bf SYM} yield
	$\Erw\brk{\bck{\PSI(\SIGMA(\vy_{1}),\ldots,\SIGMA(\vy_{k}))}_{t,\eps}}=\xi+o(1)$.
\end{proof}

\begin{claim}[\SYM, \BAL]\label{Lemma_Deltat'}
With $\RHO_1,\RHO_2,\ldots$ drawn from $\pi$ mutually independently and independently of everything else,
	\begin{align*}
	\Delta_t'	&=-\frac{1-\xi}{\xi}+\sum_{\ell\geq2}\frac{1}{\ell(\ell-1)\xi}
				\Erw\brk{\bck{1-\sum_{\tau_1,\ldots,\tau_{k-1}\in\Omega}\PSI(\tau_1,\ldots,\tau_{k-1},\SIGMA(\vy_1))
					\prod_{j=1}^{k-1}\RHO_j(\tau_j)}_{t,\eps}^\ell}.
	\end{align*}
\end{claim}
\begin{proof}
\Lem~\ref{Lemma_NishimoriInt} shows that $\hat\SIGMA_{n,m,\vec\gamma_t,\vec m_t,\THETA_\eps},\hat\SIGMA_{n,m,\vec\gamma_t+\vecone_x,\vec m_t,\THETA_\eps}$ are identically distributed and hence we can couple them identically.
Let us write $\hat\SIGMA$ for brevity.
Further, we couple
	\begin{align*}
	\G'&\stacksign{$\mathrm d$}=\hat\G_{t,\eps}(\vec m_t,\vec\gamma_t),&
	\G''&\stacksign{$\mathrm d$}=\hat\G_{t,\eps}(\vec m_t,\vec\gamma_t+\vecone_{\vy_1})
	\end{align*}
in the natural way: first choose $\G'$ from the distribution $\hat\G_{t,\eps}(\vec m_t,\vec\gamma_t)$, then obtain $\G''$ simply by adding one more unary constraint $\vb$ with $\partial\vb=\vy_1$ according to step {\bf G2} of our construction.
Then
	\begin{align}\label{eqLemma_Deltat'2}
	\Erw[\ln Z(\hat\G_{t,\eps}(\vec m_t,\vec\gamma_t+\vecone_{\vec x}))]&-\Erw[\ln Z(\hat\G_{t,\eps}(\vec m_t,\vec\gamma_t))]=
		\Erw\brk{\ln\frac{Z(\G'')}{Z(\G')}}=\Erw\brk{\ln\bck{\psi_{\vec b}(\SIGMA(\vy_1))}_{\G'}}.
	\end{align}
Since $\psi_{\vec b}(\hat\SIGMA(\vy_1))>0$ by construction, we see that $0<\bck{\psi_{\vec b}(\SIGMA(\vy_1))}_{\G'}\leq 1$ and therefore by Fubini's theorem
	\begin{align*}
	\Erw\brk{\ln\bck{\psi_{\vec b}(\SIGMA(\vy_1))}_{\G'}}&=-\Erw\brk{\sum_{\ell\geq1}\frac1\ell\bck{1-\psi_{\vec b}(\SIGMA(\vy_1))}_{\G'}^\ell}
		=-\sum_{\ell\geq1}\frac1\ell\Erw\brk{\bck{1-\psi_{\vec b}(\SIGMA(\vy_1))}_{\G'}^\ell}.
	\end{align*}
Hence, due to {\bf INT2}, the upper bound $\psi_{\vec b}\leq1$, \Lem~\ref{Lemma_NishimoriInt} and assumption {\bf SYM},
	\begin{align*}
	\Erw\brk{\ln\bck{\psi_{\vec b}(\SIGMA(\vy_1))}_{\G'}}&=-\sum_{\ell\geq1}\frac1{\xi\ell}
		\Erw\brk{\bc{\sum_{\tau\in\Omega^{k-1}}\PSI(\tau,\hat\SIGMA(\vy_1))\prod_{j<k}\RHO_j(\tau_j)}\bck{\prod_{h=1}^\ell\bc{1-\sum_{\tau\in\Omega^{k-1}}\PSI(\SIGMA_h(\vy_1))\prod_{j<k}\RHO_j(\tau_j)}}_{\G'}}\\
		&=-\sum_{\ell\geq1}\frac1{\xi\ell}
			\Erw\brk{\bck{1-\sum_{\tau\in\Omega^{k-1}}\PSI(\SIGMA(\vy_1))\prod_{j=1}^{k-1}\RHO_j(\tau_j)}_{\G'}^\ell-
				\bck{1-\sum_{\tau\in\Omega^{k-1}}\PSI(\SIGMA(\vy_1))\prod_{j=1}^{k-1}\RHO_j(\tau_j)}_{\G'}^{\ell+1}}\\
		&=\xi^{-1}\brk{\Erw\brk{\PSI(\SIGMA(\vy_1))}-1+\sum_{\ell\geq2}\frac1{\ell(\ell-1)}\Erw\brk{\bck{1-\PSI(\SIGMA(\vy_1))}_{\G'}^\ell}}\\
		&=-\frac{1-\xi}{\xi}+\sum_{\ell\geq2}\frac1{\xi\ell(\ell-1)}\Erw\brk{\bck{1-\PSI(\SIGMA(\vy_1))}_{\G'}^\ell},
	\end{align*}
as claimed.
\end{proof}

\begin{claim}[\SYM, \BAL]\label{Lemma_Deltat''}
With $\RHO_1,\RHO_2,\ldots$ drawn from $\pi$ mutually independently and independently of everything else,
	\begin{align*}
	\Delta_t''&=\frac k{d(k-1)}\frac{\partial}{\partial t}\Gamma_t=
		-\frac{1-\xi}{\xi}+\sum_{\ell\geq2}\frac{1}{\ell(\ell-1)\xi}
			\Erw\brk{\bc{1-\sum_{\tau\in\Omega^k}\PSI(\tau)\prod_{j=1}^k\RHO_j(\tau_j)}^\ell}
	\end{align*}
\end{claim}
\begin{proof}
Since $\Erw\brk{\sum_{\tau\in\Omega^k}\PSI(\tau)\prod_{j=1}^k\RHO_j(\tau_j)} = \xi$
this follows along the lines of the proof of Claim~\ref{Lemma_Deltat}.
\end{proof}

\begin{proof}[Proof of \Lem~\ref{Prop_deriv}]
The assertion is immediate from Claims~\ref{Lemma_PoissonDeriv}--\ref{Lemma_Deltat''}.
\end{proof}

\begin{proof}[Proof of \Prop~\ref{Lemma_interpolation}]
Let $\vec\pi_{t,\eps}$ be the empirical distribution of the marginals of the random probability measure $\mu_{\hat\G_{t,\eps}}$.
Write $\NU_1,\NU_2,\ldots$ for independent samples drawn from $\vec\pi_{t,\eps}$ and define
	\begin{align*}
	\Xi_{t,\ell}'&=
	\Erw\brk{\bc{1-\sum_{\sigma\in\Omega^k}\PSI(\sigma)\prod_{j=1}^k\NU_j(\sigma_j)}^\ell
		-k\bc{1-\sum_{\tau\in\Omega^k}\PSI(\tau)\NU_1(\tau_k)\prod_{j<k}\RHO_j(\tau_j)}^\ell
			+(k-1)\bc{1-\sum_{\tau\in\Omega^k}\PSI(\tau)\prod_{j=1}^k\RHO_j(\tau_j)}^\ell}.
	\end{align*}
\Lem~\ref{Lemma_lwise} implies that for any $\eta>0$, $\ell\geq1$ there is $\eps>0$ such that in the case that $\mu_{\hat\G_{t,\eps}}$ is $\eps$-symmetric 
for all $\psi\in\Psi$ and all $t\in[0,1]$ we have
	\begin{align}\label{eqLemma_interpolation_1}
	\frac1{n^k}\sum_{y_1,\ldots,y_k\in V}\abs{\bck{{1-\psi(\SIGMA( y_1),\ldots,\SIGMA( y_k))}}_{\hat\G_{t,\eps}}^\ell
		-\bc{1-\sum_{\sigma\in\Omega^k}\psi(\sigma)\prod_{j=1}^k\bck{\vecone\{\SIGMA(y_j)=\sigma_j\}}_{\hat\G_{t,\eps}}}^\ell}<\eta.
	\end{align}
Further, $\mu_{\hat\G_{t,\eps}}$ is $\eps$-symmetric with probability at least $1-\eps$ by \Lem~\ref{Lemma_pinning}.
Consequently, for any $\ell$ and any $\eta>0$ we can pick $\eps>0$ small enough so that $|\Xi_{t,\ell}-\Xi_{t,\ell}'|<\eta$.
Finally, since $|\Xi_{t,\ell}|\leq2k$ for all $t,\ell$ and because the series $\sum_{\ell\geq2}1/(\ell(\ell-1))$ converges, the assertion follows from {\bf POS}, (\ref{eqLemma_interpolation_1}) and \Lem~\ref{Prop_deriv}.
\end{proof}

\begin{proof}[Proof of \Prop~\ref{Prop_interpolation}]
By construction, $\hat\G_{1,\eps}$ is obtained from $\hat\G$ by adding further constraints.
Therefore, invoking \Prop~\ref{Lemma_interpolation} and the fundamental theorem of calculus, we find that for any $\delta>0$ there is $\eps>0$ such that
	\begin{align}\label{eqProp_interpolation1}
	\Erw[\ln Z(\hat\G)]\geq\Erw[\ln Z(\hat\G_{1,\eps})]\geq\Erw[\ln Z(\hat\G_{0,\eps})]-\Gamma_1 n-\delta n+o(n).
	\end{align}
Furthermore, since $\hat\G_{0,\eps}$ consists of unary constraints only and since the number $\THETA_\eps$ of pinned variables is bounded,
we see that $\Erw[\ln Z(\hat\G_{0,\eps})]\geq\Erw[\ln Z(\hat\G_{0,1})]-O(1)$.
Hence, taking $n\to\infty$ and then $\eps\to 0$, we obtain from (\ref{eqProp_interpolation1}) that
	\begin{align}\label{eqProp_interpolation2}
	\liminf_{n\to\infty}\frac1n\Erw[\ln Z(\hat\G)]\geq\liminf_{n\to\infty}\frac1n\Erw[\ln Z(\hat\G_{0,1})]-\Gamma_1.
	\end{align}

Thus, we are left to compute $\Erw[\ln Z(\hat\G_{0,1})]$.
We claim that with  independent $\vec\gamma=\Po(d)$, $\PSI_i$ from $P$ and $(\RHO_{h,i})_{h,i\geq1}$ chosen from $\pi$,
	\begin{align}\label{eqProp_interpolation3}
	\frac1n\Erw[\ln Z(\hat\G_{0,1})]&=\frac1{q}\Erw\brk{\xi^{-\vec\gamma}
			\Lambda\bc{\sum_{\sigma\in\Omega}\prod_{h=1}^{\vec\gamma}\sum_{\tau\in\Omega^k}\vecone\{\tau_{k}=\sigma\}\PSI_h(\tau)\prod_{j=1}^{k-1}\RHO_{h,j}(\tau_j)}}.
	\end{align}
Indeed, since $\hat\G_{0,1}$ has unary constraints only, $\Erw[\ln Z(\hat\G_{0,1})]$ is equal to $n$ times the contribution of just the component of $\hat\G_{0,1}$ that contains the constraint $x_1$.
Formally, we have
	\begin{align}\label{eqProp_interpolation4}
	\Erw[\ln Z(\hat\G_{0,1})]&=\frac nq\Erw[\xi^{-\vec\gamma_{x_1}}\Lambda(\vec z)],&\qquad \mbox{where}\qquad 
	\vec z&=\sum_{\sigma\in\Omega}\prod_{j=1}^{\vec\gamma_{x_1}}\psi_{b_{1,j}}(\sigma),	
	\end{align}
because the constraints are chosen with a probability that is proportional to the partition function.
Finally, the assertion follows from {\bf INT2} and  (\ref{eqProp_interpolation1})--(\ref{eqProp_interpolation4}).
\end{proof}

\subsection{The upper bound}\label{Sec_upperBound}
To bound $\Erw\brk{\ln Z(\hat\G(n,m,P))}$ from above we use the formula for $\Erw\brk{\ln Z(\hat\G(n,m,P_\beta))}$ from \cite{CKPZ} and take the limit $\beta\to\infty$.
To this end we need to show that $\Erw\brk{\ln Z(\hat\G(n,m,P_\beta))}$ is an asymptotic upper bound on $\Erw\brk{\ln Z(\hat\G(n,m,P))}$ for large $\beta$.

\begin{proposition}[\SYM, \BAL]\label{Prop_upperBound}
For any $d>0$ and any $\eps>0$ there exists $\beta_0>0$ and $n_0>0$ such that for all $m\in\cM(d)$, $\beta>\beta_0$ and $n>n_0$ we have
	$\Erw\brk{\ln Z(\hat\G(n,m,P))}\leq\Erw\brk{\ln Z(\hat\G(n,m,P_\beta))}+\eps n.$
\end{proposition}

\noindent
To prove \Prop~\ref{Prop_upperBound} we need the following basic fact about the random assignments $\hat\SIGMA_{n,m}$, $\hat\SIGMA_{n,m,P_\beta}$.

\begin{lemma}[\BAL]\label{Lemma_upperBound1}
For any $d>0$ and any $\eps>0$ there exists $\beta_0>0$ and $n_0>0$ such that for all $m\in\cM(d)$, $\beta>\beta_0$ and $n>n_0$ for any nearly balanced $\sigma$ we have
	$\Erw[\ln Z(\G^*(n,m,P_\beta,\sigma))]\leq\Erw[\ln Z(\G^*(n,m,P_\beta,\hat\SIGMA_{n,m,P_\beta}))]+\eps n.$
\end{lemma}

\begin{proof}
\Lem~\ref{Cor_intContig} shows that $\hat\SIGMA_{n,m,P}$ is nearly balanced with probability $1-O(n^{-2})$ and  due to \Lem~\ref{Lemma_conditions} the same holds for $\hat\SIGMA_{n,m,P_\beta}.$
Further, since $\psi_\beta(\tau)\leq1$ for all $\psi\in\Psi$, we have the deterministic upper bound
	 $$\ln Z(\G^*(n,m,\beta,\hat\SIGMA_{n,m,\beta}))\leq n\ln q.$$
Therefore, it suffices to prove that 
	\begin{align}\label{eqLemma_upperBound1}
	\Erw\brk{\ln Z(\G^*(n,m,P_\beta,\sigma))}&\leq\Erw\brk{\ln Z(\G^*(n,m,P_\beta,\hat\SIGMA_{n,m,P_\beta}))\mid\hat\SIGMA_{n,m,P_\beta} \hspace{0.1cm} \mbox{is nearly balanced}}+\eps n/2.
	\end{align}
Hence, suppose that $\hat\SIGMA_{n,m,P_\beta}$ is nearly balanced. 
Since $\sigma$ is nearly balanced as well, there is a permutation $\pi$ of $[n]$ such that the symmetric difference satisfies
	$|(\sigma\circ\pi)\triangle\hat\SIGMA_{n,m,P_\beta}|\leq 2qn^{3/5}$.
Indeed, because the value of the partition function is invariant under permutations of the variables, we may assume without loss that $\pi=\id$.

Letting $U=\sigma\triangle\hat\SIGMA_{n,m,P_\beta}$, we couple $\G^*(n,m,P_\beta,\sigma)$ and $\G^*(n,m,P_\beta,\hat\SIGMA_{n,m,P_\beta})$ as follows.
Keeping in mind that the constraints are chosen independently according to (\ref{eq:myplanted}),
we first reveal for each $i=1,\ldots,m$ whether the corresponding constraint is adjacent to a variable in $U$ in either $\G^*(n,m,P_\beta,\sigma)$ or  $\G^*(n,m,P_\beta,\hat\SIGMA_{n,m,P_\beta})$.
If not, then the definition of the models ensures that the distribution of the constraint is identical in the two models and couple such that the $i$th constraints in the two factor graphs are identical.
If, on the other hand, the $i$th constraint is adjacent to $U$ in either instance, then we insert independently chosen constraints.

Let $X$ be the number of constraints on which the two CSP instances differ under this coupling.
Since the addition or removal of a single constraint can alter the partition function by at most a factor of $\exp(\pm\beta)$, we obtain
	\begin{align}\label{eqLemma_upperBound2}
	\Erw\brk{\ln Z(\G^*(n,m,P_\beta,\sigma))-\ln Z(\G^*(n,m,P_\beta,\hat\SIGMA_{n,m,P_\beta}))\mid X,\hat\SIGMA_{n,m,P_\beta}}&\leq
		2\beta X.
	\end{align}
Hence, we are left to bound $X$.
Due to the independence of the constraints $X$ is a binomial random variable.
Moreover, since $\sigma$ is nearly balanced and  $|U|\leq 2qn^{3/5}$ assumption \SYM\ yields
	\begin{align*}
	\sum_{h_1,\ldots,h_k\in[n]}\Erw[\PSI_\beta(\sigma(x_{h_1},\ldots,x_{h_k}))]&=(\xi_\beta+o(1))n^k,&
	n^{-k}\sum_{h_1,\ldots,h_k\in[n]}\Erw[\PSI_\beta(\hat\SIGMA_{n,m,P_\beta}(x_{h_1},\ldots,x_{h_k}))]&=(\xi_\beta+o(1))n^k.
	\end{align*}
Thus, the bound $|U|\leq 2qn^{3/5}$ implies together with the construction~\eqref{eq:myplanted} of the planted model that
	\begin{align*}
	\Erw[X\mid\hat\SIGMA_{n,m,P_\beta}]&\leq\frac{k|U|m}{(\xi_\beta+o(1))n}=O(m/n^{2/5}).
	\end{align*}
Therefore, the Chernoff bound yields  $\pr[X>n^{0.9}\mid\hat\SIGMA_{n,m,P_\beta}]\leq O(n^{-2}).$
Thus, (\ref{eqLemma_upperBound1}) follows from \eqref{eqLemma_upperBound2} and the deterministic upper bound  $\ln Z(\G^*(n,m,P_\beta,\sigma))\leq n\ln q$.
\end{proof}

\begin{lemma}[\SYM,\BAL]\label{Lemma_upperBound2}
For any $d>0$ and any $\eps>0$ there exists $\beta_0>0$ and $n_0>0$ such that for all $m\in\cM(d)$, $\beta>\beta_0$ and $n>n_0$ for any nearly balanced $\sigma$ we have
	$\Erw\brk{\ln Z(\G^*(n,m,P,\sigma))}\leq\Erw\brk{\ln Z(\G^*(n,m,P_\beta,\sigma))}+\eps n$.
\end{lemma}
\begin{proof}
We use a coupling argument once more.
We begin by calculating the total variation distance of the distributions from~\eqref{eq:myplanted} according to which the constraints of
 $\G^*(n,m,P,\sigma)$ and $\G^*(n,m,P_\beta,\sigma)$ are drawn.
First, because $\sigma$ is nearly balanced, {\bf SYM} shows that
	\begin{align*}
	\sum_{j_1,\ldots,j_k\in[n]}\Erw[\PSI(\sigma(x_{j_1}),\ldots,\sigma(x_{j_k}))]&\sim\xi n^k,&
			\sum_{j_1,\ldots,j_k\in[n]}\Erw[\PSI_\beta(\sigma(x_{j_1}),\ldots,\sigma(x_{j_k}))]&\sim\xi_\beta n^k,&
	\end{align*}
Hence, plugging in the definition (\ref{eqsoft}) of the softened constraints we obtain for any $\psi\in\Psi$ and any $i_1,\ldots,i_k\in[n]$,
	\begin{align*}
	&\abs{\frac{\psi(\sigma(x_{i_1}),\ldots,\sigma(x_{i_k}))P(\psi)}
			{\sum_{j_1,\ldots,j_k\in[n]}\Erw[\PSI(\sigma(x_{j_1}),\ldots,\sigma(x_{j_k}))]}-
		\frac{\psi_\beta(\sigma(x_{i_1}),\ldots,\sigma(x_{i_k}))P_\beta(\psi_\beta)}
			{\sum_{j_1,\ldots,j_k\in[n]}\Erw[\PSI_\beta(\sigma(x_{j_1}),\ldots,\sigma(x_{j_k}))]}}\\
			&\qquad=o(n^{-k})+
				\abs{\frac{\psi(\sigma(x_{i_1}),\ldots,\sigma(x_{i_k}))P(\psi)}
			{\xi n^k}-
		\frac{\psi_\beta(\sigma(x_{i_1}),\ldots,\sigma(x_{i_k}))P_\beta(\psi_\beta)}
			{\xi_\beta n^k}}
			\leq o(n^k)+\frac{P(\psi)}{n^k}\cdot\frac{1}{\xi(1+(\eul^{\beta}-1)\xi)}.
	\end{align*}
Summing on $\psi$, $i_1,\ldots,i_k$, we conclude that the total variation distance of the distributions defined by (\ref{eq:myplanted})
for $P$ and $P_\beta$, respectively, is bounded by $O(\exp(-\beta))$ for large $\beta$.
Hence, we can couple these distributions such that they coincide with probability $1-O(\exp(-\beta))$.
We then extend this coupling of the distribution of individual constraints to a coupling of $\G^*(n,m,P,\sigma)$ and $\G^*(n,m,P_\beta,\sigma)$ by drawing $m$ times independently.

Letting $X$ be the number of constraints in which $\G^*(n,m,P_\beta,\sigma)$, $\G^*(n,m,P,\sigma)$ differ, 
we thus obtain the estimate $\Erw[X]\leq O(\exp(-\beta))m$ for large $\beta$.
Further, because the constraints are chosen independently, $X$ is a binomial random variable.
Thus, for large enough $\beta$ the Chernoff bound shows that 
	\begin{align}\label{eqLemma_upperBound2_2a}
	\pr\brk{X>n/\beta^{2}}=O(n^{-2}).
	\end{align}
Additionally, since $\psi_\beta(\sigma)\in[\exp(-\beta),1]$ for all $\psi\in\Psi$, $\sigma\in\Omega^k$, we obtain the estimate
	\begin{align}\label{eqLemma_upperBound2_2}
	\Erw\brk{\ln Z(\G^*(n,m,P,\sigma))-\ln Z(\G^*(n,m,P_\beta,\sigma))\mid X}&\leq X\beta.
	\end{align}
Finally, the assertion follows from \eqref{eqLemma_upperBound2_2a}, (\ref{eqLemma_upperBound2_2}) and the deterministic bound  $\ln Z(\G^*(n,m,P_\beta,\sigma))\leq n\ln q$, provided that $\beta=\beta(\eps)$ is sufficiently large.
\end{proof}

\noindent
Finally, \Prop~\ref{Prop_upperBound} is immediate from \Lem s~\ref{Lemma_upperBound1} and~\ref{Lemma_upperBound2}.

\begin{proof}[Proof of \Thm~\ref{Thm_planted}]
To show the first part of the theorem assume that conditions \SYM \, and \BAL\, hold. \Prop~\ref{Prop_upperBound} and \cite[\Prop~3.6]{CKPZ} readily imply that there exists $\beta_0$ such that for all $d>0$ and $\beta > \beta_0$
\begin{align*}
\limsup_{n\to\infty}\frac 1n\Erw[\ln Z(\hat\G)]&\leq
		\sup_{\pi\in\cP_*^2(\Omega)} \cB(d, P_\beta, \pi) .	 
\end{align*}
Now, as $\Lambda$ is bounded and continuous on $[0,1]$ the convergence of the Bethe functional follows from the dominated convergence theorem.

Moving on to the second part, assume that additionally condition \POS\, holds. In order to make use of \Prop~\ref{Prop_interpolation}, we need to show that every $\pi \in \mathcal{P}_\ast^2(\Omega)$ can be approximated arbitrarily well by distributions in $\mathcal{P}_\ast^2(\Omega)$ that have finite support.
To this regard let $S_q$ denote the standard simplex in $\RR^{\Omega}$, let $\pi \in \mathcal{P}_\ast^2(\Omega)$ be a probability distribution that does {\it not} have finite support and let $B: \NN_0 \times \bc{[0,1]^k}^\infty \times \bc{S_q}^\infty \to \RR$,
\begin{align*}
\bc{\gamma, \bc{\psi_i}_{i\geq 1}, \bc{\rho_i}_{i\geq 1}}\mapsto q^{-1}\xi^{-\gamma}
			\Lambda\bc{\sum_{\sigma\in\Omega}\prod_{i=1}^{\gamma}\sum_{\tau\in\Omega^k}\vecone\{\tau_k=\sigma\}\psi_i(\tau)\prod_{j=1}^{k-1}\rho_{ki+j}(\tau_j)}
	-\frac{d(k-1)}{k\xi}\Lambda\bc{\sum_{\tau\in\Omega^k}\psi_1(\tau)\prod_{j=1}^k\rho_j(\tau_j)}
\end{align*} 
be as in the definition of $\cB(d,P,\pi)$. We wish to approximate $\cB(d,P,\pi)$ by $\cB(d,P,\pi_N)$, where $\pi_N \in \mathcal{P}_\ast^2(\Omega)$ has finite support and $| \mathrm{supp}(\pi_N)|=N$.
To this end, we proceed along the following lines:
\begin{enumerate}
\item For every $N \in \mathbb{N}$, we find a discrete probability measure $\pi_N$ on $S_q$, whose support consists of exactly $N$ elements such that $\int_{\cP(\Omega)}\mu(\omega) d\pi(\mu)=1/q$ for all $\omega \in \Omega$ and $(\pi_N)_{N \geq 1}$ converges weakly to $\pi$ as $N \to \infty$.
\item This implies that $B(\vec \gamma, (\PSI_i)_{i \geq 1}, (\RHO^{\pi_N}_i)_{i \geq 1})$ converges weakly to $B(\vec \gamma, (\PSI_i)_{i \geq 1}, (\RHO^{\pi}_i)_{i \geq 1})$. Here, all occurring random variables are independent.
\item We then apply a variant of the dominated convergence theorem to show convergence of $\cB(d,P,\pi_N)$ to $\cB(d,P,\pi)$.
\end{enumerate}
Step $(1)$ is a quantisation problem: fix $N \in \NN$ and let
$\cF_N$ be the set of all Borel measurable maps $f: \RR^\Omega \to \RR^\Omega$ with $|f(\RR^\Omega)| \leq N$. The standard theory on quantisation for probability distributions, \cite[Theorem 4.1 and Theorem 4.12]{graf2007}, guarantees the existence of a function $f_N^\ast:\RR^\Omega \to \RR^\Omega$ with $|f_N^\ast(\RR^\Omega)|=N$ and $$ \Erw\brk{\left\|\RHO_1^\pi - f_N^\ast\bc{\RHO_1^{\pi}}\right\|^2} = \inf_{f \in \cF_N} \Erw\brk{\left\|\RHO_1^\pi - f\bc{\RHO_1^{\pi}}\right\|^2}.$$ 
Here, $\|\cdot \|$ denotes the $2$-norm on $\RR^\Omega$. Moreover, the use of this norm implies \cite[Remark 4.6]{graf2007} that for any such function $f_N^\ast$,
$\Erw\brk{ f_N^\ast\bc{\RHO_1^{\pi}}} = \Erw\brk{\RHO_1^{\pi}}.$ 
{
In order to see why $ \Erw[\|\RHO_1^\pi -
    f_N^\ast\bc{\RHO_1^{\pi}}\|^2] = o(1),$ we evoke the following
almost sure approximation of $\RHO_1^\pi $ which does not fix the mean
value, but provides an upper bound for $\Erw[\|\RHO_1^\pi -
    f_N^\ast\bc{\RHO_1^{\pi}}\|^2] .$ For any $L \in \NN,$ choose
a
cover of $S_q$ by open balls of radius $1/L$. As
$S_q$ is compact, this cover has a finite sub-cover. By taking
intersections of the balls in a finite sub-cover, we may assume that
$S_q$ is covered by a finite number of pairwise disjoint sets $B_1, \ldots,
B_{j(L)}$, which have diameter at most $2/L$. In each such set $B_i$, we
distinguish a point $c_
i$. Setting
$g_L^\ast(\RHO_1^\pi) = \sum_{i=1}^{j(L)} c_i \vecone\{\RHO_1^\pi \in
B_i\}$, we have that almost surely, $\|\RHO_1^\pi-
g_L^\ast(\RHO_1^\pi) \| \leq 2/L$ and the distribution of $g_L^\ast(\RHO^\pi)$
has finite support. We may thus find a sequence $(g_L^\ast)_L$ of
functions which take only finitely many values each such that
$g_L^\ast(\RHO_1^\pi)$ converges to $\RHO_1^\pi$ almost surely. Because both $\RHO_1^\pi $ and
$g_L^\ast(\RHO_1^\pi)$ are bounded, $ \Erw[\|\RHO_1^\pi -
    g_L^\ast\bc{\RHO_1^{\pi}}\|^2] = o(1)$ and thus also $ \Erw[\|\RHO_1^\pi -
    f_N^\ast\bc{\RHO_1^{\pi}}\|^2]= o(1).$
This in turn implies that, if we denote the distribution of $f_N^\ast\bc{\RHO_1^{\pi}}$ by $\pi_N$, $\bc{\pi_N}_{N \in \NN}$ converges weakly to $\pi$.}

We now turn to $(2)$: Step $(1)$ implies that $\bigotimes_{i=1}^{\infty} \pi_N$ converges weakly to $\bigotimes_{i=1}^{\infty} \pi$ as $N \to \infty$, as $\bigotimes_{i=1}^{\infty} \pi$ is determined by its finite dimensional distributions. Due to independence, it is true that also $$\bc{\vec \gamma, \bc{\PSI_i}_{i \geq 1}, \bc{\RHO^{\pi_N}_i}_{i \geq 1}}_{N \geq 1} \stackrel{N \to \infty}{\longrightarrow} \bc{\vec \gamma, \bc{\PSI_i}_{i \geq 1}, \bc{\RHO^{\pi}_i}_{i \geq 1}}$$ in distribution. Finally, $B$ is a continuous function, and thus the continuous mapping theorem implies step $(2)$.

Finally, $B(\vec \gamma, (\PSI_i)_{i \geq 1}, (\RHO^{\pi_N}_i)_{i \geq 1})$ is integrable for any $N \in \NN$ as well as dominated by the integrable random variable $q^{-1}\xi^{-\vec \gamma} + d(k-1)k^{-1}\xi^{-1}$.
Hence, the dominated convergence theorem (say, in the version \cite[Theorem A$39$]{Hofstad}) yields (3).

Finally, \Prop~\ref{Prop_interpolation} yields the second part of the theorem. 
\end{proof}

\section{Small subgraph conditioning}\label{sec:ProofPreCond}

\noindent
Having established \Thm~\ref{Thm_planted} in the previous section, we move on to prove the remaining propositions required for the small subgraph conditioning argument outlined in \Sec~\ref{Sec_Proofs}.
Subsequently we derive \Thm s~\ref{Thm_cond},  \ref{Thm_overlap} and \ref{thmContiquity} as well as Corollary~\ref{Cor_thmContiquity}.
Most of the proofs in this section are based either on standard arguments (e.g., the Laplace method or the method of moments for convergence in distribution) or the arguments developed in~\cite{SoftCon,CKPZ}.
We continue to denote by $\vx_1,\ldots,\vx_k\in V_n$ variables drawn  uniformly and independently.

\subsection{Proof of Proposition \ref{prop:belowcond-unif}}
\Prop~\ref{Prop_Deltat} provides a formula for the expected change of the logarithm of the partition function upon addition of a further constraint.
We can use this formula to estimate the derivative of $\Erw[\ln Z(\hat\G)]$ with respect to $d$ because
	\begin{equation}\label{eqprop:belowcond-unif}
	\frac{\partial}{\partial d}\Erw[\ln Z(\hat\G)]=\sum_{m\geq0}\Erw[\ln Z(\hat\G(n,m))]\frac{\partial}{\partial d}\pr\brk{\Po(dn/k)=m}
		=\frac1k\bc{\Erw[\ln Z(\hat\G(n,\vm+1))]-\Erw[\ln Z(\hat\G(n,\vm))]}.
	\end{equation}
The corresponding formula in the case of soft constraints was obtained in~\cite{SoftCon}, and thanks to \Prop~\ref{Prop_Deltat} the same argument extends to hard constraints with a little bit of care.

\begin{lemma}[\SYM, \BAL, \MIN]\label{lem:badoverlaps1}
Fix any $D>0$.
\begin{enumerate}
\item Uniformly for all $0<d<D$ we have
	\begin{align}\label{eq:derivPreCond}
	\frac{1}{n}\frac{\partial}{\partial d}\Erw[\ln Z(\hat\G)]\ge\frac{\ln\xi}{k}+o(1).
	\end{align}
\item  For any $\eps>0$ there is $\delta=\delta(\eps,P)>0$, independent of $n$ or $d$, such that uniformly for all $0<d<D$,
	\begin{align}\label{eq:derivPostCond}
	\Erw\bck{\TV{\rho_{\SIGMA,\TAU}-\bar\rho}}_{\hat\G}>\eps&\ \Rightarrow\ 
	\frac{1}{n}\frac{\partial}{\partial d}\Erw[\ln Z(\hat\G)]\ge\frac{\ln\xi}{k}+\delta+o(1).
	\end{align}
\item Conversely, we have
	\begin{align}\label{eq:derivPostCondConverse}
	\Erw\bck{\TV{\rho_{\SIGMA,\TAU}-\bar\rho}}_{\hat\G}=o(1)&\ \Rightarrow\ 
	\frac{1}{n}\frac{\partial}{\partial d}\Erw[\ln Z(\hat\G)]=\frac{\ln\xi}{k}+o(1).
	\end{align}
\end{enumerate}
\end{lemma}
\begin{proof}
The first two assertions and their proofs are nearly identical to the soft constraint version~\cite[\Cor~6.3]{SoftCon}; we still include the brief argument for completeness and because it leads up to the proof of the third assertion.
Due to (\ref{eqprop:belowcond-unif}) we obtain from \Prop~\ref{Prop_Deltat} that uniformly for all $d<D$,
	\begin{align}
	\frac kn\frac{\partial}{\partial d}\Erw[\ln Z(\hat\G)]&=
		o(1)
		+\xi^{-1}\Erw\brk{\Lambda\bc{\bck{\PSI(\SIGMA(\vy_{1}),\ldots,\SIGMA(\vy_{k}))}_{\hat\G}}}.
			\label{eqlem:badoverlaps6}
	\end{align}
Further, \eqref{eq:nishimori}, \Cor~\ref{lem:conc_coloring} and {\bf SYM} yield
	\begin{align}\label{eqlem:badoverlaps777}
	\Erw\bck{\PSI(\SIGMA(\vy_{1}),\ldots,\SIGMA(\vy_{k}))}_{\hat\G}&=
		\Erw\brk{\PSI(\hat\SIGMA(\vy_{1}),\ldots,\hat\SIGMA(\vy_{k}))}=\xi+o(1).
	\end{align}
Since
$\bck{\PSI(\SIGMA(\vx_1),\ldots,\SIGMA(\vx_k))}_{\hat\G}\in(0,1]$ and  $\Lambda''(x)\geq1/2$ for all $x\in(0,1]$, Taylor's formula gives
	\begin{align}\nonumber
	\Erw\brk{\Lambda\bc{\bck{\PSI(\SIGMA(\vx_1),\ldots,\SIGMA(\vx_k))}_{\hat\G}}}&\geq \Lambda(\xi)+
		\Lambda'(\xi)\brk{\Erw\bck{\PSI(\SIGMA(\vx_1),\ldots,\SIGMA(\vx_k))}_{\hat\G}-\xi}+
		\frac14\Erw\brk{\bc{\bck{\PSI(\SIGMA(\vx_{1}),\ldots,\SIGMA(\vx_{k}))}_{\hat\G}-\xi}^2}\\
	&=\Lambda(\xi)+
		\frac14\Erw\brk{\bck{\PSI(\SIGMA(\vx_{1}),\ldots,\SIGMA(\vx_{k}))}_{\hat\G}^2}-\frac{\xi^2}4+o(1)
			\qquad\mbox{[by \eqref{eqlem:badoverlaps777}]}.
	\label{eqlem:badoverlaps10}
	\end{align}
Thus, (\ref{eq:derivPreCond}) is immediate from (\ref{eqlem:badoverlaps777}), (\ref{eqlem:badoverlaps10}) and Jensen's inequality.

Now assume that $\Erw\bck{\TV{\rho_{\SIGMA,\TAU}-\bar\rho}}_{\hat\G}>\eps$.
Since \Cor~\ref{lem:conc_coloring} and \eqref{eq:nishimori} yield
	$\Erw\bck{\tv{\rho_{\SIGMA}-\bar\rho}+\tv{\rho_{\TAU}-\bar\rho}}_{\hat\G}=o(1)$,
assumptions {\bf MIN} and {\bf SYM} imply that there is $\delta=\delta(\eps)>0$ such that
	\begin{align}\label{eqlem:badoverlaps13}
	\sum_{\sigma,\tau\in\Omega^k}\Erw\bck{\PSI(\sigma)\PSI(\tau)\prod_{i=1}^k\rho_{\SIGMA,\TAU}\bc{\sigma_i,\tau_i}}_{\hat\G}
		>\delta+o(1)+q^{-2k}\sum_{\sigma,\tau\in\Omega^k}\Erw[\PSI(\sigma)\PSI(\tau)]=
		\xi^2+\delta+o(1).
	\end{align}
Moreover,
	\begin{align}\nonumber
	\Erw\brk{\bck{\PSI(\SIGMA(\vx_{1}),\ldots,\SIGMA(\vx_{k}))}_{\hat\G}^2}&=
		\Erw\bck{\PSI(\SIGMA(\vx_{1}),\ldots,\SIGMA(\vx_{k}))\PSI(\TAU(\vx_{1}),\ldots,\TAU(\vx_{k}))}_{\hat\G}
		=
		\sum_{\sigma,\tau\in\Omega^k}\Erw\bck{\PSI(\sigma)\PSI(\tau)\prod_{i=1}^k\rho_{\SIGMA,\TAU}(\sigma_i,\tau_i)}_{\hat\G}.
	\end{align}
Thus, (\ref{eq:derivPostCond}) follows from \eqref{eqlem:badoverlaps10} and \eqref{eqlem:badoverlaps13}.

With respect to the last assertion, we apply the full Taylor expansion $\Lambda(1-x)=-x+\sum_{\ell\geq2}x^\ell/(\ell(\ell-1))$ to obtain, due to
(\ref{eqlem:badoverlaps777}), that
		\begin{align*}
	\Erw\brk{\Lambda\bc{\bck{\PSI(\SIGMA(\vx_1),\ldots,\SIGMA(\vx_k))}_{\hat\G}}}&=\xi-1+o(1)+
		\Erw\brk{\sum_{\ell\geq2}\frac1{\ell(\ell-1)}\bck{1-\PSI(\SIGMA(\vx_1),\ldots,\SIGMA(\vx_k))}_{\hat\G}^{\ell}}
	\end{align*}
Since $0\leq\PSI\leq1$, all terms of the last sum are in $[0,1]$.
Hence, invoking Fubini's theorem and writing $\SIGMA_1,\SIGMA_2,\ldots$ for independent samples from $\mu_{\hat\G}$, we obtain
	\begin{align}\label{eqlem:badoverlaps66}
	\Erw\brk{\Lambda\bc{\bck{\PSI(\SIGMA(\vx_1),\ldots,\SIGMA(\vx_k))}_{\hat\G}}}&=\xi-1+o(1)+
		\sum_{\ell\geq2}\frac1{\ell(\ell-1)}\Erw\bck{\prod_{h=1}^\ell1-\PSI(\SIGMA_h(\vx_1),\ldots,\SIGMA_h(\vx_k))}_{\hat\G}.
	\end{align}
Moreover, since $\PSI$ is drawn independently of $\hat\G$, we obtain
	\begin{align}\nonumber
	\Erw\bck{\prod_{h=1}^\ell1-\PSI(\SIGMA_h(\vx_1),\ldots,\SIGMA_h(\vx_k))}_{\hat\G}
		=
		\sum_{\chi\in\Omega^{\ell\times k}}&
			\Erw\brk{\prod_{h=1}^\ell(1-\PSI(\chi_{h,1},\ldots,\chi_{h,k}))}\\
			&\quad\cdot\Erw
				\bck{\prod_{i=1}^k\vecone\{(\SIGMA_1(\vx_i)=\chi_{1,i},\ldots,\SIGMA_\ell(\vx_i)=\chi_{\ell,i})\}}_{\hat\G}.
				\label{eqlem:badoverlaps67}	
	\end{align}
We now claim that for any $\ell\geq2$ and for any $\chi\in\Omega^{\ell\times k}$,
	\begin{align}				\label{eqlem:badoverlaps68}	
	\Erw\bck{\prod_{i=1}^k\vecone\{(\SIGMA_1(\vx_i)=\chi_{1,i},\ldots,\SIGMA_\ell(\vx_i)=\chi_{\ell,i})\}}_{\hat\G}=q^{-k\ell}+o(1).
	\end{align}
Indeed, by \Lem~\ref{Lemma_ovsym} the assumption $\Erw\bck{\TV{\rho_{\SIGMA,\TAU}-\bar\rho}}_{\hat\G}=o(1)$ implies that
$\mu_{\hat\G}$ is $o(1)$-symmetric \whp\ and that its marginals satisfy $\sum_{i=1}^n\tv{\mu_{\hat\G,x_i}-\bar\rho}=o(n)$.
Hence, \Lem~\ref{Lemma_prodsym} shows that the $\ell$-fold product measure $\mu_{\hat\G}^{\tensor\ell}$ is $o(1)$-symmetric with asymptotically uniform marginals as well \whp\
Thus, we obtain (\ref{eqlem:badoverlaps68}).
Finally, plugging (\ref{eqlem:badoverlaps68}) into (\ref{eqlem:badoverlaps67}) and (\ref{eqlem:badoverlaps67}) into (\ref{eqlem:badoverlaps66}) and applying {\bf SYM}, we obtain the assertion.
\end{proof}

\begin{lemma}[\SYM, \BAL]\label{Lemma_mon}
For any $\eps>0$, $d>0$ there is $0<\delta=\delta(\eps,d,P)<\eps$ such that the following holds.
Assume that $m\in\cM(d)$ is a sequence such that
	\begin{equation}\label{eqLemma_mon_ass}
	\limsup_{n\to\infty}\Erw\bck{\TV{\rho_{\SIGMA_1,\SIGMA_2}-\bar\rho}}_{\hat\G(n,m)}>\eps.
	\end{equation}
Then 
	$
	\limsup_{n\to\infty}
		\min\cbc{\Erw\bck{\TV{\rho_{\SIGMA_1,\SIGMA_2}-\bar\rho}}_{\hat\G(n,m)}:\delta n<m-dn/k<2\delta n}
		>\delta.
	$
\end{lemma}

\Lem~\ref{Lemma_mon} and its proof are syntactically identical to the soft constraint version~\cite[\Lem~6.1]{SoftCon}.
The proof is included in Appendix~\ref{Sec_Lemma_mon} for the sake of completeness.

\begin{proof}[Proof of Proposition \ref{prop:belowcond-unif}]
The proof of the first assertion is nearly identical to the soft constraint version~\cite[proof of \Prop~3.3]{SoftCon}; we include the argument for completeness.
Assume that there exist $D_0<\dc$, $\eps>0$ such that
	$\limsup_{n\to\infty}\Erw\bck{\tv{\rho_{\SIGMA,\TAU}-\bar\rho}}_{\hat\G(n,\vm(D_0,n))}>\eps.$
Then \Lem~\ref{Lemma_mon} shows that there is $\delta>0$ such that
with $D_1=D_0+3\delta/2<\dc$ for infinitely many $n$ we have 
	\begin{align*}
	\Erw\bck{\TV{\rho_{\SIGMA,\TAU}-\bar\rho}}_{\hat\G(n,\vm)}>\delta+o(1)\qquad\mbox{for all }D_0+4\delta/3<d<D_1.
	\end{align*}
Hence, \Lem~\ref{lem:badoverlaps1} implies that for infinitely many $n$,
	\begin{align*}
	\frac1n\Erw[\ln Z(\hat\G(n,\vm(D_1,n)))]&=\frac1n\Erw[\ln Z(\hat\G(n,\vm(D_0,n)))]+
		\frac1n\int_{D_0}^{D_1}\frac{\partial}{\partial d}\Erw[\ln Z(\hat\G)]\dd d
		\geq\ln q+\frac{D_1}k\ln\xi+\Omega(1).
	\end{align*}
But then the second part of \Thm~\ref{Thm_planted} yields $\sup_{\pi\in\cP_*^2(\Omega)}\cB(D_1,P,\pi)>\ln q+\frac{D_1}k\ln\xi$, in contradiction to $D_1<\dc$.

Analogously, the second assertion follows from the third part of \Lem~\ref{lem:badoverlaps1} by integrating on $d$.
Specifically, assume that $D>0$ is such that (\ref{eq:overlapunif}) is true for all $d<D$.
Pick some $\Delta<D$.
Then by the third part of \Lem~\ref{lem:badoverlaps1} and dominated convergence,
	\begin{align*}
	\Erw[\ln Z(\hat\G(n,\vm(\Delta,n))]=\ln q+\int_0^\Delta\frac{\partial}{\partial d}\Erw[\ln Z(\hat\G)]\dd d
		=\ln q+\frac{\Delta}k\ln\xi+o(1).
	\end{align*}
Hence, \Thm~\ref{Thm_planted} yields $\sup_{\pi\in\cP_*^2(\Omega)}\cB(\Delta,P,\pi)\leq\ln q+k^{-1}\Delta\ln\xi$.
As this holds for all $\Delta\leq D$, we conclude that $\dc\geq D$.
\end{proof}

\subsection{Proof of \Prop~\ref{prop:FirstCondOverFirst}}\label{Sec_cycles}
The distribution of the random variables $C_Y(\G(n,m))$ of the `plain' random CSP can be calculated via a totally standard method of moments argument as set out in~\cite{BB}.
Our assumption that $P(\psi)>0$ for all $\psi\in\Psi$ ensures that $\kappa_Y>0$ for all signatures $Y$.

\begin{lemma}[{\cite{BB}}]\label{lemma:PoissonCycles4Uniform}
Let $d>0$.
For any $Y\in\cY$ we have $\Erw[C_Y(\G(n,m))] \sim  \kappa_Y$, uniformly for all $m\in\cM(d)$.
Moreover, if $Y_1,\ldots,Y_l\in\cY$ are pairwise {disjoint} and $y_1,\ldots,y_l\geq0$, then uniformly for all $m\in\cM(d)$,
	\begin{equation}\label{eqlemma:PoissonCycles4Uniform}
	\Pr\brk{\forall i\leq l: C_{Y_i}(\G(n,m))=y_i}\sim \prod^{l}_{t=1}\Pr[\Po(\kappa_{Y_t})=y_t].
	\end{equation}
\end{lemma}

In order to determine the joint distribution of the random variables $C_Y(\hat\G(n,m))$ we use the method of moments as well.
More specifically, the argument is nearly identical to the one from~\cite{SoftCon}, except that here it may be possible that $\hat\kappa_Y=0$ for some signatures $Y$.

\begin{lemma}[\SYM, \BAL]\label{lemma:PoissonCycles4Planted}
Let $d>0$.
For any $Y\in\cY$ we have $\Erw[C_Y(\hat\G(n,m))]=\hat\kappa_Y+o(1)$, uniformly for all $m\in\cM(d)$.
Moreover, if $Y_1,\ldots,Y_L\in\cY$ are pairwise distinct and $y_1,\ldots,y_L\geq0$, then uniformly for all $m\in\cM(d)$,
	$$\Pr\brk{\forall i\leq \ell: C_{Y_i}(\hat\G(n,m))=y_i}=o(1)+ \prod^{L}_{l=1}\Pr[\Po(\hat\kappa_{Y_l})=y_l].$$
\end{lemma}

\noindent
The proof of \Lem~\ref{lemma:PoissonCycles4Planted} can be found in Appendix~\ref{Sec_PoissonCycles4Planted}.

\begin{proof}[Proof of {\Prop~\ref{prop:FirstCondOverFirst}}]
The fact that $P(\psi)>0$ for all $\psi\in\Psi$ implies immediately that $\kappa_Y>0$ for all signatures $Y$.
Moreover, condition {\bf UNI} implies that $\hat\kappa_Y=0$ can hold only if $Y$ has order one.
Further, if indeed $\hat\kappa_Y=0$, then the corresponding cycle in unsatisfiable deterministically and thus any factor graph $G$ that contains such a cycle satisfies $Z(G)=0$.
Consequently, (\ref{eq:NishimoriG}) ensures that $\pr\brk{C_Y(\hat\G(n,m))>0}=0$ for all $n,m$.
The asymptotic identity (\ref{eqCyclePoisson}) is immediate from \Lem~\ref{lemma:PoissonCycles4Uniform}.
Moreover, \Lem~\ref{lemma:PoissonCycles4Planted} directly implies (\ref{eqCyclePoissonHat}).
\end{proof}

\subsection{Proof of \Thm~\ref{Thm_cond}}\label{sec:BeyondCond}
The first part of \Thm~\ref{Thm_cond} readily follows from \Thm~\ref{Thm_SSC}.

\begin{lemma}[\SYM, \BAL, \MIN, \UNI]\label{Lemma_lowerBound}
If $d<\dc$, then $\lim_{n\to\infty}\Erw\sqrt[n]{ Z(\GG)}= q \xi^{d/k}$.
\end{lemma}
\begin{proof}
Since $\cK>0$ almost surely and because $\eig(\Phi)\subset(-\infty,0]\cup\setminus\{1\}$ by \Lem~\ref{Lemma_Xi},
the assertion is immediate from \Thm~\ref{Thm_SSC}.
\end{proof}

To prove the second part of \Thm~\ref{Thm_cond} concerning $d>\dc$ we generalise an argument for the random graph colouring problem from~\cite[\Sec~4]{CKPZ} to the present broad class of random CSPs.
We begin with the following general fact that essentially goes back to~\cite{Barriers}.
Let $\cM_\eps(d)$ be the set of all sequences $m=m(n)$ such that $|m(n)-dn/k|\leq \eps n$ for all $n$.

\begin{lemma}[\SYM, \BAL]\label{lem:nthsquares}
Let $d>0$. 
For any $\delta,\eta>0$ there is $\eps>0$ such that  the following is true.
Suppose that $ (\cE_n)_n$ is a sequence of events such that uniformly for all $m\in\cM_\eta(d)$,
	\begin{align*}
	\limsup_{n \to \infty}
		\pr \brk{ \G(n,m)\not\in \cE_n}^{1/n} < 1-\delta
\qquad\mbox{ while }\qquad\limsup_{n \to \infty} \pr \brk{ \hat\G(n,m)\in \cE_n}^{1/n} < 1-\delta. 
	\end{align*}
Then uniformly for all $m\in\cM_\eta(d)$,
	\begin{equation}\label{eqlem:nthsquares}
	\limsup_{n \to \infty} \pr\brk{\sqrt[n]{Z(\G(n,m))}\geq q\xi^{d/k}-\eps}^{\frac1n} <1-\eps.
	\end{equation}
\end{lemma}
\begin{proof}
Pick $\eps=\eps(\delta)>0$ sufficiently small,  $\cU_n=\{\sqrt[n]Z\geq q\xi^{d/k}-\eps\}$ and assume that $\limsup\pr\brk{\G(n,m)\in\cU_n}^{\frac1n}=1$.
Then the assumption $	\limsup\pr \brk{ \G(n,m)\not\in \cE_n}^{1/n} < 1-\delta$ implies that for infinitely many $n$,
	\begin{align*}
	\pr \brk{ \G(n,m)\not\in \cE_n\mid\G(n,m)\in\cU_n}^{1/n}\leq\bcfr{\pr \brk{ \G(n,m)\not\in \cE_n}}{\pr \brk{ \G(n,m)\in \cU_n}}^{1/n} < 1-\delta+o(1).
	\end{align*}
Hence, \Prop~\ref{lem:FirstMoment} shows that for infinitely many $n$,
	\begin{align*}
	\pr\brk{\hat\G(n,m)\in\cE_n}&\geq\frac{\Erw\brk{Z(\G(n,m))\vecone\{Z(\G(n,m))\in\cU_n\cap\cE_n\}}}{\Erw[Z(\G(n,m))]}\\
		&\geq\bcfr{q\xi^{d/k}-\eps}{q\xi^{d/k}}^{n+o(n)}\pr \brk{ \G(n,m)\in \cE_n\mid\G(n,m)\in\cU_n}\pr\brk{\G(n,m)\in\cU_n}
			=\exp(o(n)),
	\end{align*}
in contradiction to the assumption that $\limsup \pr \brk{ \hat\G(n,m)\in \cE_n}^{1/n} < 1-\delta$.
Thus, 
$\limsup\pr\brk{\G(n,m)\in\cU_n}^{\frac1n}<1$.
Choosing $\eps>0$ small enough we obtain (\ref{eqlem:nthsquares}).
\end{proof}

\begin{lemma}[\SYM, \BAL, \POS]\label{fact:GastTooBig}
For any $d > \dcond$ there exists $\beta, \delta,\eta >0$ 
such that uniformly for all $m\in\cM_\eta(d)$,
	\begin{align}\label{eq:ElnastGToobig}
	\pr\brk{ \ln Z_\beta(\hat \G(n,m))  < \ln q+\frac dk\ln\xi_\beta  + \delta}&<\exp(-\Omega(n)),\qquad\mbox{while}\\
	\pr\brk{ \ln Z_\beta( \G(n,m))  \geq \ln q+\frac dk\ln\xi_\beta  + \delta}&<\exp(-\Omega(n)).
		\label{eq:ElnastGToobig2}
	\end{align}
\end{lemma}
\begin{proof}
\Lem~\ref{Lemma_conditions} shows that \Prop~\ref{lem:FirstMoment} applies to $\G(n,m,P_\beta)$.
Thus, $\Erw[Z(\G(n,m,P_\beta))]=O(q^n\xi_\beta^m)$ and (\ref{eq:ElnastGToobig2}) is immediate from Markov's inequality.

We move on to the proof of \eqref{eq:ElnastGToobig}.
The definition of $\dc$ implies that for any $d>\dc$ there exist $d'<d$ and $\pi\in\Pomast$ such that
	$\cB(d',P,\pi)>\ln q+(d'\ln\xi)/k$.
Hence, \Thm~\ref{Thm_planted} and show that there is $\delta>0$ and $n_0>0$ such that
	$\Erw \brk{ \ln Z( \hat\G) }>n\bc{\ln q+(d\ln\xi)/k+8\delta}$ for all $n>n_0$.
Moreover, by construction we have $Z_\beta(\hat \G)\geq Z(\hat \G)$ and $\lim_{\beta\to\infty}\xi_\beta=\xi$.
Therefore, there exists $\beta_0>0$ such that
	\begin{align}\label{eqfact:GastTooBig2}
	\Erw \brk{ \ln Z_\beta( \hat\G) }&>n\bc{\ln q+(d\ln\xi_\beta)/k+7\delta}&\mbox{ for all }n>n_0,\,\beta>\beta_0.
	\end{align}

Due to the Nishimori identity \Lem~\ref{lem:nishimori} we can write (\ref{eqfact:GastTooBig2}) as
	\begin{align}\label{eqfact:GastTooBig2}
	\Erw \brk{ \ln Z_\beta( \G^*(n,\vm,\hat\SIGMA_{n,\vm}) }&>n\bc{\ln q+(d\ln\xi_\beta)/k+7\delta}&\mbox{ for all }n>n_0,\,\beta>\beta_0.
	\end{align}
Now, fix $\beta>\beta_0$, pick a small enough $\eta=\eta(\beta,\delta)>0$ and let $\cA$ be the set of all assignments
$\sigma:V_n\to\Omega$ such that $\tv{\rho_\sigma-\bar\rho}<\eta$.
Fix any $\sigma_0\in\cA$.
Given $\hat\SIGMA_{n,\vm}\in\cA$ we can couple $\G'=\G^*(n,\vm,\hat\SIGMA_{n,\vm})$ and $\G''=\G^*(n,\vm,\sigma_0)$ such
	\begin{align}\label{eqfact:GastTooBig3}
	\pr\brk{\abs{\ln Z_\beta( \G')-\ln Z_\beta( \G'')}>\delta n\mid\hat\SIGMA_{n,\vm}\in\cA}
		\leq\exp(-\Omega(n)).
	\end{align}
Indeed, relabelling the variables if necessary, given $\hat\SIGMA_{n,\vm}\in\cA$ we may assume that $|\sigma_0\triangle\hat\SIGMA_{n,\vm}|\leq2q\eta n$.
Further, the planted model can alternatively be described as the result of adding constraints independently according to \eqref{eq:myplanted}.
Let $X$ and $Y$ be the number of constraints of $\G'$ and $\G''$ respectively that are adjacent to a variable in 
$\sigma_0\triangle\hat\SIGMA_{n,\vm}$.
Then $X,Y$ are binomial random variables and \eqref{eq:myplanted} shows that $\Erw[X+Y]<\delta n/(4\beta)$ if $\eta>0$ is chosen small enough.
Now, we couple the constraints that are non-adjacent to $\sigma_0\triangle\hat\SIGMA_{n,\vm}$ in either random CSP instance identically, and the at most $X+Y$ constraints that are adjacent to $\sigma_0\triangle\hat\SIGMA_{n,\vm}$ independently.
Hence, $\G',\G''$ differ in no more than $X+Y$ constraints.
Since the construction of the soften constraints $\psi_\beta$ ensures that the addition or removal of a single constraint can change the partition function by at most a factor of $\exp(\pm\beta)$, we conclude that $\abs{\ln Z_\beta( \G')-\ln Z_\beta( \G'')}\leq\beta(X+Y)$.
Since $\Erw[X+Y]<\delta n/(4\beta)$, (\ref{eqfact:GastTooBig3}) follows from the Chernoff bound.
Furthermore, (\ref{eqfact:GastTooBig3}) implies together with \Cor~\ref{Cor_intContig} that
	\begin{align}\label{eqfact:GastTooBig4}
	\pr\brk{\abs{\ln Z_\beta( \G')-\ln Z_\beta( \G'')}>\delta n}\leq\exp(-\Omega(n)).
	\end{align}

Let $m\in\cM_\eta(d)$.
We can couple $\G''=\G^*(n,\vm,\sigma_0)$ and $\G'''=\G^*(n,m,\sigma_0)$ such that both CSP instances coincide on $m\wedge\vm$ constraints.
Since $\vm$ is a Poisson variable, it is therefore exponentially unlikely that $\G'',\G'''$ differ on more than $\delta n/(2\beta)$ constraints, providing $\eta$ is small enough.
Consequently, 
	\begin{align}\label{eqfact:GastTooBig5}
	\pr\brk{\abs{\ln Z_\beta( \G'')-\ln Z_\beta( \G''')}>\delta n}\leq\exp(-\Omega(n))\qquad\mbox{uniformly for all }m\in\cM_\eta(d).
	\end{align}
Combining (\ref{eqfact:GastTooBig2}), \eqref{eqfact:GastTooBig4} and \eqref{eqfact:GastTooBig5}, we obtain
	\begin{align}\label{eqfact:GastTooBig6}
	\Erw \brk{ \ln Z_\beta(\G''') }&>n\bc{\ln q+(d\ln\xi_\beta)/k+4\delta}\qquad\mbox{uniformly for all }m\in\cM_\eta(d).
	\end{align}
Furthermore, since $\G'''$ consists of independent constraints drawn from the distribution \eqref{eq:myplanted} and because each of these constraints can shift $\ln Z_\beta(\G''') $ by no more than $\pm\beta$, Azuma's inequality and \eqref{eqfact:GastTooBig6} yield
	\begin{align}\label{eqfact:GastTooBig7}
	\pr\brk{\ln Z_\beta( \G''')\leq n\bc{\ln q+(d\ln\xi_\beta)/k+3\delta}}\leq\exp(-\Omega(n))\qquad\mbox{uniformly for all }m\in\cM_\eta(d).
	\end{align}
Finally, we couple $\G'''$ and $\G''''=\G(n,m,\hat\SIGMA_{n,m})$ just as in the proof of (\ref{eqfact:GastTooBig2}) to see that uniformly for all $m\in\cM_\eta(d)$,
	\begin{align}\label{eqfact:GastTooBig8}
	\pr\brk{\abs{\ln Z_\beta( \G''')-\ln Z_\beta( \G'''')}>\delta n\mid\hat\SIGMA_{n,m}\in\cA}
		\leq\exp(-\Omega(n)).
	\end{align}
Combining \Cor~\ref{Cor_intContig} with (\ref{eqfact:GastTooBig7}) and \eqref{eqfact:GastTooBig8}, we obtain (\ref{eq:ElnastGToobig}).
\end{proof}

\begin{proof}[Proof of Theorem \ref{Thm_cond}]
The first part of the theorem is immediate from \Lem~\ref{Lemma_lowerBound}.
With respect to the second assertion suppose that $d>\dc$ and fix $\beta,\delta,\eta$ as provided by \Lem~\ref{fact:GastTooBig}.
Then the events
	\begin{equation}\label{eqThm_cond_events}
	\cE_n=\cbc{\ln Z_\beta( \G(n,m))  < \ln q+\frac dk\ln\xi_\beta  + \delta}.
	\end{equation}
satisfy the assumptions of \Lem~\ref{lem:nthsquares}.
Since for any $\eta>0$ we have $\pr\brk{|\vm-dn/k|>\eta}\leq\exp(-\Omega(n))$, 
\Lem~\ref{lem:nthsquares} thus  shows that $\pr\brk{Z(\G)>q \xi^{d/k}-\eps}\leq\exp(-\eps n+o(n))$ for some $\eps>0$.
Because $\GG$ is distributed as $\G$ given $\fS$ and $\pr[\fS]=\Omega(1)$ by \Prop~\ref{prop:FirstCondOverFirst},
the second assertion follows.
\end{proof}

\subsection{Proof of Theorem \ref{thmContiquity}}
To prove the contiguity statement we first show that $\hat\G(n,m)$ and $\G^*(n,m,\SIGMA^*)$ are mutually contiguous.
More specifically, we have the following.

\begin{lemma}[\SYM, \BAL]\label{Prop_contig}
Let for any $D,\eps>0$ there exist $\delta>0$ and $n_0>0$ such that for all $n>n_0$ and all $m\leq Dn/k$ the following two statements are true.
\begin{enumerate}
\item If $\cE$ is an event such that $\pr\brk{(\G^*(n,m,\SIGMA^*),\SIGMA^*)\in\cE}<\delta$, then $\pr\brk{(\G^*(n,m,\hat\SIGMA_{n,m}),\hat\SIGMA_{n,m})\in\cE}<\eps$.
\item If $\cE$ is an event such that $\pr\brk{(\G^*(n,m,\hat\SIGMA_{n,m}),\hat\SIGMA_{n,m})\in\cE}<\delta$, then $\pr\brk{(\G^*(n,m,\SIGMA^*),\SIGMA^*)\in\cE}<\eps$.
\end{enumerate}
\end{lemma}

The proof of \Lem~\ref{Prop_contig} is identical to that of the soft constraint version  \cite[\Cor~4.8]{SoftCon}.
The details can be found in Appendix~\ref{Apx_Prop_contig}.

\begin{proof}[Proof of Theorem \ref{thmContiquity}]
We use a similar argument as in~\cite{SoftCon}, except that here we explicitly deal with the conditioning on $\fS$.
With respect to the first assertion, suppose that $d<\dcond$ and let $(\cE_n)_n $ be a sequence of events.
Let us first assume that $\pr\brk{\GG^*\in\cE_n}=o(1)$.
Then \Prop~\ref{prop:FirstCondOverFirst} implies that $\pr\brk{\G^*\in\cE_n\cap\fS}=o(1)$.
Thus, \Lem s~\ref{lem:nishimori} and~\ref{Prop_contig} yield $\pr\brk{\hat\G\in\cE_n\cap\fS}=o(1)$.
Furthermore,   \Thm~\ref{Thm_SSC} shows that for any $\eps>0$ there is $\delta>0$ such that 
	$\pr\brk{Z(\GG)<\delta q^n\xi^{\vm}}<\eps$ for large enough $n$, because $\cK>0$ almost surely.
Consequently,
	\begin{align*}
	\pr\brk{\GG\in\cE_n}&\leq\eps+\pr\brk{\GG\in\cE_n,\,Z(\GG)\geq\delta q^n\xi^{\vm}}
		=\eps+\pr\brk{\G\in\cE_n,\,Z(\G)\geq\delta q^n\xi^{\vm}\mid\fS}\\
		&=\eps+\frac{\pr\brk{\G\in\cE_n\cap\fS,\,Z(\G)\geq\delta q^n\xi^{\vm}}}{\pr\brk{\G\in\fS}}
		\leq\eps+\frac1{\delta\pr\brk{\G\in\fS}}\cdot\Erw\brk{\frac{\Erw[Z(\G)\vecone\{\G\in\cE\cap\fS\}\mid\vm]}{q^n\xi^{\vm}}}.
	\end{align*}
Hence, combining \eqref{eq:NishimoriG}, \Prop s~\ref{lem:FirstMoment} and~\ref{prop:FirstCondOverFirst}, we obtain a number $c=c(P,d)>0$ such that
for large $n$,
	\begin{align*}
	\pr\brk{\GG\in\cE_n}&\leq\eps+c\cdot\pr\brk{\hat\G\in\cE_n\cap\fS}.
	\end{align*}
Since this bound holds for any $\eps>0$ and because $\pr\brk{\hat\G\in\cE_n\cap\fS}=o(1)$, we conclude that $\pr\brk{\GG\in\cE_n}=o(1)$.

Conversely, assume that $\pr\brk{\GG\in\cE_n}=o(1)$.
Then $\pr\brk{\G\in\cE_n\cap\fS}=o(1)$.
Hence, as $\pr[\cZ(\hat\G)=Z(\hat\G)]=1-o(1)$ by \Prop~\ref{prop:belowcond-unif}, we obtain
	\begin{align}\nonumber
	\pr\brk{\hat\G\in\cE_n\cap\fS}&=o(1)+\pr\brk{\hat\G\in\cE_n\cap\fS,\,\cZ(\hat\G)=Z(\hat\G)}\\
		&\leq o(1)+\Erw\brk{\frac{\Erw\brk{\cZ(\G)\vecone\{\G\in\cE_n\cap\fS\}\mid\vm}}{\Erw[Z(\G)\mid\vm]}
			\cdot\vecone\cbc{\abs{\vm-dn/k}\leq\sqrt{n}\ln n}}\label{eqmycontig1}
	\end{align}
Further, the second moment bound from \Prop~\ref{lem:SecondMoment}
shows together with the formula for the first moment from \Prop~\ref{lem:FirstMoment} 
that on the event $ \cbc{\abs{\vm-dn/k}\leq\sqrt{n}\ln n}$ the quotient $\Erw[\cZ(\G)^2|\vm]/\Erw[Z(\G)|\vm]^2$ is bounded.
Hence, for any $\eps>0$ there is $C=C(\eps,P,d)>0$ such that
	$\Erw\brk{\vecone\{\cZ(\G)>C\Erw[Z(\G)]\}}<\eps$.
Therefore, (\ref{eqmycontig1}) yields
	\begin{align*}
	\pr\brk{\hat\G\in\cE_n\cap\fS}&\leq o(1)+\eps+C\cdot\pr\brk{\G\in\cE_n\cap\fS}.
	\end{align*}
Since this bound holds for every fixed $\eps>0$ and $\pr\brk{\G\in\cE_n\cap\fS}=o(1)$, we obtain $\pr\brk{\hat\G\in\cE_n\cap\fS}=o(1)$.
Finally, since $\G^*$ and $\hat\G$ are mutually contiguous by \Lem~\ref{Prop_contig} and since $\pr\brk{\hat\G\in\fS}=\Omega(1)$ by \Prop~\ref{prop:FirstCondOverFirst}, we obtain
$\pr\brk{\GG^*\in\cE_n\cap\fS}=\pr\brk{\G^*\in\cE_n\cap\fS\mid\fS}=o(1)$, as desired.

Now assume that $d>\dc$.
The events $\cE_n$ from \eqref{eqThm_cond_events} satisfy the assumptions of \Lem~\ref{lem:nthsquares}.
Thus, $\pr\brk{\G\in\cE_n}=1-\exp(-\Omega(n))$, while 
$\pr\brk{\hat\G\in\cE_n}=\exp(-\Omega(n))$.
Indeed, since $\pr\brk{\G\in\fS},\pr\brk{\hat\G\in\fS}=\Omega(1)$ by \Prop~\ref{prop:FirstCondOverFirst}, we conclude that
$\pr\brk{\GG\in\cE_n}=1-\exp(-\Omega(n))=1-o(1)$, while $\pr\brk{\GG^*\in\cE_n}=o(1)$ by \Lem~\ref{Prop_contig}.
Thus, $\GG$ and $\GG^*$ are mutually orthogonal.
\end{proof}

\subsection{Proof of \Cor~\ref{Cor_thmContiquity}}
Assume that $(\cE_n)_n$ is a sequence of events such that $\pr\brk{(\GG,\SIGMA)\in\cE_n}=o(1)$.
Then there exists a sequence $\eps_n=o(1)$ such that the events
	$\cE_n'=\{\bck{\vecone\{(\GG,\SIGMA)\in\cE_n\}}_{\GG}\geq\eps_n\}$ satisfy $\pr\brk{\GG\in\cE_n'}=o(1)$.
Hence, \Thm~\ref{thmContiquity} yields $\pr\brk{\GG^*\in\cE_n'}=\pr\brk{\G^*\in\cE_n'\mid\fS}=o(1)$ and thus $\pr\brk{\G^*\in\cE_n'\cap\fS}=o(1)$.
Therefore, applying \Lem~\ref{Prop_contig}, we obtain $\pr\brk{\hat\G\in\cE_n'\cap\fS}=o(1)$, whence
 \Lem~\ref{lem:nishimori} yields $\pr\brk{(\hat\G,\hat\SIGMA)\in\cE_n\cap\fS}=o(1)$.
Thus, applying \Lem~\ref{Prop_contig} a second time, we obtain  $\pr\brk{(\G^*,\SIGMA^*)\in\cE_n\cap\fS}=o(1)$.
Since the probability of $\fS$ is bounded away from $0$ by \Prop~\ref{prop:FirstCondOverFirst}, we finally obtain
$\pr\brk{(\GG^*,\SIGMA^*)\in\cE_n}=\pr\brk{(\G^*,\SIGMA^*)\in\cE_n\mid\fS}=o(1)$.

Conversely, assume that $\pr\brk{(\GG^*,\SIGMA^*)\in\cE_n}=o(1)$.
Then $\pr\brk{(\G^*,\SIGMA^*)\in\cE_n\cap\fS}=o(1)$ and thus \Lem~\ref{Prop_contig} yields 
$\pr[(\hat\G,\hat\SIGMA)\in\cE_n\cap\fS]=o(1)$.
Hence, \Lem~\ref{lem:nishimori} shows that there exists a sequence $\eps_n=o(1)$ such that for the event 
$\cE_n'=\{\langle\vecone\{(\hat\G,\SIGMA)\in\cE_n\}\rangle_{\hat\G}\geq\eps_n\}$ we have $\pr\brk{\hat\G\in\cE_n'\cap\fS}=o(1)$.
Thus, \Lem~\ref{Prop_contig} yields $\pr\brk{\G^*\in\cE_n'\cap\fS}=o(1)$ and therefore 
$\pr\brk{\GG^*\in\cE_n'}=o(1)$ by \Prop~\ref{prop:FirstCondOverFirst}.
Consequently, \Thm~\ref{thmContiquity} yields $\pr\brk{\GG\in\cE_n'}=o(1)$.
Finally, unravelling the definition of $\cE_n'$, we obtain $\pr\brk{(\GG,\SIGMA)\in\cE_n}=o(1)$.

\subsection{Proof of Theorem \ref{Thm_overlap}}
Suppose that $d<\dc$.
Then Proposition \ref{prop:belowcond-unif}, \Prop~\ref{prop:FirstCondOverFirst} and \Lem~\ref{Prop_contig} yield
	$\Erw\langle\tv{\rho_{\SIGMA_1,\SIGMA_2}-\bar\rho}\rangle_{\GG^*}=o(1).$
Hence, Theorem \ref{thmContiquity} implies $\Erw\langle\tv{\rho_{\SIGMA_1,\SIGMA_2}-\bar\rho}\rangle_{\GG}=o(1)$, as desired.

\section{Reconstruction and local weak convergence}\label{Sec_lwc}

\noindent
In this section we prove \Thm s~\ref{Thm_lwc} and~\ref{thrm:TreeGraphEquivalence}.
Recall that for a variable $x$ of a CSP instance $G$ we denote by $\neigh_{2\ell}(G,x)$ the depth-$2\ell$ neighbourhood of $x$, rooted at $x$.
Moreover, let $\partial^{2\ell}(G,x)$ be the set of variables at distance precisely $2\ell$ from $x$.
We drop $G$ from the notation where the reference is apparent.

\subsection{Proof of \Thm~\ref{Thm_lwc}}
We remember the random CSP $\TT=\TT(d,P)$ generated by a Galton-Watson process that describes the local neighbourhood structure and we continue to denote its root by $r$.
Moreover, we write $\TT^{2\ell}=\TT^{2\ell}(d,P)$ for the CSP instance obtained from $\TT$ by deleting all constraint and variables at a distance greater than $2\ell$ from $r$.
Due to condition {\bf SYM} the partition function $Z(\TT^{2\ell})$ is strictly positive.
Hence, throughout this section we denote by $\CHI^{2\ell}$ a sample from the Boltzmann distribution $\mu_{\TT^{2\ell}}$.
The following lemma shows that the Galton-Watson process $\TT$ also describes the local structure of the planted random CSP $\G^*$.

\begin{lemma}[\SYM]\label{Lemma_trees}
Let $\ell\geq1$.
For any possible outcome $T$ of $\TT^{2\ell}$ and for any assignment $\chi:V(T)\to\Omega$ the following is true.
Let $X$ be the number of variables of $\G^*$ for which there exists an isomorphism $\thet:T\to\nabla_{2\ell}(\G^*,x)$
such that $\chi=\SIGMA^*\circ\thet$.
Then $X/n$ converges  to $\pr\brk{\TT^{2\ell}\ism T,\CHI^{2\ell}=\chi}$ in probability.
\end{lemma}
\begin{proof}
Consider the following enhanced multi-type Galton-Watson process $(\hat\TT,\hat\CHI)$ whose types are variables $x$ endowed with values $\CHI(x)$ and constraints endowed with weight functions $\psi\in\Psi$ and indices $h\in[k]$.
The process starts from the tree $\hat\TT^0$ consisting of the root $r$ only, for which a value $\hat{\CHI}^0(r)\in\Omega$ is chosen uniformly at random.
Then $\hat\TT^{2\ell+2},\hat\CHI^{2\ell+2}$ is obtained by appending two more layers to $\hat\TT^{2\ell},\hat\CHI^{2\ell}$ as follows.
Each variable $x$ of $\hat\TT^{2\ell}$ at distance exactly $2\ell$ from $r$ independently generates $D=\Po(d)$ constraints $a_{x,1},\ldots,a_{x,D}$ as offspring.
The associated constraint functions $\psi_{a_{x,i}}$ are drawn independently from $P$, and the position $h_{x,i}$ where $x$ appears in the constraint $a_{x,i}$ is drawn uniformly from $[k]$, independently for every $i$.
Further, each constraint $a_{x,i}$ spawns $k-1$ variables $(y_{x,i,j})_{j\in[k]\setminus\{h_{x,i}\}}$.
Their values $\hat\CHI^{2\ell+2}(y_{x,i,j})$ are jointly drawn from the distribution
	\begin{align}\label{eqhatCHI}
	\pr\brk{\forall j\neq h_{x,i}:\hat\CHI^{2\ell+2}(y_{x,i,j})=\sigma_j \vert \hat\CHI^{2\ell}(x)}&=
		q^{1-k}\xi^{-1}\psi_{a_{x,i}}(\sigma_1,\ldots,\sigma_{h_{x,i}-1},\hat\CHI^{2\ell}(x),\sigma_{h_{x,i}+1},\ldots,\sigma_k)
			&(\sigma_j\in\Omega).
	\end{align}
In other words, the $\hat\CHI^{2\ell+2}(y_{x,i,j})$ are chosen with probability proportional to the weight induced by $\psi_{a,i}$ given that $x$ has value $\hat\CHI^{2\ell}(x)$.

Crucially, {\bf SYM} guarantees that the distributions of $(\hat\TT^{2\ell},\hat\CHI^\ell)$ and $(\TT^{2\ell},\CHI^{2\ell})$ coincide.
Indeed, it is immediate from the construction that $\hat\TT^{2\ell}$ is distributed precisely as $\TT^{2\ell}$.
Furthermore, {\bf SYM} ensures that the marginal distribution $\mu_{\TT^{2\ell},r}$ is uniform on $\Omega$.
Hence, induction on $\ell$ shows that $\CHI^{2\ell}$ satisfies the recurrence (\ref{eqhatCHI}) that gives rise to $\hat\CHI^{2\ell}$.
(One could say that the trees $\TT,\hat\TT$ satisfy a Nishimori identity.)

To complete the proof we set up a coupling of $\hat\TT^{2\ell},\hat\CHI^{2\ell}$ and the depth-$2\ell$ neighbourhood of variable $x_1$ of $\G^*$.
Because $\SIGMA^*$ is chosen uniformly at random, $\SIGMA^*(x_1)$ is uniformly distributed, just as $\hat\CHI^0(r)$.
Furthermore, a standard random hypergraph argument shows that the degree of each variable $x_i$ in $\G^*$ is asymptotically Poisson with mean $d$.
Therefore, {\bf SYM} and (\ref{eq:myplanted}) show that the constraints $\psi_{r,1},\ldots,\psi_{r,D}$ pending on $r$ are distributed just like the constraints pending on $r$ in the construction of $\TT^1$, up to an error of $o(1)$ in total variation distance.
This error stems from the fact that the degree $D$ of $r$ is asymptotically but not precisely a Poisson variable, and that some constraint may contain $r$ twice; the latter occurs with probability $O(1/n)$.
Further, (\ref{eq:myplanted}) also shows that the values under $\SIGMA^*$ of the variables at distance two from $r$ are asymptotically distributed according to \eqref{eqhatCHI} because $\SIGMA^*$ is {nearly balanced} \whp\
The coupling extends to the higher levels $\ell\geq1$ of the tree by induction.
\end{proof}

Consider the planted model $\hat\G$ and a sample $\SIGMA$ from its Boltzmann distribution.
For any finite value of $\ell$ there will likely be substantial dependencies between the values $(\SIGMA(y))_{y\in\partial^{2\ell}x_1}$ of the variables at distance precisely $2\ell$ from some reference variable, say $x_1$.
Indeed, these variables are `close' to $x_1$ and therefore their values are going to be correlated with the value $\SIGMA(x_1)$, and thus with each other.
In other words, the sub-CSP $\nabla^{2\ell}(\hat\G,x_1)$ induces dependencies between the variables $\partial^{2\ell}x_1$.
But are there additional correlations between these variables?
To answer this question we introduce the following notation.
For a set $U$ of variable/constraints of a CSP instance $G$ we let $\mu_{G\to U}$ be the Gibbs measure of the CSP from $G$ by deleting all constraints in $U$.
Thus, in particular $\hat\G\to \nabla^{2\ell}x_1$ is the sub-CSP obtained by deleting all constraints within a radius $2\ell$ of $x_1$.
The following proposition, which  constitutes the main step toward the proof of Theorem \ref{Thm_overlap}, shows that once we delete these constraints, the correlations between the variables $\partial^{2\ell}x_1$ disappear.

\begin{proposition}[\SYM, \BAL, \MIN]\label{Prop_cavity}
If $0<d<\dc$, then for every $\ell\geq1$ we have
	\begin{align}\label{eqProp_cavity}
	\lim_{n\to\infty}\Erw\brk{\sum_{\sigma:\partial^{2\ell}x_1\to\Omega}
		\abs{\mu_{\hat\G\to \nabla^{2\ell}x_1}(\sigma)-q^{-|\partial^{2\ell}x_1|}}}&=0.
	\end{align}
\end{proposition}

\noindent
To prove \Prop~\ref{Prop_cavity} we need a few preparations.

\begin{lemma}[\SYM, \BAL, \MIN]\label{Lemma_cavity}
Suppose that $0<d<\dc$ and let $m\in\cM(d)$.
There exists a sequence $\omega=\omega(n)\to\infty$ such that with probability at least $1-1/\omega$ the random factor graph $\G^*(n,m,\SIGMA^*)$ has the following two properties.
	\begin{enumerate}[(i)]
	\item $\mu_{\G^*(n,m,\SIGMA^*)}$ is $(1/\omega,\omega)$-symmetric.
	\item the marginals of $\mu_{\G^*(n,m,\SIGMA^*)}$ satisfy $\sum_{i=1}^n\TV{\mu_{\G^*,x_i}-\bar\rho}<n/\omega$.
	\end{enumerate}
\end{lemma}
\begin{proof}
This is immediate from  \Lem~\ref{Lemma_ovsym}, \Lem~\ref{lem:nishimori} and \Prop~\ref{prop:belowcond-unif}.
\end{proof}

\begin{corollary}[\SYM, \BAL, \MIN]\label{Cor_cavity}
Suppose that $0<d<\dc$, let $C>0$ and let $m\in\cM(d)$.
There exists a sequence $\omega=\omega(n)\to\infty$ such that with probability at least $1-1/\omega$ for all $\sigma\in\Omega^{V_n}$ with
$\tv{\rho_\sigma-\bar\rho}\leq C n^{-1/2}$ the following two statements hold.
	\begin{enumerate}[(i)]
	\item $\mu_{\G^*(n,m,\sigma)}$ is $(1/\omega,\omega)$-symmetric.
	\item the Gibbs marginals satisfy $\sum_{i=1}^n\TV{\mu_{\G^*(n,m,\sigma),x_i}-\bar\rho}<n/\omega$.
	\end{enumerate}
\end{corollary}
\begin{proof}
Suppose that $\sigma,\tau\in\Omega^{V_n}$ both satisfy $\tv{\rho_\sigma-\bar\rho}\leq C n^{-1/2}$
and let $\vm_\psi(\sigma),\vm_\psi(\tau)$ be the number of constraints endowed with the constraint function $\psi\in\Psi$ in $\G^*(n,m,\sigma)$ and $\G^*(n,m,\tau)$, respectively.
Then the vectors $(\vm_\psi(\sigma))_{\psi\in\Psi}$, $(\vm_\psi(\tau))_{\psi\in\Psi}$ are multinomially distributed.
Furthermore, because $\tv{\rho_\sigma-\rho_\tau}\leq 2Cn^{-1/2}$, condition {\bf SYM} and the characterisation (\ref{eq:myplanted}) of the planted model imply that $\abs{\Erw[\vm_\psi(\sigma)]-\Erw[\vm_\psi(\tau)]}=O(1)$.
Therefore, the local limit theorem for the multinomial distribution shows that $(\vm_\psi(\sigma))_{\psi\in\Psi}$, $(\vm_\psi(\tau))_{\psi\in\Psi}$ have total variation distance $o(1)$.
Consequently, there is a coupling such that these vectors coincide with probability $1-o(1)$.

Now, given that $\vm_\psi(\sigma)=\vm_\psi(\tau)$ for all $\psi$, we claim that the isomorphism classes of $\G^*(n,m,\sigma)$, $\G^*(n,m,\tau)$ are mutually contiguous.
Specifically, permuting the assignment $\tau$ suitably, we may assume that the symmetric difference $\sigma\triangle\tau$ contains no more than $2C\sqrt n$ variables.
Let $\cI$ be the event that all the variables in $\sigma\triangle\tau$ are isolated.
Given $\vm_\psi(\sigma)=\vm_\psi(\tau)$ for all $\psi$ and $\cI$, the factor graphs $\G^*(n,m,\sigma)$, $\G^*(n,m,\tau)$ are identically distributed (due to (\ref{eq:myplanted})).
Hence, it suffices to prove that the isomorphism classes of $\G^*(n,m,\sigma)$ and of the conditional CSP $\G^*(n,m,\sigma)$ given $\cI$ are mutually contiguous; of course the same construction will apply to $\G^*(n,m,\tau)$.

To derive this contiguity result let $\cJ$ be the event that $\G^*(n,m,\sigma)$ has at least $n^{2/3}$ isolated variables in each of the sets $\sigma^{-1}(\chi)$, $\chi\in\Omega$.
Because the constraints of $\G^*(n,m,\sigma)$ are chosen independently and $\sigma$ is nearly balanced, standard arguments show that $\cJ$ occurs with (very) high probability.
Hence, let $\G'$ denote the random CSP $\G^*(n,m,\sigma)$ given $\cJ$.
Then we construct a random factor graph $\G''\in\cI\cap\cJ$ as follows: choose a one-to-one map $\iota$ from the set $\sigma\triangle\tau$
to the set of isolated variables of $\G'$ such that $\sigma(\iota(x))=\sigma(x)$ for all $x$ uniformly at random.
Then obtain $\G''$ from $\G'$ by swapping the variables $x$ and $\iota(x)$ for all $x\in\sigma\triangle\tau$.
Clearly, $\G'$ and $\G''$ are isomorphic.
Moreover, with $I$ the number of isolated variables, we see that for every possible outcome $G$,
	$$\pr\brk{\G''=G\mid I(\G'')=I(G)}=\pr\brk{\G^*(n,m,\sigma)=G\mid I(\G^*(n,m,\sigma))=I(G),\,\cI}.$$
Finally, $I(\G'')$ and $I(\G^*(n,m,\sigma))$ given $\cI$ are mutually contiguous;
for both satisfy a local limit theorem with standard deviation $\Theta(\sqrt n)$, and their means differ by no more than $O(\sqrt n)$.
Since $\cJ$ occurs with high probability, we obtain the desired contiguity of the isomorphism classes of $\G^*(n,m,\sigma)$ and $\G^*(n,m,\sigma)$ given $\cI$.

In summary, for any $\sigma,\tau$ with $\tv{\rho_\sigma-\bar\rho},\tv{\rho_\tau-\bar\rho}\leq C n^{-1/2}$ we can couple the $\vm_\psi(\sigma),\vm_\psi(\tau)$ such that the isomorphism classes of $\G^*(n,m,\sigma)$, $\G^*(n,m,\tau)$ are mutually contiguous \whp\
To complete the proof, we simply observe that the event $\tv{\rho_{\SIGMA^*}-\bar\rho}\leq C n^{-1/2}$ occurs with a probability that is bounded away from zero by the central limit theorem.
Therefore, the assertion follows from \Lem~\ref{Lemma_cavity}.
\end{proof}

\begin{proof}[Proof of \Prop~\ref{Prop_cavity}]
By \Lem~\ref{lem:nishimori} and \Prop~\ref{Prop_contig} it suffices to prove (\ref{eqProp_cavity}) with $\hat\G$ replaced by $\G^*=\G^*(n,\vm,\SIGMA^*)$.
Indeed, fix a large number $C>0$ and let $\cA$ be the event that $\tv{\rho_{\SIGMA^*}-\bar\rho}\leq Cn^{-1/2}$.
Then the probability of the event $\cA$ is bounded away from $0$ and, in fact, approaches $1$ in the limit of large $C$.
Further, let $\G'$ be the factor graph obtained from $\G^*$ by deleting all variable and constraints at a distance less than $2\ell$ from $x_1$.
Let $\vn',\vm'$ be the number of variable and constraints of $\G'$ and let $\SIGMA'$ be the assignment induced by $\SIGMA^*$ on the set of variables of $\G'$.
Since $\ell$ is a fixed number, \Lem~\ref{Lemma_trees} implies that
	\begin{align}\label{eqProp_cavity1}
	\pr\brk{n+\vm-\vn'-\vm'\leq\ln n\mid\cA}&=1-o(1).
	\end{align}
In particular, if $n-\vn'\leq\ln n$ and if $\cA$ occurs, then $\tv{\rho_{\SIGMA'}-\bar\rho}\leq Cn^{-1/2}$
Moreover, since $\vm$ is a Poisson variable with mean $dn/k$, (\ref{eqProp_cavity1}) implies that
	\begin{align*}
	\pr\brk{\vm'-d\vn'/k\leq n^{3/5}\mid\cA}&=1-o(1).
	\end{align*}
Hence, recalling the definition of the set $\cM(d)$, we see that on $\cA$ we can apply \Cor~\ref{Cor_cavity} to $\G'$ with high probability.
Consequently, with high probability on $\cA$ the Gibbs measure
	\begin{equation}\label{eqProp_cavity3}
	\mu_{\G'}\mbox{ is $(1/\omega,\omega)$-symmetric and }\sum_{i=1}^n\TV{\mu_{\G',x_i}-\bar\rho}<n/\omega
	\end{equation}
for some $\omega\to\infty$.

To complete the proof let $\iota:\partial^{2\ell}_{\G^*}x_1\to V_n$ be a uniformly random map such that $\SIGMA^*(\iota(x))=\SIGMA^*(x)$ for all $x$.
Moreover, let $\G''$ be the random factor graph obtained from $\G^*$ by connecting the constraints at distance $2\ell-1$ from $x_1$ with the images
$\iota(x)$ instead of their original neighbours $x\in\partial^{2\ell}_{\G^*}x_1$.
Then the distribution of $\G''$ is identical to the distribution of $\G^*$.
Furthermore, since \Lem~\ref{Lemma_trees} implies that $|\partial^{2\ell}_{\G^*}x_1|\leq\omega^{1/2}$ with high probability,
(\ref{eqProp_cavity3}) yields
	\begin{align*}
	\Erw\brk{\TV{\mu_{\G',\partial^{2\ell}_{\G''}x_1}-\bar\rho}\mid\cA}&=o(1).
	\end{align*}
Thus, on $\cA$ with high probability the boundary condition $\mu_{\G',\partial^{2\ell}_{\G''}x_1}$ of the depth-$2\ell$ neighbourhood of $x_1$ is close in total variation distance to the free boundary condition, and therefore
	\begin{align}\label{eqProp_cavity2}
	\Erw\brk{\TV{\mu_{\G^*,\partial^{2\ell} x_1}-\mu_{\partial^{2\ell} x_1}}\mid\cA}&=o(1).
	\end{align}
Finally, since the probability of $\cA$ converges to $1$ as $C\to\infty$, the assertion follows from (\ref{eqProp_cavity2}).
\end{proof}

\begin{proof}[Proof of \Thm~\ref{Thm_lwc}]
The theorem is immediate from \Lem~\ref{Lemma_trees} and \Prop~\ref{Prop_cavity}.
\end{proof}

\subsection{Proof of \Thm~\ref{thrm:TreeGraphEquivalence}}
\Thm~\ref{thrm:TreeGraphEquivalence} is almost an immediate consequence of \Thm~\ref{Thm_lwc}, except that a bit of care is required to prove that $\dr^\star\leq\dc$.
To this end, we first show that $\mu_{\hat\G}$ is $\eps$-symmetric with mostly uniform marginals for $d<\dr^\star$.

\begin{lemma}[\SYM, \BAL]\label{Lemma_recsym}
Assume that $d<\dr^\star$ and let $\eps>0$.
Then with probability at least $1-\eps+o(1)$ the Boltzmann distribution $\mu_{\hat\G}$ is $\eps$-symmetric 
and its marginals satisfy $\sum_{i=1}^n\tv{\mu_{\hat\G,x_i}-\bar\rho}<\eps n$.
\end{lemma}
\begin{proof}
Fix a small enough $\delta>0$.
If $d<\dr^\star$, then there exists a bounded $\ell=\ell(\delta)>0$ such that 
	\begin{align*}
		\pr\brk{\exists\omega\in\Omega:\abs{\bck{\vecone\{\SIGMA(r)=\omega\}\big|
			\forall x\in\partial_{2\ell}(\TT(d,P),r):\SIGMA(x)=\CHI^{2\ell}(x)}
		-1/q}_{\TT^{2\ell}(d,P)}>\delta}&\leq\delta.
	\end{align*}
Hence, \Lem s~\ref{Prop_contig} and~\ref{Lemma_trees} show that for large enough $n$,
	\begin{align*}
		\pr\brk{\exists\omega\in\Omega:\abs{\bck{\vecone\{\SIGMA(x_1)=\omega\}\big|
			\forall y\in\partial_{2\ell}(\G^*(n,\vm,\hat\SIGMA),x_1):\SIGMA(y)=\hat\SIGMA(y)}
		-1/q}_{\G^*(n,\vm,\hat\SIGMA)}>\delta}&<\eps/2.
	\end{align*}
Therefore, the Nishimori identity \eqref{eq:nishimori} yields
	\begin{align}\label{eqTrick17}
	\pr\brk{\exists\omega\in\Omega:
			\bck{\abs{\bck{\vecone\{\SIGMA(x_1)=\omega\}\big|\overline{\nabla_{2\ell}(\hat\G,x_1)}}_{\hat\G}-1/q}}_{\hat\G}>\delta}<\eps/2.
	\end{align}
Now let $\cE$ be the event that $x_1,x_2$ have distance at least $2\ell+2$ in $\hat\G$ and that
	\begin{equation}\label{eqTrick18}
	\forall\omega\in\Omega:\bck{\abs{\bck{\vecone\{\SIGMA(x_1)=\omega\}\big|\overline{\nabla_{2\ell}(\hat\G,x_1)}}_{\hat\G}-1/q}}_{\hat\G}\leq\delta.
	\end{equation}
Since the average degree of $\hat\G$ is bounded \whp, (\ref{eqTrick17}) shows immediately that $\pr\brk{\cE}\geq1-2\eps/3+o(1)$.

But given $\cE$ it is immediate that $\tv{\mu_{\hat\G,x_1,x_2}-\bar\rho}<\eps/4$, provided that $\delta$ is small enough,
	an observation that goes back to~\cite{Mossel}.
Indeed, since $x_1,x_2$ have distance greater than $2\ell$, conditioning on the values of {\em all} variables at distance $2\ell$ from $x_1$ is stronger than just conditioning on the value of $x_2$.
Thus, we conclude that 
	\begin{equation}\label{eqTrick19}
	\pr\brk{\tv{\mu_{\hat\G,x_1,x_2}-\bar\rho}<\eps/4}\geq1-2\eps/3+o(1).
	\end{equation}
Finally, because the distribution of $\hat\G$ is invariant under permutations of the variables, 
	(\ref{eqTrick19}) yields the assertion.
\end{proof}

\begin{corollary}[\SYM, \BAL, \MIN, \POS ]\label{Cor_recsym}
We have $\dr^\star\leq\dc$.
\end{corollary}
\begin{proof}
Let $0<D<\dr^\star$ and let $\vx_1,\ldots,\vx_k$ denote uniformly and independently chosen variables.
Due to \eqref{eqprop:belowcond-unif} and \Prop~\ref{Prop_Deltat} we have, uniformly for all $d\leq D$,
	\begin{align}\label{eqCor_recsym1}
	\frac{k\xi}n\frac{\partial}{\partial d}\Erw[\ln Z(\hat\G)]&
		=f_n(d)+o(1)&\mbox{with }\qquad f_n(d)&=
			\Erw\brk{\Lambda\bc{\bck{\PSI(\SIGMA(\vx_{1}),\ldots,\SIGMA(\vx_{k}))}_{\hat\G}}}.
	\end{align}
We claim that for every $d\leq D$,
	\begin{align}\label{eqCor_recsym2}
	\lim_{n\to\infty}f_n(d)&=\Lambda(\xi).
	\end{align}
Indeed, plugging in the expansion $\Lambda(1-x)=-x+\sum_{\ell\geq2}x^\ell/(\ell(\ell-1))$, valid for all $x\in[-1,1]$, we obtain
	\begin{align}\nonumber
	f_n(d)&=1-\Erw\brk{\bck{\PSI(\SIGMA(\vx_{1}),\ldots,\SIGMA(\vx_{k}))}_{\hat\G}}+
		\sum_{\ell\geq2}\frac1{\ell(\ell-1)}\Erw\brk{\bck{1-\PSI(\SIGMA(\vx_{1}),\ldots,\SIGMA(\vx_{k}))}_{\hat\G}^\ell}\\
		&=1-\xi+o(1)+\sum_{\ell\geq2}\frac{\Erw\bck{\prod_{j=1}^\ell
			1-\PSI(\SIGMA_j(\vx_{1}),\ldots,\SIGMA_j(\vx_{k}))}_{\hat\G}}{\ell(\ell-1)}
			\qquad\mbox{[by \Lem~\ref{lem:nishimori}, \Cor~\ref{Cor_intContig} and {\bf SYM}]}.
			\label{eqCor_recsym3}
	\end{align}
Further, by \Lem~\ref{Lemma_recsym} there is a function $\eps_n(d)=o(1)$ such that $\mu_{\hat\G}$ is $\eps_n(d)$-symmetric with probability at least $1-\eps_n(d)$.
Therefore, \Lem~\ref{Lemma_prodsym} implies that the $\ell$-fold product measure $\mu_{\hat\G}^{\tensor\ell}$ is $o(1)$-symmetric \whp\
for any fixed $\ell>0$.
Hence, every $\ell\geq2$ we have \whp\
	\begin{align*}
	\bck{\prod_{j=1}^\ell
			1-\PSI(\SIGMA_j(\vx_{1}),\ldots,\SIGMA_j(\vx_{k}))}_{\hat\G}&
				=o(1)+\frac1{n^k}\sum_{\psi\in\Psi}\sum_{i_1,\ldots,i_k=1}^n\sum_{\sigma_1,\ldots,\sigma_\ell\in\Omega^k}
					P(\psi)\bc{\prod_{j=1}^\ell1-\psi(\sigma_j)}\prod_{h=1}^k\prod_{j=1}^\ell\mu_{\hat\G,x_h}(\sigma_{j,h}).
	\end{align*}
Thus, invoking the asymptotic uniformity of the Boltzmann marginals supplied by \Lem~\ref{Lemma_recsym} and applying {\bf SYM}, we see that \whp
	\begin{align}			\label{eqCor_recsym4}
	\bck{\prod_{j=1}^\ell
			1-\PSI(\SIGMA_j(\vx_{1}),\ldots,\SIGMA_j(\vx_{k}))}_{\hat\G}&
				=(1-\xi)^\ell+o(1)
	\end{align}
Combining (\ref{eqCor_recsym3}) and (\ref{eqCor_recsym4}), we obtain (\ref{eqCor_recsym2}).	

Finally, (\ref{eqCor_recsym1}), \eqref{eqCor_recsym2} and dominated convergence yield
	\begin{align*}
	\lim_{n\to\infty}\frac{1}n\Erw[\ln Z(\hat\G(n,\vm_D)]&=
		\ln q+\lim_{n\to\infty}\frac{1}n\int_0^d\frac{\partial}{\partial D}\Erw[\ln Z(\hat\G)]\dd d
		=\ln q+\frac1{k\xi}\lim_{n\to\infty}\int_0^Df_n(d)\dd d=\ln q+D\ln(\xi)/k.
	\end{align*}
Hence,  \Thm~\ref{Thm_planted} shows that $\dc\geq D.$ Since this holds for any $D<\dr$, the assertion follows. 
\end{proof}

\begin{proof}[Proof of \Thm~\ref{thrm:TreeGraphEquivalence}]
\Cor~\ref{Cor_recsym} shows that $\dr^\star\leq\dc$.
Thus, we are left to show that $\dr^\star=\dr$.
To prove that $\dr^\star\leq\dr$ suppose that $d<\dr^\star$.
Then~(\ref{def:PlantedCorr}) ensures that for any $\eps>0$ there is $\ell$ such that \whp\ we have 
	\begin{equation}\label{eqthrm:TreeGraphEquivalence1}
	\Erw\bck{\abs{\bck{\vecone\{\SIGMA(r)=\omega\}\big|\overline{\nabla_{2\ell}(\TT^{2\ell}(d,P),r)}}_{\TT^{2\ell}(d,P)}-1/q}}_{\TT^{2\ell}(d,P)}
		<\eps.
	\end{equation}
Further, \Thm~\ref{Thm_lwc} shows that $\tv{\mu_{\GG,\nabla_{2\ell}(\GG,x_1)},\mu_{\GG,\nabla_{2\ell}(\GG,x_1)}}=o(1)$ \whp\
Moreover, by \Thm~\ref{thmContiquity} and \Lem~\ref{Lemma_trees} the distribution of the neighbourhood $\nabla_{2\ell}(\GG,x_1)$ is
at total variation distance $o(1)$ of the distribution of the random tree $\TT^\ell(d,P)$.
Therefore, (\ref{eqthrm:TreeGraphEquivalence1}) shows that $\corr(d)\leq\eps$.
Since this is true for any $\eps>0$, we conclude that $d\leq\dr$.

Conversely, assume that $d<\dr\leq\dc$.
Then we can just put the argument from the previous paragraph in reverse.
Indeed, \Thm~\ref{thmContiquity} and \Lem~\ref{Lemma_trees} show that the neighbourhood $\nabla_{2\ell}(\GG,x_1)$ is
the distribution as $\TT^{2\ell}(d,P)$, up to $o(1)$ in total variation.
Further, for any $\eps>0$ there exists $\ell$ such that
	\begin{align*}
	\sum_{\omega\in\Omega}
		\Erw
			\bck{\abs{\bck{\vecone\{\SIGMA(x_1)=\omega\}\big|\overline{\nabla_{2\ell}(\GG,x_1)}}_{\GG}-1/q}}_{\GG}<\eps+o(1).
	\end{align*}
because $\corr(d)=0$.
Hence, \Thm~\ref{Thm_lwc} yields $\corr^\star(d)\leq\eps$.
Finally, because this bound holds for any $\eps>0$ we obtain $\corr^\star(d)=0$.
\end{proof}

\subsection*{Acknowledgment}
We thank Chris Brzuska for bringing~\cite{applebaum2018algebraic} and the parity-majority problem from \Sec~\ref{Sec_Parity-Majority} to our attention.
We also thank Charilaos Efthymiou and Will Perkins for helpful discussions.
Finally, we thank Peter Ayre and Catherine Greenhill for sharing~\cite{Ayre3}.

\begin{appendix}

\section{Proof of \Lem~\ref{Lemma_ovsym}}\label{Apx_epssymm}

\noindent
To establish \Lem~\ref{Lemma_ovsym} we will utilise regularity results for discrete probability measures from \cite{Victor}.
For $\epsilon > 0$ choose $\eta= \eta(\epsilon) >0$ and $ n>1/\eta$ sufficiently large. By {\cite[Corollary 2.2]{Victor}}, for any $ \mu \in \cP(\Omega^n)$, there exist $L \in \NN$, $ \mu^{(0)}, \dots, \mu^{(L)} \in \cP(\Omega^n)$ and $w_0, w_1, \dots, w_L$ such that we can decompose $\mu = \sum_{i = 0}^L w_i \mu^{(i)} $ and 
\begin{enumerate}
\item[(i)] $\mu^{(1)}, \dots, \mu^{(L)}$ are $\eta$-symmetric,
\item[(ii)] $w_0, \ldots, w_L \geq 0$, $\sum_{i=0}^L w_i = 1$, $\sum_{i=1}^L w_i \geq 1-\eta$ and
\item[(iii)] $w_i \geq \eta/L $ for all $i\in [L].$
\end{enumerate}
Let us use the shorthand notation $\bck{\,\cdot\,}_i = \bck{\,\cdot\,}_{\mu^{(i)}}$ and note that $\tv{ \,\cdot\, }$ and $ \| \cdot \|_2$ are equivalent norms in $\RR^{q^2}$. For $i \in [L]$, we have
\begin{align}
\bck{\| \rho_{\SIGMA,\TAU}-\bar\rho \|_2^2}_{i} &= \sum_{s,t \in \Omega} \frac 1{n^2} \sum_{v,w \in [n]} \mu^{(i)}_{v,w}(s,s) \mu^{(i)}_{v,w}(t,t) -q^{-2} \nonumber \\
&= \sum_{s,t \in \Omega} \frac 1{n^2} \brk{ \sum_{v,w \in [n]} \mu^{(i)}_{v,w}(s,s) \mu^{(i)}_{v,w}(t,t) - \bc{\sum_{v \in [n]} \mu^{(i)}_{v}(s)\mu^{(i)}_{v}(t)}^2 } + \brk{ \sum_{s,t \in \Omega}   \bc{\frac 1{n}\sum_{v \in [n]} \mu^{(i)}_{v}(s)\mu^{(i)}_{v}(t)}^2 -q^{-2} }.\label{eq:Ovlapsym2}
\end{align}
Combining (i) and \Lem~\ref{Lemma_prodsym} yields that $\mu^{(i)} \tensor \mu^{(i)} $ is $\zeta$-symmetric for a suitable $\zeta = \zeta (\eta) >0.$  Thus, the first summand of \eqref{eq:Ovlapsym2} is $O(\zeta).$ Now assume that $\bck{\tv{\rho_{\SIGMA,\TAU}-\bar\rho}}_{\mu} < \delta$ for $\delta(\eta, \zeta)>0$ sufficiently small. Due to (iii) and Jensen's inequality, $\bck{\| \rho_{\SIGMA,\TAU}-\bar\rho \|_2^2}_{i}<\sqrt{\delta/\eta}$ and consequently, \eqref{eq:Ovlapsym2} implies that for all $s,t \in \Omega$ we have 
\begin{align*}
\left|  \frac 1{n}  \sum_{v \in [n]} \mu^{(i)}_{v}(s)\mu^{(i)}_{v}(t) -q^{-2} \right|\leq O\bc{\zeta^{1/2}}.
\end{align*}
Hence for all $s \in \Omega $ we have 
\begin{align}\label{eq:Ovlapsym3}
\left|  \frac 1{n}  \sum_{v \in [n]} \bc{\mu^{(i)}_{v}(s)}^2 -q^{-2} \right| &\leq O\bc{\zeta^{1/2}}, &\left| \frac 1{n}  \sum_{v \in [n]} \mu^{(i)}_{v}(s) -q^{-1} \right| &\leq O\bc{\zeta^{1/2}}.
\end{align}
As the sum of squares is minimised by a uniform distribution, Taylor expanding the function $f \bc{ \bc{\mu^{(i)}(s)}_{s \in \Omega}} = \frac 1n \bc{ \sum_{v \in [n]} \bc{\mu^{(i)}(s)}^2}_{s \in \Omega}$ around $q^{-1} \vecone_{q\times n}$ together with \eqref{eq:Ovlapsym3} yields  
\begin{align*}
\left| \frac 1n O(\| \mu^{(i)} - q^{-1}\vecone_{q\times n} \|_2^2) \right| \leq O(\zeta^{1/2}).
\end{align*}
Thus, for all $i \in [L]$ we have 
\begin{align}\label{eq:Ovlapsym4}
\frac 1{n}  \sum_{v \in [n]} \TV{ \mu^{(i)}_{v}( \cdot ) -q^{-1}\vecone }< \zeta^{1/5}.
\end{align}
The $\epsilon$-symmetry of $\mu$ now follows from \eqref{eq:Ovlapsym4} and \cite[Lemma 2.8]{Victor}. Moreover, equation \eqref{eq:Ovlapsym4} and (ii) imply 
\begin{align*}
\frac 1{n}  \sum_{v \in [n]} \TV{ \mu_{v}( \cdot ) -q^{-1}\vecone }< \epsilon.
\end{align*}

We now turn to the converse implication. First, a calculation as in \ref{eq:Ovlapsym2} yields that
\begin{align} \label{eq:Ovlapsym5}
\bck{\| \rho_{\SIGMA,\TAU}-\bar\rho \|_2}^2_{\mu} \leq \bck{\| \rho_{\SIGMA,\TAU}-\bar\rho \|_2^2}_{\mu} \leq \bc{1+\frac{1}{q}} \frac{2}{n^2} \sum_{v,w \in [n]} \tv{\mu_{v,w}-\bar{\rho}}.
\end{align}
Secondly, we bound
\begin{align} \label{eq:Ovlapsym6}
\frac{1}{n^2} \sum_{v,w \in [n]} \tv{\mu_v\otimes \mu_w - \bar{\rho}} &\leq \frac{1}{2n^2} \sum_{v,w \in [n]} \sum_{s,t \in \Omega}\bc{\left|\mu_v(s)-\frac{1}{q} \right|\left|\mu_w(t)-\frac{1}{q} \right| + \frac{1}{q}\bc{\left| \mu_v(s)-\frac{1}{q}\right| + \left| \mu_w(t)-\frac{1}{q}\right|}} \nonumber \\
& = 2 \bc{\frac{1}{n} \sum_{v \in [n]}\tv{\mu_v-\bar{\rho}}}^2 + \frac{2}{n} \sum_{v \in [n]}\tv{\mu_v-\bar{\rho}}.
\end{align}
Now, inequalities (\ref{eq:Ovlapsym5}), (\ref{eq:Ovlapsym6}) and the triangle inequality imply that by choosing $\delta>0$ small enough, we have that any $\delta$-symmetric $\mu$ with $1/n \sum_{v \in [n]}\tv{\mu_v-\bar{\rho}} <\delta$ satisfies $\bck{\tv{\rho_{\SIGMA,\TAU}-\bar\rho}}_{\mu}<\eps$. 

\section{Moment calculations}\label{Apx_moments}

\noindent
The proofs of \Prop s~\ref{lem:FirstMoment} and~\ref{lem:SecondMoment} are straightforward applications of the Laplace method; the calculations are identical to those performed in \cite[\Sec~7]{SoftCon}.

\subsection{Proof of \Prop~\ref{lem:FirstMoment}}\label{Sec_moments1}
Let $R_n$ be the set of all distributions $\rho\in\cP(\Omega)$ such that the vector $n\rho\in\RR^\Omega$ has integer entries.
For $\rho\in R_{n}$ let
	$Z_{\rho}(\G(n,m)) =Z(\G(n,m))\langle\vecone\{\rho_{\SIGMA}=\rho\}\rangle_{\G(n,m)}$
be the number of satisfying assignments $\sigma\in\cS(\G(n,m))$ with empirical distribution $\rho_\sigma=\rho$, so that
	\begin{align}\label{eqProp_genFirstMmt1}
	\Erw[Z(\G(n,m))]&=\sum_{\rho\in R_{n}}\Erw[Z_\rho(\G(n,m))].
	\end{align}
Since the total number of assignments $\sigma\in\Omega^{V_n}$ with empirical distribution $\rho$ is given by the multinomial coefficient
$\bink n{n\rho}$ and because the $m$ constraints of $\G(n,m)$ are chosen independently, we can express the mean $\Erw[Z_\rho(\G(n,m))]$
easily in terms of the function $\phi$ from condition {\bf BAL}.
Namely,
	\begin{align}\label{eqProp_genFirstMmt1_a}
	\Erw[Z_\rho(\G(n,m))]&=\bink n{\rho n}\phi(\rho)^m.
	\end{align}
Further, because by {\bf BAL} both the multinomial coefficient and the function $\phi(\rho)$ take their maximum at the uniform distribution $\bar\rho$,
the contribution of the summands from the set $R_n'=\{\rho\in R_n:\|\rho-\bar\rho\|_2<n^{-1/2}\ln n\}$ dominates.
Thus, approximating the multinomial coefficient in (\ref{eqProp_genFirstMmt1_a}) via Stirling's formula, we obtain from (\ref{eqProp_genFirstMmt1})
	\begin{align}\label{eqProp_genFirstMmt1a}
	\Erw[Z(\G(n,m))]&\sim\sum_{\rho\in R_{n}'}\Erw[Z_\rho(\G(n,m))]
		\sim\sum_{\rho\in R_{n}'}\frac{\exp(n f_n(\rho))}{\sqrt{(2\pi n)^{q-1}\prod_{\omega\in\Omega}\rho(\omega)}},\qquad\mbox{where}\\
	f_n(\rho)&=\cH(\rho)+\frac mn\ln\phi(\rho).\nonumber
	\end{align}
The gradient and the Hessian of the function $f_n(\rho)$ at $\rho=\bar\rho$ are computed easily.
Indeed, using {\bf SYM} and \Lem~\ref{define_phi} we obtain
	\begin{align}\label{eqProp_genFirstMmt4}
	D f_n(\bar\rho)&=(\ln(q)-1+km/n)\vecone,&
		D^2f_n(\bar\rho)&=-q(\id-(k(k-1)m/n)\Phi)+(k^2m/n)\vecone,
	\end{align}
and all third partial derivatives of $f_n$ are uniformly bounded on $R_n'$.
Furthermore, for all $\rho\in R_n'$ we have $\vecone\perp\rho-\bar\rho$ because $\bar\rho,\rho'$ are probability distributions on $\Omega$.
Hence, 
	\begin{align*}
	f_n(\rho)&=f_n(\bar\rho)-q\scal{(\id-(k(k-1)m/n)\Phi)(\rho-\bar\rho)}{(\rho-\bar\rho)}+O(n^{-3/2}\ln^3n)&\mbox{uniformly for all }\rho\in R_n'.
	\end{align*}
Thus, (\ref{eqProp_genFirstMmt1a}) boils down to	
	\begin{align}\nonumber
	\Erw[Z(\G(n,m))]&
		\sim\frac{q^{q/2}\exp(n f(\bar\rho))}{(2\pi n)^{(q-1)/2}}
			\sum_{\rho\in R_{n}'}\exp\brk{-qn\scal{(\id-(k(k-1)m/n)\Phi)(\rho-\bar\rho)}{(\rho-\bar\rho)}}\\
		&=\frac{q^{n+\frac12}\xi^m}{(2\pi n/q)^{(q-1)/2}}\sum_{\rho\in R_{n}'}\exp\brk{-qn\scal{(\id-(k(k-1)m/n)\Phi)(\rho-\bar\rho)}{(\rho-\bar\rho)}}\qquad\mbox{[due to {\bf SYM}]}.\label{eqProp_genFirstMmt1b}
	\end{align}
By \Lem~\ref{Lemma_Phi} the matrix $\Phi$ has precisely one positive eigenvalue, namely $1$, with corresponding eigenvector $\vecone$.
Since in (\ref{eqProp_genFirstMmt1b}) we sum only over $\rho$ such that $\rho-\bar\rho\perp\vecone$, we can approximate the sum by a Gaussian integral over the $(q-1)$-dimensional orthogonal complement of $\vecone$ in $\RR^q$ to obtain
	\begin{align}\nonumber
	\sum_{\rho\in R_{n}'}\exp\brk{-qn\scal{(\id-(k(k-1)m/n)\Phi)(\rho-\bar\rho)}{(\rho-\bar\rho)}}
		&\sim\int_{\RR^{q-1}}\exp\bc{-qn\sum_{\lambda\in\eig[\Phi]\setminus\{1\}}(1-(k(k-1)m/n)\lambda)z_i^2}\dd z\\
		&\sim\frac{(2\pi n/q)^{(q-1)/2}}{\prod_{\lambda\in\eig[\Phi]\setminus\{1\}}\sqrt{1-d(k-1)\lambda}}.
			\label{eqProp_genFirstMmt1c}
	\end{align}
Combining (\ref{eqProp_genFirstMmt1b}) and (\ref{eqProp_genFirstMmt1c}) completes the proof.

\subsection{Proof of \Prop~\ref{lem:SecondMoment}}\label{Sec_moments2}
Let $R_n$ be the set of all distributions $\rho\in\cP(\Omega\times\Omega)$ such that $n\rho\in\RR^{\Omega\times\Omega}$ is integral
and such that $\tv{\rho-\bar\rho}\leq\zeta$.
Let $\cZ_\rho(\G(n,m))$ be the number of pairs $(\sigma_1,\sigma_2)\in\cS(\G(n,m))$ with overlap $\rho_{\sigma_1,\sigma_2}=\rho$.
Recalling the definition of $\cZ$ from~\eqref{eqcZeps}, we get 
	\begin{align}\label{eqsm1}
	\Erw[\cZ(\G(n,m))^2]&=
		\Erw\brk{Z(\G(n,m))^2\vecone\cbc{\bck{\TV{\rho_{\SIGMA_1,\SIGMA_2}-\bar\rho}}_{\G(n,m)}\leq\zeta}}
		=\sum_{\rho\in R_n}\Erw[\cZ_\rho(\G(n,m))].
	\end{align}
Clearly, the total number of pairs $(\sigma_1,\sigma_2)\in\Omega^{V_n}$ with overlap $\rho$ equals $\bink n{\rho n}$.
Hence, recalling the function $\varphi$ from condition {\bf MIN}, using the independence of the constraints of $\G(n,m)$
and applying Stirling's formula, we obtain
	\begin{align}\label{eqsm2}
	\Erw[\cZ_\rho(\G(n,m))]&=\bink n{\rho n}\varphi(\rho)^m
		\sim
		\sum_{\rho\in R_{n}}\frac{\exp(n f_n(\rho))}{\sqrt{(2\pi n)^{q^2-1}\prod_{\omega,\omega'\in\Omega}\rho(\omega,\omega')}},\qquad\mbox{where}\\
	f_n(\rho)&=\cH(\rho)+\frac mn\ln\varphi(\rho).\nonumber
	\end{align}
Once more it is straightforward to calculate the gradient and the Hessian of $f_n$ at the point $\bar\rho$:
condition {\bf SYM} yields
	\begin{align}\label{eqsm3}
	D f_n(\bar\rho)&=(2\ln(q)-1+km/n)\vecone,&
		D^2f_n(\bar\rho)&=-q^{2}(\id-(k(k-1)m/n)\Xi)+(k^2m/n)\vecone,
	\end{align}
and all third partial derivatives are uniformly bounded.
Consequently, since $\vecone\perp\rho-\bar\rho$ for all  $\rho\in R_n$, we obtain
	\begin{align*}
	f_n(\rho)&=f_n(\bar\rho)-q^2\scal{(\id-(k(k-1)m/n)\Xi)(\rho-\bar\rho)}{(\rho-\bar\rho)}+O(\tv{\rho-\bar\rho}^3)&\mbox{uniformly for all }\rho\in R_n.
	\end{align*}
Hence, (\ref{eqsm2}) becomes
	\begin{align}
	\Erw[\cZ(\G(n,m))^2]
		&\sim\frac{q^{2n+1}\xi^m}{(2\pi n/q^2)^{(q^2-1)/2}}\sum_{\rho\in R_{n}'}\exp\brk{-q^2n\scal{(\id-(k(k-1)m/n)\Xi)(\rho-\bar\rho)}{(\rho-\bar\rho)}}.\label{eqsm4}
	\end{align}
Since $d<\dc$, \Lem~\ref{Lemma_Xi} and \Prop~\ref{prop_KS} show that $\vecone$ is the only eigenvector of $\id-(k(k-1)m/n)\Xi$ with a non-negative eigenvalue.
Consequently, because the sum only ranges over $\rho$ such that $\vecone\perp\rho-\bar\rho$, we can approximate the sum by a Gaussian integral:
	\begin{align}\nonumber
	\sum_{\rho\in R_{n}'}\exp\brk{-q^2n\scal{(\id-(k(k-1)m/n)\Xi)(\rho-\bar\rho)}{(\rho-\bar\rho)}}
		&\sim\int_{\RR^{q^2-1}}\exp\bc{-q^2n\sum_{\lambda\in\eig'[\Xi]}(1-(k(k-1)m/n)\lambda)z_i^2}\dd z\\
		&\sim\frac{(2\pi n/q^2)^{(q^2-1)/2}}{\prod_{\lambda\in\eig'[\Xi]}\sqrt{1-d(k-1)\lambda}}.
			\label{eqsm5}
	\end{align}
Thus, the assertion follows from (\ref{eqsm4}) and (\ref{eqsm5}).

\section{Proof of \Lem~\ref{Lemma_mon}}\label{Sec_Lemma_mon}

\noindent
In order to prove Lemma \ref{Lemma_mon}, we first establish a uniform upper bound on the total variation distance of $\hat \SIGMA_{n,m}$ and $\hat \SIGMA_{n,m'}$ for $m$ and $m'$ that are not too far from $dn/k$.

\begin{lemma}[\SYM, \BAL]\label{Coupling_of_Hats}
For any $\eta>0, d >0$ there is $\delta>0$ such that
\begin{align}\label{TVforCoupling}
\limsup_{n \to \infty} \max\cbc{d_{\mathrm{TV}}\cbc{\hat \SIGMA_{n,m}, \hat \SIGMA_{n,m'}}: |m-dn/k| + |m'-dn/k| < \delta n} < \eta.
\end{align}
\end{lemma}
\begin{proof}
 Fix $\eta >0, d>0$ and recall the function $\phi$ from condition {\bf BAL}. 
Lemma \ref{Cor_F} shows that there exists $c>0$ such that for all $0<\delta<1$ and all $m,m' \leq \bc{d/k+\delta}n$ the bounds
\begin{align}\label{TVforCoupling1}
c \bc{\frac{\phi(\rho_\sigma)}{\xi}}^{m-m'} \leq \frac{\pr\brk{\hat \SIGMA_{n,m} = \sigma}}{\pr\brk{\hat \SIGMA_{n,m'} = \sigma}}\leq \frac{1}{c} \bc{\frac{\phi(\rho_\sigma)}{\xi}}^{m-m'}
\end{align}
are valid.
Moreover, Corollary \ref{lem:conc_coloring} yields $C>0$ such that for all $m,m' \leq \bc{d/k+\delta}n$ we have
	\begin{align}\label{big_TV}
	\pr\brk{\TV{\rho_{\hat \SIGMA_{n,m}} - \bar{\rho}} > C/\sqrt{n}} + \pr\brk{\TV{\rho_{\hat \SIGMA_{n,m'}} - \bar{\rho}} > C/\sqrt{n}} \leq \eta/4.
	\end{align}
Further, suppose that $\sigma \in \Omega^{V_n}$ satisfies $\TV{\rho_{\sigma} - \bar{\rho}} \leq C/\sqrt{n}$.
Because {\bf BAL} ensures that the first derivative of $\phi$ vanishes at $\bar\rho$, we have
	\begin{align}\label{TVforCoupling2}
	\phi(\rho_\sigma)&= \phi(\bar{\rho}) + O\bc{\TV{\rho_{\sigma} - \bar{\rho}}^2} = \xi + O\bc{n^{-1}}.
	\end{align}
Combining (\ref{TVforCoupling1}) and (\ref{TVforCoupling2}), we obtain $c_1,c_2 > 0$ such that for all $\sigma$ satisfying $\TV{\rho_{\sigma} - \bar{\rho}} \leq C/\sqrt{n}$, for all $0<\delta<1$ and for all $m,m'$ satisfying $|m-dn/k| + |m'-dn/k| < \delta n$ the estimates
	\begin{align}\label{expo_bound}
	c_1\exp\bc{-\delta c_2}  \leq \frac{\pr\brk{\hat \SIGMA_{n,m} = \sigma}}{\pr\brk{\hat \SIGMA_{n,m'} = \sigma}}\leq\exp\bc{\delta c_2}/c_1
	\end{align}
hold.
Finally, the assertion follows from (\ref{big_TV}) and (\ref{expo_bound}) by choosing $\delta>0$ sufficiently small.
\end{proof} 

\begin{proof}[Proof of Lemma \ref{Lemma_mon}]
Assume that $d,\eps>0$ and $m \in \cM(d)$ are such that $\limsup_{n\to\infty}\Erw\bck{\TV{\rho_{\SIGMA_1,\SIGMA_2}-\bar\rho}}_{\hat\G(n,m)}>\epsilon.$
Choose $\eta=\eta(\eps)>0$ small enough and then pick $\delta=\delta(\eta)>0$ as in Lemma \ref{Coupling_of_Hats}.
By assumption, there exist infinitely many $n$ such that $|m-dn/k|<\delta n/2$ and
	\begin{align}\label{eqLemma_mon_ass2}
	\Erw\bck{\TV{\rho_{\SIGMA_1,\SIGMA_2}-\bar\rho}}_{\hat\G(n,m)}>\eps/2
	\end{align}
Fix a large enough such $n$ along with $m'$ such that $m<m'< dn/k + 2\delta n$.
We are going to argue that 
	$$\Erw\bck{\TV{\rho_{\SIGMA_1,\SIGMA_2}-\bar\rho}}_{\hat\G(n,m')}>\delta.$$

By Lemma \ref{Coupling_of_Hats}, there is a coupling of $\hat \SIGMA_{n,m}, \hat \SIGMA_{n,m'}$ such that
\begin{align}\label{highproba}
\pr\brk{\hat \SIGMA_{n,m} =  \hat \SIGMA_{n,m'}} > 1 - \eta.
\end{align}
We extend this coupling to a coupling of
$\G'\stacksign{$\mathrm d$}=\G^*(n,m, \hat \SIGMA_{n,m})$ and $\G''\stacksign{$\mathrm d$}=\G^*(n,m', \hat \SIGMA_{n,m'})$ in the natural way.
Specifically, given $\hat \SIGMA_{n,m} =  \hat \SIGMA_{n,m'}$ we draw $\G''$ with constraints $a_1,\ldots,a_{m'}$ from the distribution $\G^*(n,m', \hat \SIGMA_{n,m'})$ and we let $\G'$ simply be the factor graph comprising the first $m$ constraints $a_{1},\ldots,a_{m}$.
Moreover, if $\hat \SIGMA_{n,m}\neq \hat \SIGMA_{n,m'}$, then we draw $\G',\G''$ independently from their respective marginal distributions.

Due to (\ref{eqLemma_mon_ass2}) and the Nishimori identity (\ref{eq:NishimoriS}) we have
$\pr\brk{\bck{\TV{\rho_{\SIGMA, \vec \tau} - \bar{\rho}}}_{\G'} >  \eps/2} \geq \eps/2$.
Recalling the notion of \textit{nearly balanced} below Lemma \ref{Lemma_multiOverlap}, we observe that because a random sample $\vec \tau$ from $\mu_{\G'}$ and $\hat \SIGMA_{n,m}$ are identically distributed, $\vec \tau$ is nearly balanced with probability $1-o(1)$  by Corollary \ref{lem:conc_coloring}.
Hence, provided that $n$ is large enough, with probability at least $\eps/3$ the random factor graph $\G'$ possesses a nearly balanced satisfying assignment $\tau_G$ such that $\langle\tv{\rho_{\SIGMA, \tau_{\G'}} - \bar{\rho}}\rangle_{\G''}>\eps/2$.
In the event that there is no such event we just let $\tau_{\G'}$ be an arbitrary nearly balanced assignment (not necessarily a satisfying one).
This construction ensures that
	$$\Erw\brk{\bck{\TV{\rho_{\SIGMA, \tau_{\G'}} - \bar{\rho}}}_{\G'} } \geq \eps^2/6.$$
Hence, provided that $\eta$ was chosen small enough (\ref{highproba}) and the Nishimori identity (\ref{eq:NishimoriS}) yield
	\begin{align}\label{eq_mon_1}
	\Erw\brk{\TV{\rho_{\hat \SIGMA_{n,m}, \tau_{\G'}} - \bar{\rho}} \Big \vert \hat \SIGMA_{n,m} =  \hat \SIGMA_{n,m'}} \geq \eps^2/7.
	\end{align}

Finally, we also designate a nearly balanced assignment $\tilde\tau_{\G''}$ for the factor graph $\G''$ by simply letting $\tilde\tau_{\G''}$ be the assignment $\tau_{\G'''}$ of the factor graph $\G'''$ obtained from $\G''$ by deleting the last $m'-m$ constraints $a_{m+1},\ldots,a_{m'}$.
Since given $\hat \SIGMA_{n,m} =  \hat \SIGMA_{n,m'}$ we have $\G'''=\G'$, (\ref{eq_mon_1}) yields
\begin{align*}
\Erw\brk{\TV{\rho_{\hat \SIGMA_{n,m'}, \tilde\tau_{\G''}} - \bar{\rho}} \Big \vert \hat \SIGMA_{n,m} =  \hat \SIGMA_{n,m'}}
	=\Erw\brk{\TV{\rho_{\hat \SIGMA_{n,m}, \tau_{\G'}} - \bar{\rho}} \Big \vert \hat \SIGMA_{n,m} =  \hat \SIGMA_{n,m'}} \geq \eps^2/7.
\end{align*}
Therefore, the Nishimori identity (\ref{eq:NishimoriS}) and (\ref{highproba}) imply
\begin{align}\label{eq_mon_2}
\Erw\brk{\bck{\TV{\rho_{\SIGMA,\tilde\tau_{\G''}} - \bar{\rho}}}_{\G''}} =\Erw\brk{\TV{\rho_{\hat \SIGMA_{n,m'}}-\tilde\tau_{\G''}}} \geq \eps^2/8.
\end{align}
Since $\tau_{\G''}$ is nearly balanced, the assertion follows from (\ref{eq_mon_2}) and Lemma  \ref{Lemma_multiOverlap}.
\end{proof}

\section{Proof of \Lem~\ref{lemma:PoissonCycles4Planted}}\label{Sec_PoissonCycles4Planted}

\noindent
The Nishimori identity (\ref{eq:nishimori}) shows that $\hat\G\stacksign{$\mathrm d$}=\G^*(n,m,\hat\SIGMA_{n,m})$.
Moreover, \Cor~\ref{lem:conc_coloring} shows that $\hat\SIGMA_{n,m}$ is nearly balanced with probability at least $1-O(n^{-1})$.
Hence, it suffices to prove that for any nearly balanced $\sigma$,
	\begin{align}\label{eqlemma:PoissonCycles4Planted_1}
	\Erw[C_Y(\G^*(n,m,\sigma))]&=\hat\kappa_Y+o(1)\qquad\mbox{and}\qquad
		\Pr\brk{\forall l\leq L: C_{Y_l}(\G^*(n,m,\sigma)=y_l}=o(1)+ \prod_{l=1}^L\Pr[\Po(\hat\kappa_{Y_l})=y_l].
	\end{align}

We begin by calculating $\Erw[C_Y(\G^*(n,m,\sigma))]$.
Suppose that $\vec i=(i_1,\ldots,i_\ell)\in[n]$ is a family of distinct indices such that $i_1<\min\{i_2,\ldots,i_\ell\}$
and let $\vec h=(h_1,\ldots,h_\ell)\in[m]$ be pairwise distinct such that $h_1<h_\ell$ if $\ell>1$.
Set $i_{\ell+1}=i_1$.
Let $\cC_Y(\vec i,\vec h)$ be the event that $x_{i_1},a_{h_1},\ldots,x_{i_\ell},a_{h_\ell}$ constitute a cycle with signature $Y=(\psi_1,s_1,t_1,\ldots,\psi_\ell,s_\ell,t_\ell)$.
Then for any nearly balanced $\sigma\in\Omega^{V_n}$ we have
	\begin{align}\nonumber
	\pr\brk{\G^\ast(n,m,\sigma) \in\cC_Y(\vec i,\vec h)}&=\prod_{j=1}^\ell
		\frac{\sum_{u_1,\ldots,u_k\in[n]}\vecone\{u_{s_j}=i_j,u_{t_j}=i_{j+1}\}\psi_j(\sigma(x_{u_1}),\ldots,\sigma(x_{u_k}))P(\psi_j)}
			{\sum_{u_1,\ldots,u_k\in[n]}\Erw[\PSI(\sigma(x_{u_1}),\ldots,\sigma(x_{u_k}))]}\\
		&=o(n^{-2\ell})+\prod_{j=1}^\ell\frac{P(\psi_j)}{n^{k}\xi}
		\sum_{u_1,\ldots,u_k\in[n]}\vecone\{u_{s_j}=i_j,u_{t_j}=i_{j+1}\}
			\psi_j(\sigma(x_{u_1}),\ldots,\sigma(x_{u_k}))&\mbox{[by {\bf SYM}]}\nonumber	\\ 
		&=o(n^{-2\ell})+ n^{-2\ell}q^\ell \prod_{j=1}^\ell P(\psi_j)\Phi_{\psi_j,s_j,t_j}(\sigma(x_{i_h}),\sigma(x_{i_{h+1}})).\label{eqCycCount1}
	\end{align}
Because $\sigma$ is nearly balanced, summing \eqref{eqCycCount1} over $\vec i,\vec h$ yields
\begin{align}\nonumber
	\Erw[C_Y(\G^*(n,m,\sigma))]&=\sum_{\vec i, \vec h}\pr\brk{\G^\ast(n,m,\sigma) \in\cC_Y(\vec i,\vec h)}
		=o(1)+\frac1{2\ell}\bcfr{m}{n}^\ell\Tr \prod_{j=1}^\ell P(\psi_j)\Phi_{\psi_j,s_j,t_j}=o(1)+\hat\kappa_Y,
\end{align}
which is the first part of (\ref{eqlemma:PoissonCycles4Planted_1}).

The second part of (\ref{eqlemma:PoissonCycles4Planted_1}) follows from the first part and a standard method of moments argument.
More specifically, since $m=O(n)$ the random factor graph $\G^*(n,m,\sigma)$ does not contain two overlapping cycles of bounded length \whp\
Therefore, a straightforward extension of the above calculation shows that for any $j_1,\ldots,j_L\geq2$ the joint factorial moment of the random variables
$C_{Y_1}(\G^*(n,m,\sigma)),\ldots,C_{Y_L}(\G^*(n,m,\sigma))$ comes to
	\begin{align*}
	\Erw\brk{\prod_{l=1}^L\bc{\prod_{u=0}^{j_l-1}C_{Y_l}(\G^*(n,m,\sigma))-u}}&=o(1)+
		\prod_{l=1}^L\Erw\brk{C_{Y_l}(\G^*(n,m,\sigma))}^{j_l}=o(1)+\prod_{l=1}^L\hat\kappa_{Y_l}^{j_l}.
	\end{align*}
In effect, the number of cycles with signature $Y$ is asymptotically Poisson with mean $\hat\kappa_Y$ by standard results on the joint convergence to asymptotic Poisson variables~\cite{BB}.

\section{Proof of \Lem~\ref{Prop_contig}}\label{Apx_Prop_contig}

\noindent
To prove \Lem~\ref{Prop_contig} we need the following rough but uniform estimate of the first moment.

\begin{claim}[\SYM, \BAL]\label{Cor_F}
For any $D>0$ there exists $c>0$ such that 
	$cq^n\xi^m\leq\Erw[Z(\G(n,m))]\leq q^n\xi^m$
for all $m\leq Dn/k$.
\end{claim}
\begin{proof}
Since constraints are chosen independently we have
	$\Erw[Z(\G(n,m))]=\sum_{\sigma\in\Omega^{V_n}}\phi(\rho_\sigma)^m.$
Because {\bf SYM} and {\bf BAL} yield $\phi(\rho_\sigma)\leq\xi$
	for every $\sigma$, the upper bound $\Erw[Z(\G(n,m))]\leq q^n\xi^m$ is immediate.
With respect to the lower bound, we observe that there $\Omega(q^n)$ assignments $\sigma:V_n\to\Omega$ with $\TV{\rho_\sigma-\bar\rho}\leq n^{-1/2}$.
\Lem~\ref{define_phi} shows that for any such $\sigma$,
	$
	\phi(\rho_\sigma)=\phi(\bar\rho)+k\xi\langle{\vecone,\rho_\sigma-\bar\rho}\rangle+O(\tv{\rho_\sigma-\bar\rho}^2)
		=\phi(\bar\rho)+O(1/n).
	$
Thus, $\Erw[Z(\G(n,m))]\geq\Omega(q^n)(\phi(\bar\rho)+O(1/n))^m=\Omega(q^n\xi^m)$ uniformly for all $m\leq Dn/k$.
\end{proof}

\begin{claim}[\SYM, \BAL]\label{lem:AstsigmaHatSigma}
Let $D>0$.
Uniformly for all $m\leq Dn/k$ and for all nearly balanced $\sigma\in\Omega^{V_n}$ we have
	\begin{align}\label{eqlem:AstsigmaHatSigma}
	\pr \brk{ \hat\SIGMA_{n,m} = \sigma } = \pr \brk{ \SIGMA^\ast = \sigma} 
		\exp \bc { n O \bc{ \Vert \rho_{\sigma} - \bar \rho \Vert_{\mathrm{TV}}^2 } +O(1)}.
\end{align}
Furthermore, for any $\eps>0$ there is $C>0$ such that $\pr \brk{\tv{\rho_{\hat\SIGMA_{n,m}}-\bar\rho}>Cn^{-1/2}}<\eps$.
\end{claim}
\begin{proof}
Recall $\phi$ from Lemma \ref{define_phi}.
Since constraints are chosen independently, we obtain
	$\Erw \brk{ \psi_{\G(n,m)} (\sigma)}=\phi(\rho_\sigma)^m$ for every $\sigma$.
Due to Lemma \ref{define_phi}, $\phi (\rho) = \xi + O\bc{ \Vert \rho - \bar \rho \Vert_{\mathrm{TV}}^2 }.$
Hence, Claim~\ref{Cor_F} reveals that for a nearly balanced $\sigma$,
	\begin{align*}
	\pr \brk{ \hat\SIGMA = \sigma }&=\frac{\Erw[\psi_{\G(n,m)}(\sigma)]}{\Erw[Z(\G(n,m))]}
		=\Theta(q^{-n}\xi^{-m})\phi(\rho_\sigma)^m=\Theta(q^{-n})\bc{1 + O\bc{ \Vert \rho - \bar \rho \Vert_{\mathrm{TV}}^2}}^m
		=q^{-n}\exp \bc { n O \bc{ \Vert \rho_{\sigma} - \bar \rho \Vert_{\mathrm{TV}}^2 } +O(1)},
	\end{align*}
which yields the first assertion.
The second assertion follows from the estimate $\pr \brk{ \hat\SIGMA = \sigma }=\Theta(q^{-n}\xi^{-m})\phi(\rho_\sigma)^m$ and assumption {\bf BAL}, which provides that $\bar\rho$ is the maximiser of $\phi$.
\end{proof}

\begin{proof}[Proof of \Lem~\ref{Prop_contig}]
Let $D,\eps>0$, pick $\delta>0$ small enough and $n_0>0$ big enough.
As a first step we observe that for any event $\cA$ the following two implications are true:
	\begin{align}\label{eqproofofProp_contig}
	\pr\brk{\SIGMA^*\in\cA}<\delta&\Rightarrow\pr\brk{\hat\SIGMA_{n,m}\in\cA}<\eps,&
	\pr\brk{\hat\SIGMA_{n,m}\in\cA}<\delta&\Rightarrow\pr\brk{\SIGMA^*\in\cA}<\eps.
	\end{align}
These implications are immediate from Claim~\ref{lem:AstsigmaHatSigma}.
Indeed, assume that $\pr\brk{\SIGMA^*\in\cA}<\delta$.
Then for a large $C>0$,
	\begin{align*}
	\pr\brk{\hat\SIGMA_{n,m}\in\cA}&\leq\pr\brk{\hat\SIGMA_{n,m}\in\cA\mid\tv{\rho_{\hat\SIGMA_{n,m}}-\bar\rho}\leq Cn^{-1/2}}+\eps/2
		\leq
		\exp \bc { C^3}\pr\brk{\SIGMA^*\in\cA}+\eps/2<\eps,
	\end{align*}
provided $\delta>0$ was chosen small enough.
The proof of the second implication is analogous.

To derive the assertion from \eqref{eqproofofProp_contig}, let $\cE$ be an event and assume that  $\pr\brk{(\G^*(n,m,\SIGMA^*),\SIGMA^*)\in\cE}<\delta$.
Further, for an assignment $\sigma$ let $\cE_\sigma$ be the set of all pairs $(G,\sigma)$ contained in $\cE$.
Assuming $\delta>0$ is sufficiently small, we obtain
	\begin{align*}
	\pr\brk{(\G^*(n,m,\hat\SIGMA_{n,m}),\hat\SIGMA_{n,m})\in\cE}&=
		\sum_{\sigma\in\Omega^{V_n}}
		\pr\brk{(\G^*(n,m,\sigma),\sigma)\in\cE}\pr\brk{\hat\SIGMA_{n,m}=\sigma}\\
		&\leq\eps+\sum_{\sigma\in\Omega^{V_n}}
		\pr\brk{(\G^*(n,m,\sigma),\sigma)\in\cE}\pr\brk{\SIGMA^*=\sigma}<2\eps.
	\end{align*}
The proof of the reverse direction is analogous.
\end{proof}

\end{appendix}

\end{document}